\documentclass[a4paper,10pt]{article}
\usepackage{comment}

\usepackage{amsthm,amsmath,amssymb}
\usepackage{nicefrac}
\usepackage{bbm}
\usepackage{epsfig}
\usepackage{enumerate}

\newtheorem{lemma}{Lemma}[section]
\newtheorem{remark}[lemma]{Remark}
\newtheorem{proposition}[lemma]{Proposition}
\newtheorem{theorem}[lemma]{Theorem}

\newtheorem{corollary}[lemma]{Corollary}

\usepackage{geometry}
\providecommand{\eps}{{\ensuremath{\varepsilon}}}

\newcommand{\sgn}{\operatorname{sgn}}

\providecommand{\eps}{{\ensuremath{\varepsilon}}}

\providecommand{\N}{{\ensuremath{\mathbbm{N}}}}

\providecommand{\R}{{\ensuremath{\mathbbm{R}}}}

\providecommand{\E}{{\ensuremath{\mathbb{E}}}}
\renewcommand{\P}{{\ensuremath{\mathbb{P}}}}
\providecommand{\1}{{\ensuremath{\mathbbm{1}}}}

\begin{document}
\title{
Strong convergence rates and temporal regularity for
 Cox-Ingersoll-Ross processes and Bessel processes\\
            with accessible boundaries }
\author{Martin Hutzenthaler, Arnulf Jentzen and Marco Noll}
\maketitle
\abstract{
Cox-Ingersoll-Ross (CIR) processes are widely used in financial modeling
such as in the Heston model
for the approximative pricing of financial derivatives.
Moreover,
CIR processes are mathematically interesting
due to the irregular square root function in the diffusion coefficient.
In the literature, positive strong convergence rates for 
numerical approximations of CIR processes have
been established in the case of an inaccessible boundary point.
Since calibrations of the Heston model frequently result in parameters
such that the boundary is accessible, we focus on this interesting case.
Our main result shows for every $p\in(0,\infty)$
that the drift-implicit square-root Euler approximations
proposed in Alfonsi (2005) converge in the strong $ L^p $-distance
with a positive rate for half of the parameter regime in which the boundary point is accessible.
A key step in our proof is temporal regularity of Bessel processes.
More precisely, we prove for every $p\in(0,\infty)$
that Bessel processes are temporally
$\nicefrac{1}{2}$-H\"older continuous
in $L^p$.
}
\tableofcontents

\section{Introduction}

Cox-Ingersoll-Ross (CIR) processes
have been proposed
in Cox, Ingersoll \& Ross~\cite{cir85}
as a model for short-term interest rates.
Since then, CIR processes are widely used in financial modeling.
For instance, CIR processes appear as the instantaneous squared volatility 
in the Heston model~\cite{h93}
which is one of the most popular
equity derivatives pricing models among practitioners
and which is  day after day being numerically approximated 
in the financial engineering industry.
Building on explicitly known Fourier transforms of
CIR processes at fixed time points,
it is well known how to valuate
plain vanilla options in the Heston model
or how to calibrate the Heston model with European calls and puts
(see, however, Kahl \& J\"ackel~\cite{KahlJaeckel2005} for numerical issues).
In contrast, 
complicated path-dependent financial derivatives
within the Heston model 
are prized by using 
time-discrete approximations
of the Heston model and Monte Carlo methods.
In addition, positive strong convergence rates of
the time-discrete approximations
are important
for applying efficient multilevel Monte Carlo methods
(see Giles~\cite{g08a}, Kebaier~\cite{k05}, Heinrich~\cite{h98,h01}).
%
In view of this,
the central goal of this article is to establish a positive
strong convergence rate for time-discrete approximations of CIR processes.

For a formal introduction of CIR processes, let
$ x, \beta \in (0,\infty) $, $\delta\in[0,\infty)$, $ \gamma \in \R $,
let
$ 
  ( 
    \Omega, \mathcal{F}, \P,  
    ( \mathcal{F}_t )_{ t \in [ 0 , \infty ) }
  ) 
$
be a stochastic basis
(see Section~\ref{ssec:Notation} for this and further notation),
let
$
  W \colon [0,\infty) \times \Omega \to \R
$
be a standard $ ( \mathcal{F}_t )_{ t \in [ 0, \infty ) } $-Wiener process
and let
$ X \colon [ 0 , \infty ) \times \Omega \to [0,\infty)$
be an adapted stochastic process
with continuous sample paths
satisfying the stochastic differential equation (SDE)
\begin{equation}
\label{eq:CIR_intro}
  d X_t = \left(\delta - \gamma X_t \right) dt
  +
  \beta \sqrt{ X_t } \, dW_t
  ,
  \qquad
  t \in [0,\infty) ,
  \qquad
  X_0 = x .
\end{equation}
If $\delta,\gamma\in(0,\infty)$, then the processes $X$ is
the CIR process
with long-time mean $ \nicefrac{ \delta }{ \gamma } $, 
speed of adaptation $ \gamma $ 
and volatility $ \beta $.
Moreover if $\gamma=0$ and $\beta=2$, then $X$ is the squared Bessel process
and $\sqrt{X}$ is the \emph{Bessel process of ``dimension'' $\delta$};
see, e.g., G\"oing-Jaeschke \& Yor~\cite{GoeingJaeschkeYor2003} for a review
on Bessel processes.
We note that according to Feller's boundary classification
(see, e.g., Theorem V.51.2 in~\cite{RogersWilliams2000b}),
the boundary point
zero is inaccessible
(that is, $ \P[ \, \forall \, t \in (0,\infty) \colon X_t > 0 \, ] = 1 $)
if and only if $2\delta\geq\beta^2$.
Alfonsi~\cite{Alfonsi2005} proposed in the case $4\delta>\beta^2$
the 
numerical
approximations
$
  Y^h \colon[0,\infty)\times\Omega\to[0,\infty)
$, 
$ h \in (0,\infty) \cap ( 0, \nicefrac{ 2 }{ \gamma^- } )
$, 
satisfying
for all 
$ 
  h \in ( 0, \nicefrac{ 2 }{ \gamma^- } ) 
$,
$ n \in \N_0 $,
$ t \in (nh, (n+1)h ] $
that
$
  Y_0^h = X_0
$ 
and
\begin{equation} \label{eq:LBE}
  Y_{t}^h
  =
  \tfrac{(n+1)h-t}{h}
  Y_{nh}^h+
  \tfrac{t-nh}{h}
  \left[
    \tfrac{
      ( Y_{nh}^h )^{ 1 / 2 }
      +
      \frac{
        \beta
      }{ 2 }
      ( W_{ ( n + 1 ) h } - W_{ n h } )
      +
      \sqrt{
        \left[
          ( Y_{nh}^h )^{ 1 / 2 }
          +
          \frac{ \beta }{ 2 }
          ( W_{ ( n + 1 ) h } - W_{ n h } )
        \right]^2
          +
          ( 2 + \gamma h ) 
          ( \delta - \frac{ \beta^2 }{ 4 } )
          h
      }
    }{ 
      ( 2 + \gamma h ) 
    }
  \right]^2.
\end{equation}
Following Dereich, Neuenkirch \& Szpruch~\cite{DereichNeuenkirchSzpruch2012}, we refer
to the numerical approximations~\eqref{eq:LBE} as
\emph{linearly interpolated drift-implicit square-root Euler approximations}.

In the literature, the
goal of establishing positive strong convergence rates
for CIR processes has been achieved in the case of an inaccessible boundary point.
The first positive rates were established in Theorem 2.2 in
Berkaoui, Bossy \& Diop~\cite{BerkaouiBossyDiop2008}
which proves 
for every $ p \in [1,\infty) $
with
$
  \tfrac{2\delta}{\beta^2}>
  1+\sqrt{8}\max(\tfrac{1}{\beta}\sqrt{\gamma^+(16p-1)},16p-2)
$
the uniform $ L^p $-rate $ \nicefrac{1}{2} $ 
in the case of a symmetrized Euler scheme.
In addition, Theorem~1.1
in Dereich, Neuenkirch \& Szpruch~\cite{DereichNeuenkirchSzpruch2012}
implies
that if the boundary point zero is inaccessible, then
the linearly interpolated
drift-implicit square-root Euler approximations~\eqref{eq:LBE}
converge
for every $p\in[1,\tfrac{2\delta}{\beta^2})$
with uniform $L^p$-rate $\nicefrac{1}{2}-$.
Theorem 2 in Alfonsi~\cite{Alfonsi2013}
and
Proposition 3.1 in Neuenkirch \& Szpruch~\cite{NeuenkirchSzpruch2014}
even imply
for every $p\in[1,\tfrac{4\delta}{3\beta^2})$
the uniform $L^p$-rate $1$
in the regime $\delta>\beta^2$.
These results exploit that the process $Z\colon[0,\infty)\times\Omega\to[0,\infty)$
defined through $Z_t:=\sqrt{X_t}$, $t\in[0,\infty)$, satisfies
in the case $2\delta\geq\beta^2$ the
SDE with additive noise
\begin{equation} \label{eq:Lamperti}
  dZ_t=\left[\tfrac{4\delta-\beta^2}{8Z_t}-\tfrac{\gamma}{2}Z_t\right] dt
       +\tfrac{ \beta }{ 2 } \, dW_t
  ,
  \qquad
  t \in [0,\infty) ,
\end{equation}
which follows from It\^o's lemma applied to the $C^2((0,\infty),\R)$-function
$(0,\infty)\ni x\mapsto \sqrt{x}\in\R$.
The reason for considering the SDE~\eqref{eq:Lamperti}
is that if the boundary point zero is inaccessible, then
the SDE~\eqref{eq:Lamperti} has a globally one-sided Lipschitz continuous drift coefficient
and a globally Lipschitz continuous diffusion coefficient.
To the best of our knowledge, the above goal remained an open problem
when the boundary point zero is 
accessible.
For results on strong convergence without rate
which apply to CIR processes with accessible boundary, see, e.g.,
Deelstra \& Delbaen~\cite{DeelstraDelbaen1998}, 
Alfonsi~\cite{Alfonsi2005},
Higham \& Mao~\cite{HighamMao2005},
Lord, Koekkoek \& Dijk~\cite{LordKoekkoekDijk2010},
Gy\"ongy \& R{\'a}sonyi~\cite{GyoengyRasonyi2011},
Halidias~\cite{Halidias2012}.
For results on pathwise convergence
which apply to CIR processes with accessible boundary, see, e.g.,
Milstein \& Schoenmakers~\cite{MilsteinSchoenmakers2013}.
To establish strong convergence rates in the case
$ 2 \delta < \beta^2 $
is of particular difficulty as the coefficients 
of the SDE~\eqref{eq:CIR_intro} are
not even \emph{locally} Lipschitz continuous
on the state space
in that case.

In real-world applications of the Heston model, calibrations
frequently result in parameters such that 
$2\delta<\beta^2$
so that
zero is 
an accessible boundary point of the volatility process;
see, e.g.,
Table III in~\cite{BakshiCaoChen1997},
Table 1 in~\cite{DuffiePanSingleton2000}
and Table 1 in~\cite{EngelmannKosterOeltz2011}
for calibration results in the literature.
So it is important for practitioners to know which
numerical approximations
converge with a positive rate
even in the case of an accessible boundary point.
To the best of our knowledge,
our main theorem,
Theorem~\ref{thm:main.result},
is the first result which establishes
a positive rate of strong convergence
of numerical approximations of Cox-Ingersoll-Ross processes
which hit the boundary point $0$ with positive probability.

\begin{theorem}   \label{thm:main.result}
Assume the above setting,
assume 
$ \nicefrac{ 2 \delta }{ \beta^2 } > \nicefrac{ 1 }{ 2 } $
and
let $ T, \varepsilon \in (0,\infty) $,
$ p \in \left[ 1, \infty \right) $.
Then
there exists a real number $ C \in [0,\infty) $ 
such that for all 
$
  h \in (0,T] \cap 
  (
    0, 
    \nicefrac{ 1 }{ 2 \gamma^{ - } } 
  )
$ 
it holds that
\begin{align} \label{eq:theorem_cir:ass8}
\bigg(
  \E\bigg[ 
    \sup_{
      t \in [0,T]
    }
    \big| 
      X_{ t } - Y^h_t 
    \big|^p
  \bigg]
\bigg)^{ \! 1 / p }
 \leq  
 C
 \cdot
 h^{ 
   \left[
     \frac{ (\nicefrac{2 \delta}{\beta^2}) \wedge 1  - \nicefrac{1}{2} }
{p}
     - \eps  
   \right]
 }
 .
\end{align}
\end{theorem}

Theorem~\ref{thm:main.result} follows
from Corollary~\ref{c:theorem_cir} below.
Corollary~\ref{c:theorem_cir}
in turn
is a corollary of
Theorem~\ref{thm:num_theorem} which proves a positive rate of strong convergence
for
drift-implicit square-root Euler approximations
for more general square-root diffusion processes with accessible boundaries
under the assumption of certain inverse moments.
In addition, Corollary~\ref{l:num_corollary} below
and Theorem~\ref{l:cir_hr} below yield a strong approximation result for
the Bessel process $Z$.
We emphasize that
it is not clear to us whether the strong convergence rate proved in 
Theorem~\ref{thm:main.result} is sharp.
It might very well be the case that the rate is not sharp
(see the numerical simulation in Figure 2 in Alfonsi~\cite{Alfonsi2005}
which suggests an $L^1$-rate between $\nicefrac{1}{2}$ and $1$ for
$\nicefrac{2\delta}{\beta^2}\in(\tfrac{1}{2},1]$).

We sketch the main steps in our proof of Theorem~\ref{thm:main.result}.
Corollary~\ref{l:num_corollary}
below
shows that suitable uniform temporal H\"older regularity 
of the Bessel-type processes~\eqref{eq:Lamperti}
together with uniform moment bounds
of the numerical approximation processes~\eqref{eq:LBE}
is actually enough
to deduce a strong rate of convergence
of the linearly interpolated
drift-implicit square-root Euler approximations~\eqref{eq:LBE}.
So this temporal H\"older regularity
of the Bessel-type processes~\eqref{eq:Lamperti}
is
an important step
in our proof of Theorem~\ref{thm:main.result}
and is subject of the following 
theorem, Theorem~\ref{thm:Bessel.regularity}.
To the best of our knowledge,
Theorem~\ref{thm:Bessel.regularity} (choose $\beta=2$ and $\gamma=0$) 
is the first result which proves
for every $ q\in(0,\infty) $, $ \eps \in ( 0, \frac{ 1 }{ 2 } ) $ 
that
$ 
  Z \in 
  L^q(
    \Omega ; 
    C^{ \nicefrac{ 1 }{ 2 } - \eps }( [0,T], [0,\infty) )
  )
$
holds
for the Bessel process $Z$ with ``dimension'' $\delta\in[0,2)$;
see, e.g., Lemma 3.2 in Dereich, Neuenkirch \& Szpruch~\cite{DereichNeuenkirchSzpruch2012}
for the case
$\tfrac{4\delta}{\beta^2}\in(2,\infty)$.
\begin{theorem}\label{thm:Bessel.regularity}
Assume the above setting.
  Then
for all $ T, \varepsilon, p \in (0,\infty) $
it holds that
\begin{equation} 
\begin{split} 
     \sup_{ \substack{ s, t \in [0,T] \\ s \neq t } }
 \left\|
   \tfrac{ Z_t  - Z_s  }{ |t-s|^{\nicefrac{1}{2} } }
 \right\|_{L^p( \Omega; \R )} 
 +
 \bigg\|
     \sup_{ \substack{ s, t \in [0,T] \\ s \neq t } }
   \tfrac{\left| Z_t  - Z_s \right| }{ |t-s|^{\nicefrac{1}{2} - \eps} }
 \bigg\|_{L^p( \Omega; \R )} 
< \infty.
\end{split} 
\end{equation}
\end{theorem}
Theorem~\ref{thm:Bessel.regularity} follows
from Corollary~\ref{cor_hr} below.
Corollary~\ref{cor_hr}
in turn
is a corollary of
Theorem~\ref{l:cir_hr} which proves
temporal $\nicefrac{1}{2}$-H\"older continuity
for more general diffusion processes with additive noise
(see also Lemma \ref{l:transformierte_gleichung} below).

\subsection{Notation}
\label{ssec:Notation}

Throughout this article we use the following notation.
For $ d \in \N := \{ 1, 2, \dots \} $ and
$ v = ( v_1, \dots, v_d ) \in \R^d $,
we denote by
$
  \| v \|_{\R^d}
  := 
  \left[
    | v_1 |^2
    +
    \ldots
    +
    | v_d |^2
  \right]^{ 1 / 2 }
$
the Euclidean norm of $ v $.
For two sets
$ A $
and
$ B $
we denote by
$ \mathcal{M}( A, B ) $
the set of all mappings from
$ A $ to $ B $.
For two measurable spaces
$ ( A, \mathcal{A} ) $
and
$ ( B, \mathcal{B} ) $
we denote by
$ \mathcal{L}^0( A ; B ) $
the set of all $ \mathcal{A} $/$ \mathcal{B} $-measurable mappings 
from $ A $ to $ B $.
If $ \left( \Omega, \mathcal{F}, \P \right) $ is a probability space,
if $ I \subseteq\R $ is a closed and non-empty interval and if
$ ( \mathcal{F}_t )_{ t \in I } $ is a normal filtration on
$ ( \Omega, \mathcal{F}, \P ) $,
then we call the quadrupel 
$
  ( \Omega, \mathcal{F}, \P, ( \mathcal{F}_t )_{ t \in I } )
$
a \emph{stochastic basis}
(cf., e.g., Appendix~E in Pr{\'e}v{\^o}t \&\ R\"{o}ckner~\cite{PrevotRoeckner2007}).
For
an interval $ O \subseteq \R $
with $ \#( O ) = \infty $
and two functions
$
  \mu \colon O \to \R
$
and
$
  \sigma \colon O \to \R
$,
we define the linear operator
$  
  \mathcal{G}_{ \mu, \sigma }
  \colon
  C^2( O, \R )
  \to
  \mathcal{M}( O, \R )
$
by
\begin{equation}
  ( \mathcal{G}_{ \mu, \sigma } \phi)( x )
:=
  \phi'(x) \,
  \mu( x )
  +
  \tfrac{ 1 }{ 2 } \,
  \phi''(x) 
  \left(
    \sigma(x)
  \right)^2
\end{equation}
for all 
$ x \in O $,
$ \phi \in C^2( O, \R ) $.
Throughout this article we also often calculate and formulate
expressions in the extended real numbers 
$ [ - \infty, \infty ] = \R \cup \{ -\infty, \infty \} $.
In particular, 
we frequently use the conventions
$
  \sup( \emptyset ) = - \infty
$,
$
  \frac{ 0 }{ 0 } 
  =
  0 \cdot \infty
  = 0
$,
$
  0^0 = 1
$,
$
  \frac{ a }{ 0 } = \infty
$,
$
  \frac{ - a }{ 0 } = - \infty
$,
$
  \frac{ b }{ \infty } = 0
$,
$ 
  0^a
  = 0
$,
$
  0^{ - a } = 
  \frac{ 1 }{ 0^{ a } } 
  = 
  \infty
$
for all $ a \in ( 0,\infty) $, $ b \in \R $.
Furthermore, we define 
$
  x \wedge y := \min(x,y)
$ 
and 
$
  x \vee y
  := \max(x,y)
$ 
for all
$
  x, y \in \R
$.
%
%
%
Finally, if 
$ T \in (0,\infty) $, 
$ s \in [0,T] $,
$ n \in \N $,
$\theta=(t_0,\dots, t_n) \in [0,T]^{n+1}$ satisfy
$0=t_0< t_1 < \dots < t_n=T$,
then we define 
$
  | \theta | := \max_{k \in \{1, \dots,n\}} |t_k -t_{k-1}|
$,
$
  \lfloor s \rfloor_{ \theta } :=
  \sup\big( 
    \{
      t_0, t_1, \dots, t_n 
    \}
    \cap [0,s]
  \big)
$
and
$
  \lceil s \rceil_{\theta} :=
  \inf\big( 
    \{
      t_0, t_1, \dots, t_n 
    \}
    \cap
    [ s, T ]
  \big)
$.
For
two normed vector spaces 
$
  (V_1, \left\| \cdot \right\|_{V_1})
$ 
and 
$
  (V_2, \left\| \cdot \right\|_{V_2})
$, a real number $\theta \in [0,\infty)$ and a function 
$ f \colon V_1 \to V_2 $ 
from $ V_1 $ to $ V_2 $, 
we define 
$ 
  \| f \|_{ \mathcal{C}^{ \theta }( V_1, V_2 ) } 
  := 
  \sup_{ v, w \in V_1 }
  \frac{ 
    \| f(v) - f(w) \|_{ V_2 } 
  }{
    \| v - w \|_{ V_1 }^{ \theta } 
  }
$.
For a probability space 
$ \left( \Omega, \mathcal{F}, \P \right) $, 
a  normed vector space 
$ 
  (V, \left\| \cdot \right\|_V ) 
$,
an $ \mathcal{F} $/$ \mathcal{B}( V ) $-measurable mapping 
$
  X \colon \Omega \to V 
$
and a real number 
$ p \in (0,\infty) $, 
we define 
$
  \left\| 
      X
   \right\|_{ L^0(\Omega; V) } := 1
$
and
$
   \left\| 
     X
   \right\|_{ L^p(\Omega; V) } 
   := 
   \left( 
     \E\!\left[  
       \|X\|_V^p 
     \right] 
   \right)^{
     \nicefrac{ 1 }{ p } 
   } 
   \in [0,\infty]
$.
%
%
%



%
%
%
%

\section{Temporal H\"older regularity}
\label{sec:time}

The main results of this section, 
Proposition~\ref{p:lemma_conv_order_1} 
and
Theorem~\ref{l:cir_hr}
below,
establish temporal H\"older regularity for Bessel-type processes.
A central step in the proof of Proposition~\ref{p:lemma_conv_order_1}
is Lemma~\ref{l:lemma_inv_bound_new} below
which proves exponential inverse moments for Bessel-type processes.

\subsection{Setting}
\label{sec:setting_numeric}
Throughout Section~\ref{sec:time} 
we will frequently use the following setting.
Let $ T \in [0,\infty) $, 
$ \mu \in \mathcal{L}^0 ([0,\infty); \R) $,
let
$
  ( 
    \Omega, \mathcal{F}, \P, ( \mathcal{F}_t )_{ t \in [0,T] } 
  )
$
be a stochastic basis, let
$
  W \colon [0,T] \times \Omega \to \R
$
be a standard $ ( \mathcal{F}_t )_{ t \in [0,T] } $-Brownian motion,
  let $X\colon[0,T]\times \Omega\to[0,\infty)$
  be an adapted stochastic process with continuous sample paths satisfying 
  $\int_0^T \left| \mu(X_s) \right| ds < \infty$ $\P$-a.s.\
  and
  $
    X_t=X_0+\int_0^t  \mu(X_s) \,ds+ W_t
  $ $\P$-a.s.\
  for all $t\in[0,T]$.

\begin{remark}
Observe that if $ \sigma \in (0,\infty) $ and 
if $ Y \colon [0,T] \times \Omega \to \R $ 
is an adapted stochastic process with continuous
sample paths satisfying 
$ 
  Y_t = Y_0 + \int_0^t \mu( Y_s ) \, ds
  + \sigma W_t
$
$ \P $-a.s.\ for all $ t \in [0,T] $,
then the process $ \tilde{X} \colon [0,T] \times \Omega
\to \R $ given by 
$ \tilde{X}_t = \frac{ 1 }{ \sigma } Y_t $
for all $ t \in [0,T] $ is an adapted stochastic
process with continuous sample paths which
satisfies
$
  \tilde{X}_t 
  = \tilde{X}_0 + 
  \int_0^t ( \frac{ 1 }{ \sigma } \mu )( \tilde{X}_s ) \, ds
  +
  W_t
$
$ \P $-a.s.\ for all $ t \in [0,T] $.
This illustrates that w.l.o.g.\ one may assume
that $X$ is a solution of an SDE with a diffusion coefficient
which is identically equal to the constant $1$.
\end{remark}

\subsection{Exponential inverse moments for Bessel-type processes}
\begin{lemma}  
\label{l:lemma_mom_bound}
Assume the setting in Section~\ref{sec:setting_numeric},
let $ c, \beta \in [0,\infty) $
and assume $ x \mu(x) \leq c \left( 1 + x^2 \right) $ 
for all $ x \in (0,\infty) $.
Then it holds for all $ t\in [0,T]$ that
\begin{align}\label{l:lemma_mom_bound_new:5} 
 & \E\!\left[ 
 \exp\!\left( 
 \frac{1+(X_t)^2}{e^{2t(c+1)+t\beta}}
 +
 \int_0^t  \frac{\beta (1+(X_s)^2)}{e^{2s(c+1)+s\beta}}\, ds
\right)
\right]
\leq 
\E \Big[ \exp\!\big( 1+ (X_0)^{2} 
\big)
\Big].
\end{align}
\end{lemma}
\begin{proof} [Proof of Lemma~\ref{l:lemma_mom_bound}]
First of all, we define functions 
$
  \sigma \colon [0,\infty) \to \R
$,
$
  U, \overline{U} \colon \R \to \R
$
by 
$
  \sigma(x) = 1
$
for all $ x \in [0,\infty) $
and by
$
  U(x) := 1 + x^2
$ 
and 
$
  \overline{U}(x)
  :=
  \beta \left( 1 + x^2 \right)
$
for all $ x \in \R $.
Then observe that for all $x \in [0,\infty)$ it holds that
$|\nabla U(x) |^2 =  4 x^{2}$ and 
\begin{align}\label{l:lemma_mom_bound_new:3}
  ( \mathcal{G}_{ \mu, \sigma }U) (x)
=
2  x \mu(x)
+ 1
\leq 
2 \, c \left( 1 + x^2 \right) 
+ 1.
\end{align}
This implies that
for all $x \in [0,\infty)$ it holds that
\begin{equation}
\label{l:lemma_mom_bound_new:4} 
\begin{split}
 ( \mathcal{G}_{ \mu, \sigma }U) (x) + \tfrac{1}{2} |\nabla U (x) |^2 + \overline{U}(x)
 & \leq
 2c (1+x^2) 
+ 1
+ 2 x^2
+ \beta (1 + x^2)
\\ & \leq \big( 2 ( c + 1 ) + \beta \big) 
  \left( 1 + x^2 \right) .
\end{split}
\end{equation}
Corollary 2.4 in Cox et.~al~\cite{CoxHutzenthalerJentzen2013} together with \eqref{l:lemma_mom_bound_new:4} shows that for all $ t\in [0,T]$ it holds that
\begin{align}
 \E\!\left[ 
 \exp\!\left( 
 \frac{ 1 + (X_t)^2 }{
   e^{ 2 t ( c + 1 ) + t \beta }
 }
 +
 \int_0^t  
 \frac{\beta (1+(X_s)^2)}{e^{2s(c+1)+s\beta}}\, ds
\right)
\right]
\leq 
\E\Big[ 
  \exp\!\big( 1 + ( X_0 )^2 
\big)
\Big].
\end{align}
This finishes the proof of Lemma~\ref{l:lemma_mom_bound}.
\end{proof}

Theorem 3.1 in Hurd \& Kuznetsov~\cite{HurdKuznetsov2008}, in particular,
implies in the setting of the introduction
for the Bessel-type process $Z$ that
if $2\delta>\beta^2$ and if 
$
  \E\big[
    ( Z_0 )^{ 
      (1 - \frac{2 \delta }{ \beta^2 })
    }
  \big]
  < \infty
$, then
$
\E\big[
  \tfrac{\beta^2}{8}\big(\tfrac{2\delta}{\beta^2}-1\big)^2
    \int_0^T (Z_s)^{-2}\,ds
    \big]
<
\infty
$.
Inequality~\eqref{l:lemma_inv_bound_new:state1B}
in Lemma~\ref{l:lemma_inv_bound_new} 
below
together with Lemma~\ref{l:lemma_mom_bound}
above complements this result with exponential 
inverse moments
in the case 
$ 
  \frac{ \beta^2 }{ 2 } < 2 \delta 
  \leq \beta^2 
$.
In particular,
we reveal in 
Lemma~\ref{l:lemma_inv_bound_new}
a suitable exponentially growing
\emph{Lyapunov-type function}
for Bessel- and Cox-Ingersoll-Ross-type
processes respectively
(see \eqref{l:lemma_inv_bound_new:state1}
and \eqref{l:lemma_inv_bound_new:4}
below).

\begin{lemma}
\label{l:lemma_inv_bound_new}
Assume the setting in Section~\ref{sec:setting_numeric},
let
$
  \alpha \in (0, \infty)
$, 
$ 
  p \in ( 1, ( 1 + 2 \alpha ) \wedge 2 )
$,
$ 
  \tilde{c} \in [0,\infty)
$,
$ 
  c \in [p, \infty) 
$ 
and assume 
$
  \alpha - \tilde{c} \, x^{c} \leq x \mu(x) 
$ 
for all $x \in (0,\infty)$.
Then it holds for all 
$ t \in [0,T] $, 
$ q, \rho \in (0,\infty) $, 
$ \tilde{\rho} \in (0,\rho) $,
$ r \in (1,\infty) $
that
\begin{equation} 
\begin{split}
\label{l:lemma_inv_bound_new:state1}
 &
 \E\!\left[ 
   \exp\!\Big( 
     \tilde{\rho} \left( 2 - p \right) 
     \left(
       \alpha + \tfrac{ 1 - p }{ 2 } 
     \right) 
     \smallint_0^t  
     ( X_s )^{ - p } 
     \, ds
     - \tilde{c} \, \rho \left( 2 - p \right)  
     \smallint_0^t  
     (X_s)^{ (c - p) } 
     \, ds 
     - \rho \, ( X_t )^{ (2 - p) } 
  \Big)
\right]
\\ &
\leq
\exp\!\bigg( 
  \tfrac{ t \rho^2 }{ 2 }
  \left[
    \tfrac{ 
      \rho^2 
    }{
      ( \rho - \tilde{\rho} ) 
      \,
      \left( 2 \alpha + 1 - p \right)
    } 
  \right]^{ 
    \frac{ ( 2 p - 2 ) }{ ( 2 - p ) }
  }
\bigg)
  \,
  \E\!\left[ 
    e^{
      - \rho \, ( X_0 )^{ (2 - p) } 
    }
  \right] < \infty
  ,
\end{split}
\end{equation}%
\begin{equation}
\label{l:lemma_inv_bound_new:state1B} 
\begin{split}
 &
  \E\!\left[ 
    \exp\!\left(
      \smallint_0^t   
      \tilde{\rho} \left( 2 - p \right)
      \left(
        \alpha + \tfrac{ 1 - p }{ 2 } 
      \right)  
      \left( X_s \right)^{ - p } 
      ds 
    \right)
  \right]
\leq
  \exp\!\bigg(  
    \tfrac{ t r \rho^2 }{ 2 }
    \left[ 
      \tfrac{  
        r \rho^2 
      }{
        \left( \rho - \tilde{\rho} \right) \,
        \left( 
          2 \alpha + 1 - p 
        \right)
      } 
    \right]^{
      \frac{ (2 p - 2) }{ (2 - p) }
    }
  \bigg)
\\ & \cdot
    \left\|
      e^{ 
        - \rho \, \left( X_0 \right)^{ (2 - p) } 
      } 
    \right\|_{
      L^r( \Omega; \R )
    } 
    \left\|
      \exp\!\left( 
          \rho 
          \left( X_t \right)^{ (2 - p) } 
        + 
        \smallint_0^t 
          \tilde{c} \, \rho 
          \left( 2 - p \right) 
          \left( X_s \right)^{ (c - p) } 
        ds 
      \right)
    \right\|_{
      L^{ r / ( r - 1 ) }( \Omega; \R )
    }
  \quad
  \text{and}
\end{split}
\end{equation}
\begin{equation}
\label{l:lemma_inv_bound_new:state3} 
\begin{split}
  \Big\|  
    \smallint_0^t  (X_s)^{ - p } \, ds
  \Big\|_{L^{q}(\Omega;\R)}
&
\leq 
  \tfrac{ 
    2^{ 
      1 \vee ( 1 / q )
    }
  }{ 
    \tilde{\rho} \,
    \left( 2 - p \right) \,
    \left(
      2 \alpha + 1 - p 
    \right)
  }
\Bigg[
    \rho
\left\|   
  \left( X_t \right)^{ (2 - p) } 
  +  
  \smallint\nolimits_0^t   
    \tilde{c} 
    \left( 2 - p \right) 
    \left( X_s \right)^{ (c - p) } 
  ds 
\right\|_{L^{q}(\Omega;\R)} 
\\ & \quad
  +
  \exp\!\bigg( 
    \tfrac{ t q \rho^2 }{ 2 }
    \left[ 
      \tfrac{ q \rho^2 }{
        ( \rho - \tilde{\rho} ) 
        \left( 2 \alpha + 1 - p \right)
      } 
    \right]^{
      \frac{ (2 p - 2) }{ (2 - p) } 
    }
    - 1
  \bigg)
  \left\|
    e^{
      - 
      \rho \, 
      ( X_0 )^{ (2 - p) } 
    }
  \right\|_{
    L^q( \Omega; \R )
  }
  \Bigg] 
.
\end{split}
\end{equation}
\end{lemma}
\begin{proof} [Proof of Lemma~\ref{l:lemma_inv_bound_new}]
Throughout this proof we fix
$
  \rho \in (0,\infty)
$, 
$
  \tilde{\rho} \in (0,\rho)
$,
we define a function 
$ \sigma \colon [0,\infty) \to \R $
by
$
  \sigma(x) = 1
$
for all $ x \in [0,\infty) $
and we define
\begin{equation}
\label{l:lemma_inv_bound_new:1}
 \tilde{q} := \tfrac{ p }{ 2 \, \left( p - 1 \right) } \in (1,\infty) 
 ,
\quad 
  \tilde{p} := \tfrac{ 1 }{ ( 1 - 1 / \tilde{q} ) } \in (1,\infty)
  ,
\quad
  \delta
  :=
  \left[
    \tfrac{ 
      2 \, \tilde{q} \, \left( \rho - \tilde{ \rho } \right) 
    }{ 
      \left( 2 - p \right) \, \rho^2 
    }
    \left( 
      \alpha + \tfrac{ 1 - p }{ 2 } 
    \right) 
  \right]^{ 1 / \tilde{q} }
  \in (0,\infty)
  .
\end{equation}
Then we observe that
Young's inequality shows 
that for all 
$ \eps \in (0,1) $, $ x \in [0,\infty) $ 
it holds that
\begin{align} 
\label{l:lemma_inv_bound_new:2}
 (\eps + x)^{ (2 - 2 p) } 
\leq 
  \frac{ 1 }{ \tilde{p} } 
  \left[
    \frac{ 1 }{ \delta } 
  \right]^{ \tilde{p} } 
  + 
  \frac{ 
    \delta^{ \tilde{q} } 
    \left( \eps + x \right)^{ ( 2 - 2 p )
      \tilde{q}
    }
  }{
    \tilde{q}
  }
 =
 \frac{ 1 }{ \tilde{p} } 
 \left[
   \frac{ 1 }{ \delta } 
 \right]^{ 
   \tilde{p}
 } 
 + 
 \frac{ 
   \delta^{ \tilde{q} } 
 }{
   \tilde{q}
   \left( \eps + x \right)^{ p }
 }
 .
\end{align}
In the next step 
we define functions 
$
  U_{ \eps }, 
  \overline{U}_{ \eps } \colon [0,\infty) \to \R
$, 
$
  \eps \in (0,1)
$, 
by
$
  U_{ \eps }(x) := - \rho \left( \eps + x \right)^{ (2 - p) }
$ 
and 
\begin{align} 
  \overline{U}_{\eps}(x)
:=
  \tilde{\rho} \left( 2 - p \right)  
  \left(
    \alpha + 
    \tfrac{ 1 - p }{ 2 } 
  \right) 
  \left( \eps + x \right)^{ - p }
  - 
  \tilde{c} \, \rho \left( 2 - p \right)  
  \left( \eps + x \right)^{ (c - p) }
  - 
  \tfrac{
    \rho^2 \left( 2 - p \right)^2 
  }{
    2 \tilde{p}
  }
  \left[
    \tfrac{ 1 }{ \delta }
  \right]^{
    \tilde{p}
  } 
\end{align}
for all $x \in [0,\infty)$ and all $\eps \in (0,1)$.
Observe that for all $ \eps \in (0,1) $ 
it holds that 
$ U_{\eps} \in C^2([0,\infty), \R) $ 
and that for all 
$
  x \in [0,\infty)
$, 
$ \eps \in (0,1)
$ 
it holds that
$
  | ( \nabla U_{\eps} )(x) |^2  
  = 
  \rho^2 
  \left( 2 - p \right)^2 
  \left( \eps + x \right)^{ (2 - 2 p) }
$ 
and 
\begin{align}
\label{l:lemma_inv_bound_new:3}
  ( \mathcal{G}_{ \mu, \sigma } U_{\eps} ) (x)
=
  - 
  \rho \left( 2 - p \right) 
  \left( \eps + x \right)^{ (1 - p) } 
  \mu(x)
  -
  \rho \left( 2 - p \right)  
  \tfrac{ ( 1 - p ) }{ 2 } 
  \left( \eps + x \right)^{ - p }
  .
\end{align}
Moreover, note that 
definition~\eqref{l:lemma_inv_bound_new:2}
ensures that
$
  \tfrac{ \rho^2 \left( 2 - p \right) }{ 2 \tilde{q} } 
  \delta^{ \tilde{q} } 
  + 
  ( \tilde{\rho} - \rho )
  \left( 
    \alpha  + \tfrac{ 1 - p }{ 2 } 
  \right) = 0
$.
This, the fact that for all 
$ x \in (0,\infty) $ it holds that
$ 
  - \mu(x) 
\leq 
  \tilde{c} \, x^{ (c - 1) } - \tfrac{ \alpha }{ x }  
\leq 
  \tilde{c} 
  \left( \eps + x \right)^{ (c - 1) } 
  -
  \tfrac{ \alpha }{ (\eps + x) } 
$ 
and \eqref{l:lemma_inv_bound_new:2} imply that 
for all $x \in [0,\infty)$, $\eps \in (0,1)$ it holds that
 {\allowdisplaybreaks  
\begin{align} \nonumber
\label{l:lemma_inv_bound_new:4}  
&
  ( \mathcal{G}_{ \mu, \sigma } U_{ \eps } )( x ) 
  + 
  \tfrac{ 1 }{ 2 } 
  \left| ( \nabla U_{ \eps } )( x ) \right|^2 
  + 
  \overline{U}_{ \eps }( x )
\\ &  \nonumber
 = 
  - \rho \left( 2 - p \right) \left( \eps + x \right)^{ (1 - p) } \mu(x)
  - 
  \rho \left( 2 - p \right)  
  \tfrac{ (1 - p) }{ 2 }
  \left( \eps + x \right)^{ - p }
+
 \tfrac{ \rho^2 \, ( 2 - p )^2 }{ 2 }
 \left( \eps + x \right)^{ (2 - 2 p ) }
\\ &   \nonumber
+
  \tilde{\rho} \left( 2 - p \right)  
  \left(
    \alpha + \tfrac{ 1 - p }{ 2 } 
  \right) 
  \left( \eps + x \right)^{ - p } 
- \tilde{c} \, \rho \left( 2 - p \right)   
  \left( \eps + x \right)^{ ( c - p ) } 
- 
  \tfrac{ \rho^2 \, \left( 2 - p \right)^2 
  }{
    2 \tilde{p}
  }
  \left[ 
    \tfrac{ 1 }{ \delta } 
  \right]^{ \tilde{p} } 
\\   \nonumber
&
 \leq 
-  
  \alpha \, \rho \left( 2 - p \right)  
  \left( \eps + x \right)^{ - p } 
  + 
  \tilde{c} \, \rho \left( 2 - p \right)    
  \left( \eps + x \right)^{ (c - p) } 
  - 
  \rho \left( 2 - p \right)  
  \tfrac{ ( 1 - p ) }{2}(\eps + x )^{-p}
\\   \nonumber
&  
+
 \tfrac{ \rho^2 \, ( 2 - p )^2 }{ 2 } 
 \left( \eps + x \right)^{ (2 - 2 p) }
+
 \tilde{\rho} \left( 2 - p \right)  
 \left( 
   \alpha + \tfrac{ 1 - p }{ 2 } 
 \right) 
 \left( \eps + x \right)^{ - p }
-  
  \tilde{c} \, \rho \left( 2 - p \right)     
  \left( \eps + x \right)^{ (c - p) } 
- 
  \tfrac{ \rho^2 \, ( 2 - p )^2 }{ 2 \tilde{p} }
  \left[ \tfrac{ 1 }{ \delta } \right]^{ \tilde{p} } 
%
%
%
%
\\  \nonumber
&
\leq
  \left( 2 - p \right)  
  \left( \tilde{\rho} - \rho \right)
  \left( \alpha  + \tfrac{ 1 - p }{ 2 } \right) 
  \left( \eps + x \right)^{ - p } 
+
  \tfrac{ \rho^2 \, (2-p)^2 }{ 2 } 
  \left[
    \tfrac{ 1 }{ \tilde{p} } 
    \left[
      \tfrac{ 1 }{ \delta }
    \right]^{ \tilde{p} } 
    + 
    \tfrac{ 
      1
    }{
      \tilde{q}
    }
    \,
    \delta^{ \tilde{q} } 
    \left( \eps + x \right)^{ - p } 
  \right]
  - 
  \tfrac{ 
    \rho^2 \, ( 2 - p )^2 
  }{
    2 \tilde{p}
  }
  \left[ 
    \tfrac{ 1 }{ \delta }
  \right]^{ \tilde{p} } 
\\ 
&
=
\left(  
  \left( 2 - p \right) 
  \left( \tilde{ \rho } - \rho \right)
  \left( 
    \alpha  + \tfrac{ 1 - p }{ 2 } 
  \right) 
  + 
  \tfrac{ \rho^2 \, ( 2 - p )^2 
  }{ 2 } 
  \tfrac{ \delta^{ \tilde{q} } 
  }{
    \tilde{q}
  } 
\right)
  \left( \eps + x \right)^{ - p } 
= 0 .
\end{align}
 }%
Corollary 2.4 in 
Cox et.~al~\cite{CoxHutzenthalerJentzen2013} 
and \eqref{l:lemma_inv_bound_new:4} 
show that for all $ \eps \in (0,1) $, $ t \in [0,T] $ 
it holds that
\begin{equation} 
\begin{split} 
\label{l:lemma_inv_bound_new:5}
 &
   \E\!\left[ 
   \exp\!\left( 
     - \rho \left( \eps + X_t \right)^{ (2 - p) } 
     + 
     \smallint_0^t 
     \tfrac{
       \tilde{\rho} \, \left( 2 - p \right) 
       \,
       \left(
         \alpha + \frac{ 1 - p }{ 2 } 
       \right) 
     }{
       \left( \eps + X_s \right)^{ p }
     }
       - 
     \tfrac{
       \tilde{c} \, \rho \, \left( 2 - p \right)  
     }{
       \left( \eps + X_s \right)^{ (p - c) } 
     }
    - 
    \tfrac{ 
      \rho^2 \, ( 2 - p )^2 
    }{ 
      2 \tilde{p}
    } 
    \left[ \tfrac{ 1 }{ \delta } \right]^{ \tilde{p} } 
  ds 
  \right)
\right]
\\ & \leq 
  \E\Big[ 
    \exp\!\big( 
      - \rho \left( \eps + X_0 \right)^{ (2 - p) } 
    \big) 
  \Big]
  .
\end{split} 
\end{equation} 
%
%
%
%
%
%
%
%
%
%
%
%
%
%
%
%
%
%
%
%
%
In the next step we observe that the
identities 
$
  \tfrac{ 1 }{ 2 \tilde{q} } 
= 
  \tfrac{ ( p - 1 ) }{ p } 
$, 
$
  \tilde{q} - 1 
  = 
  \tfrac{ ( 2 - p ) }{ ( 2 p - 2 ) }
$, 
$
  \tilde{p} = \tfrac{ p }{ (2 - p) }
$ 
and 
$
  \tfrac{ \tilde{p} }{ \tilde{q} } 
  = 
  \tfrac{ (2 p - 2) }{ (2 - p) }
$ 
and the estimate
$
  \frac{
    \left( 2 - p \right) \, \left( p - 1 \right)
  }{
    p
  }
  \leq 
  \frac{ 1 }{ 2 }
$
show that
\begin{equation}
\label{l:lemma_inv_bound_new:5b}
\begin{split}
&
 \tfrac{
   t \rho^2 \, ( 2 - p )^2 
 }{
   2 \, \tilde{p}
 } 
 \left[ 
   \tfrac{ 1 }{ \delta }
 \right]^{
   \tilde{p}
 }
=
 \tfrac{
   t \rho^2 \, ( 2 - p )^2 
 }{
   2 \, \tilde{p}
 } 
  \left[ 
    \tfrac{
      \rho^2 \, ( 2 - p )
    }{
      2 \, \tilde{q} \, ( \rho - \tilde{\rho} ) 
      \left( \alpha + \frac{ 1 - p }{ 2 } \right)
    } 
  \right]^{ 
    \frac{ (2 p - 2) }{ (2 - p) } 
  }
\\ & =
  \tfrac{ 
    t \rho^2 \, ( 2 - p )^3
  }{ 
    2 \, p 
  }
  \left[ 
    \tfrac{
      \rho^2 \, ( 2 - p ) \, ( p - 1 ) 
    }{
      \left( \rho - \tilde{\rho} \right) 
      p 
      \left( \alpha + \frac{ 1 - p }{ 2 } \right)
    } 
  \right]^{ 
    \frac{ (2 p - 2) }{ (2 - p) } 
  }
\leq
  \tfrac{ t \rho^2 }{ 2 }
  \left[
    \tfrac{ 
      \rho^2 
    }{ 
      \left( \rho - \tilde{\rho} \right) 
      \left( 2 \alpha + 1 - p \right)
    } 
  \right]^{ 
    \frac{ (2 p - 2) }{ (2 - p) } 
  }
  .
\end{split}
\end{equation}
In the next step we note that 
Fatou's lemma, the dominated convergence theorem, the observation 
that for every $ \omega \in \Omega $ 
it holds that the function
$
  [0,T] \ni s \to 
  \left( X_s( \omega ) \right)^{ (c - p) } \in [0,\infty)
$ 
is bounded from above and
\eqref{l:lemma_inv_bound_new:5} and \eqref{l:lemma_inv_bound_new:5b} 
imply that for all 
$ t \in [0,T] $ 
it holds that
\begin{equation}
\label{l:lemma_inv_bound_new:6} 
\begin{split}
 &
 \E\!\left[ 
   \exp\!\left( 
     - \rho \left( X_t \right)^{ (2 - p) } 
     + 
     \smallint_0^t 
     \tilde{ \rho } 
     \left( 2 - p \right) 
     \left( 
       \alpha + \tfrac{ 1 - p }{ 2 } 
     \right)  
     \left( X_s \right)^{ - p } 
     ds
     - \smallint_0^t  
     \tilde{c} \, \rho \left( 2 - p \right)   
     \left( X_s \right)^{ (c - p) } 
     ds 
  \right)
\right]
\\ & \leq
  \E\!\left[ 
  \exp\!\left( 
    - \rho  
    \left( X_t \right)^{ (2 - p) } 
    + 
    \liminf_{ (0,1) \ni \eps \to 0 } 
    \left[
    \smallint_0^t 
      \tfrac{ 
        \tilde{\rho} 
        \, 
        \left( 2 - p \right) 
        \, 
        \left( \alpha + \frac{ 1 - p }{ 2 } \right) 
      }{
        \left( \eps + X_s \right)^p
      } 
    \, ds
      -  
      \smallint_0^t 
      \tfrac{ 
        \tilde{c} \, \rho \, \left( 2 - p \right)   
      }{
        \left( \eps + X_s \right)^{ ( p - c ) }
      }
      \,
      ds 
    \right]
  \right)
\right]
\\ & \leq 
  \liminf_{
    (0,1) \ni \eps \to 0
  }
  \E\!\left[  
    \exp\!\left( 
      - \rho  
      \left( \eps + X_t \right)^{ (2 - p) }
      +  
      \left( 2 - p \right) 
      \smallint_0^t
      \tfrac{
        \tilde{\rho}
        \,
        \left( 
          \alpha + \frac{ 1 - p }{ 2 } 
        \right) 
      }{
        \left( \eps + X_s \right)^p
      }
      -
        \tilde{c} \, \rho 
        \left( \eps + X_s \right)^{ (c - p) } 
      ds 
    \right)
  \right]
\\ & \leq
  \exp\!\left( 
    \tfrac{ t \rho^2 \left( 2 - p \right)^2 }{
      2 \tilde{p}
    } 
    \left[
      \tfrac{ 1 }{ \delta }
    \right]^{ \tilde{p} } 
  \right)
  \liminf_{
    (0,1) \ni \eps \to 0
  }
  \E\!\left[ 
    \exp\!\left(  
      - \rho \left( \eps + X_0 \right)^{ (2 - p) } 
    \right) 
  \right] 
\\ & \leq 
  \exp\!\left(  
    \tfrac{ t \rho^2 }{ 2 }
    \left[ 
      \tfrac{ \rho^2 }{ \left( \rho - \tilde{\rho} \right) \, \left( 2 \alpha + 1 - p \right) } 
    \right]^{ \frac{ (2 p - 2) }{ (2 - p) } }
  \right)
  \E\!\left[ 
    \exp\!\left(  
      - \rho \left( X_0 \right)^{ (2 - p) } 
    \right) 
  \right] .
\end{split} 
\end{equation}
This proves 
\eqref{l:lemma_inv_bound_new:state1}.
In the next step observe that
H\"older's inequality and
\eqref{l:lemma_inv_bound_new:state1}
imply that for all $ t \in [0,T] $,
$ r \in (1,\infty) $ 
it holds that
 {\allowdisplaybreaks  
\begin{equation}
\label{l:lemma_inv_bound_new:7}
\begin{split}
 &
  \E\!\left[ 
    \exp\!\left(
      \smallint_0^t   
        \tilde{\rho} \left( 2 - p \right)
      \left(
        \alpha + \tfrac{ 1 - p }{ 2 } 
      \right)  
      \left( X_s \right)^{ - p } 
      ds 
    \right)
  \right]
\\ & =
  \E\!\left[ 
    \exp\!\left( 
      \smallint_0^t 
      \tfrac{ 
        \tilde{\rho} \,
        \left( 2 - p \right) \,
        \left( 
          \alpha + \frac{ 1 - p }{ 2 } 
        \right)  
      }{ 
        \left( X_s \right)^p
      }  
      - 
        \tilde{c} \, \rho \left( 2 - p \right) 
        \left( X_s \right)^{ (c - p) } 
      ds 
      + 
      \smallint_0^t  
        \tilde{c} \, \rho \left( 2 - p \right) 
        \left( X_s \right)^{ (c - p) } 
      ds 
    \right)
  \right]
\\ & \leq
    \left\|
      \exp\!\left( 
        \smallint_0^t 
          \tilde{\rho} 
          \left( 2 - p \right) 
          \left(
            \alpha + \tfrac{ 1 - p }{ 2 } 
          \right)  
          \left( X_s \right)^{ - p }
          -  
          \tilde{c} \, \rho 
          \left( 2 - p \right)   
          \left( X_s \right)^{ (c - p) } 
        ds  
        - \rho  
        \left( X_t \right)^{ (2 - p) } 
      \right)
    \right\|_{ L^r( \Omega; \R ) }
\\ & \quad 
  \cdot 
    \left\|
      \exp\!\left( 
        \rho 
        \left( X_t \right)^{ (2 - p) } 
        + 
        \smallint_0^t 
          \tilde{c} \, \rho 
          \left( 2 - p \right)   
        \left( X_s \right)^{ (c - p) } 
        ds 
      \right)
    \right\|_{
      L^{ r / ( r - 1 ) }( \Omega; \R )
    }
\\ & \leq
  \exp\!\left(  
    \tfrac{ t \left( r \rho \right)^2 }{ 2 r }
    \left[ 
      \tfrac{  
        \left( r \rho \right)^2 
      }{
        \left( r \rho - r \tilde{\rho} \right) \,
        \left( 
          2 \alpha + 1 - p 
        \right)
      } 
    \right]^{
      \frac{ (2 p - 2) }{ (2 - p) }
    }
  \right)
  \left( 
    \E\!\left[ 
      e^{ 
        - r \, \rho \, \left( X_0 \right)^{ (2 - p) } 
      } 
    \right] 
  \right)^{
    \frac{ 1 }{ r }
  }
\\ & \quad \cdot
    \left\|
      \exp\!\left( 
        \rho 
        \left( X_t \right)^{ (2 - p) } 
        + 
        \smallint_0^t 
          \tilde{c} \, \rho 
          \left( 2 - p \right)   
        \left( X_s \right)^{ (c - p) } 
        ds 
      \right)
    \right\|_{
      L^{ r / ( r - 1 ) }( \Omega; \R )
    }
.
\end{split}
\end{equation}}This
shows \eqref{l:lemma_inv_bound_new:state1B}.
It thus remains to prove
\eqref{l:lemma_inv_bound_new:state3}.
For this observe 
that for all 
$ x,y \in \R $
it holds that
\begin{equation}
\label{l:lemma_inv_bound_new:6110}
  x \vee 0
\leq 
  \left[
    ( x - y ) \vee 0
  \right]
  +
  \left[
    y \vee 0
  \right]
\leq 
  \left[ (x-y) \vee 0 \right] 
  + |y| 
\leq 
  \exp\!\left( x - y - 1 \right) + |y| 
  .
\end{equation}
This and 
\eqref{l:lemma_inv_bound_new:state1}
imply that for all 
$ t \in [0,T] $, $ q \in (0,\infty) $ 
it holds that
 {\allowdisplaybreaks  
\begin{equation} 
\begin{split}
&
  \left\|
      \smallint_0^t  
        \tilde{\rho} \left( 2 - p \right) 
        \left( \alpha + \tfrac{ 1 - p }{ 2 } \right)  
        \left( X_s \right)^{ - p } 
      ds
  \right\|_{ L^q( \Omega; \R ) }
\\ & \leq 
  \bigg\|
    \exp\!\left( 
      \smallint_0^t 
      \tfrac{
        \tilde{\rho} \, \left( 2 - p \right) \,
        \left( \alpha + \frac{ 1 - p }{ 2 } \right)  
      }{
        \left( X_s \right)^p 
      }
      \, ds
      - \rho \left( X_t \right)^{ (2 - p) } 
      - 
      \smallint_0^t  
        \tilde{c} \, \rho \left( 2 - p \right)   
        \left( X_s \right)^{ (c - p) } 
      ds 
      - 1
    \right)
\\ & \quad
  +
      \rho \left( X_t \right)^{ (2 - p) } 
      +
      \smallint_0^t  
        \tilde{c} \, \rho \left( 2 - p \right)   
        \left( X_s \right)^{ (c - p) } 
      ds 
  \bigg\|_{ L^q( \Omega; \R ) }
\\ & \leq 
    \tfrac{ 
      2^{
        0 \vee ( 1 / q - 1 )
      }
    }{ e }
  \left\|
    \exp\!\left( 
      \smallint_0^t 
      \tfrac{
        \tilde{\rho} \, \left( 2 - p \right) \,
        \left( \alpha + \frac{ 1 - p }{ 2 } \right)  
      }{
        \left( X_s \right)^p 
      }
      \, ds
      - \rho \left( X_t \right)^{ (2 - p) } 
      - 
      \smallint_0^t  
        \tilde{c} \, \rho \left( 2 - p \right)   
        \left( X_s \right)^{ (c - p) } 
      ds 
    \right)
  \right\|_{ L^q( \Omega; \R ) }
\\ & \quad
  +
  2^{
    0 \vee ( 1 / q - 1 )
  }
  \left\|
      \rho \left( X_t \right)^{ (2 - p) } 
      +
      \smallint_0^t  
        \tilde{c} \, \rho \left( 2 - p \right)   
        \left( X_s \right)^{ (c - p) } 
      ds 
  \right\|_{ L^q( \Omega; \R ) }
\\ & \leq 
    \tfrac{ 
      2^{
        0 \vee ( 1 / q - 1 )
      }
    }{ e }
    \exp\!\left(  
      \tfrac{ t \left( q \rho \right)^2 }{ 2 q }
      \left[ 
        \tfrac{ 
          \left( q \rho \right)^2 
        }{
          ( q \rho - q \tilde{\rho} ) 
          \left( 2 \alpha + 1 -p \right)
        } 
      \right]^{ 
        \!
        \frac{ (2 p - 2) }{ (2 - p) }
      }
    \right)
  \left|
  \E\!\left[ 
    e^{
      - q \, \rho \, \left( X_0 \right)^{ (2 - p) } 
    }
  \right]
  \right|^{ 1 / q }
\\ & \quad
  +
  \rho \,
  2^{
    0 \vee ( 1 / q - 1 )
  }
  \left\|
      \left( X_t \right)^{ (2 - p) } 
      +
      \smallint\nolimits_0^t  
        \tilde{c} \left( 2 - p \right)   
        \left( X_s \right)^{ (c - p) } 
      ds 
  \right\|_{ L^q( \Omega; \R ) }
  .
\end{split}
\end{equation}}This 
proves \eqref{l:lemma_inv_bound_new:state3}
and thereby finishes the proof
of Lemma~\ref{l:lemma_inv_bound_new}.
\end{proof}
%
%
%
%
%
%
%
%
%
%
%
%
%
%
%
%
%
%
%
%
%
%
%
%
%
%
%
%
%
%
%
%
%
%
%
%
%

An important assumption that we use in 
Lemma~\ref{l:lemma_inv_bound_new} above
and in several results below is the assumption
that there exist real numbers 
$
  \alpha, c \in (0, \infty)
$, 
$ 
  \tilde{c} \in [0,\infty)
$
such that for all $ x \in (0,\infty) $
it holds that
\begin{equation}
\label{eq:key_assumption}
  \alpha - \tilde{c} \, x^{c} \leq x \mu(x) 
  .
\end{equation}
The next remark presents a slightly modified version
of this assumption (see \eqref{eq:modified}) and shows
that the modified version in \eqref{eq:modified} also ensures that
\eqref{eq:key_assumption} is fulfilled.

\begin{remark} 
\label{remark:assumptions}
Assume the setting in Section~\ref{sec:setting_numeric}.
If there exist real numbers
$ \alpha \in (0, \infty) $, $ c \in ( 1, \infty ) $ 
such that for all 
$
  x \in (0,\infty)
$ 
it holds that 
\begin{equation}
\label{eq:modified}
  \alpha 
  - 
  c \, \big( x^{ 1 / c } + x^c \big) 
  \leq x \mu(x) 
  ,
\end{equation}
then we obtain 
from Young's inequality
that for all 
$
  x, \delta \in (0,\infty)
$
it holds that
\begin{align}
  \left[ 
    \alpha - 
    \delta^{ 
      \nicefrac{ 1 }{ ( c^2 - 1 ) } 
    } 
    \left[
     c 
     -
      \tfrac{ 
        1 
      }{ 
        c 
      } 
    \right]
  \right] 
  - 
  \left[
    \tfrac{ 1 }{ \delta c^2 } + c 
  \right]
  x^c
\leq
  \alpha 
  - 
  c
  \,
  \big( 
    x^{ 1 / c } + x^c 
  \big) 
\leq 
  x \mu(x)
  .
\end{align}
\end{remark}

%
%
%
%
%
%

%
%
%
%
%

%
%

%

The next lemma (together 
with Lemma~\ref{l:transformierte_gleichung} below)
extends
Theorem 3.1 in Hurd \& Kuznetsov ~\cite{HurdKuznetsov2008}
to more general square-root diffusion processes.
Note -- in the setting of Lemma~\ref{l:lemma_inv_mom} -- that
if 
$
  \mu(x) = \tfrac{ \alpha }{ x }
$ 
for all $ x \in (0,\infty) $
and some $ \alpha \in (\frac{1}{2},\infty) $, 
then the argument of the exponential function
in \eqref{l:lemma_inv_mom:statement}
is positive
if
$ p \in (0, 2 \alpha - 1 ) $.

\begin{lemma}
\label{l:lemma_inv_mom}
Assume the setting in Section~\ref{sec:setting_numeric}
and assume $X_t(\omega)\in(0,\infty)$ for all $(t,\omega)\in[0,T]\times\Omega$.
Then it holds for all $t \in [0,T]$, $p \in \R$ that
\begin{align} \label{l:lemma_inv_mom:statement}
&  
  \E\!\left[ 
    (X_t)^{-p} 
    \exp\!\left( 
      p \int_0^t  \tfrac{\mu(X_s)}{X_s} - \tfrac{p+1}{2(X_s)^2}  \,  ds   
    \right) 
  \right] 
 \leq \E\! \left[ (X_0)^{-p} \right].
\end{align}
\end{lemma}
\begin{proof} [Proof of Lemma~\ref{l:lemma_inv_mom}]
First of all, we define a function 
$
  \sigma \colon [0,\infty) \to \R 
$ 
by
$
  \sigma(x) = 1
$
for all $ x \in [0,\infty) $
and
we fix a real number $ p \in \R $.
Then we define a function 
$
  U \colon (0,\infty) \to \R
$ 
by
$
  U(x) := - p \log(x)
$
for all $ x \in (0,\infty) $ and 
we note that 
$
  U \in C^2( (0,\infty) , \R )
$.
Corollary~2.5
in Cox et.~al~\cite{CoxHutzenthalerJentzen2013}
can thus be applied to obtain
that for all 
$
  t \in [0,T]
$ 
it holds that
\begin{align} 
\label{l:lemma_inv_mom:1}
  \E\!\left[ 
    ( X_t )^{ - p } 
    \exp\!\left(  
      \smallint_0^t 
      - ( \mathcal{G}_{ \mu, \sigma } U)( X_s ) 
      - \tfrac{ 1 }{ 2 } \left| ( \nabla U )( X_s ) \right|^2 ds
    \right) 
  \right] 
  \leq 
  \E\!\left[ 
    ( X_0 )^{ - p } 
  \right]
  .
\end{align}
This together with the observation 
that for all $x \in (0,\infty)$ it holds that $\nabla U (x) = - \tfrac{p}{x}$ and
 $(\mathcal{G}_{ \mu, \sigma } U ) (x) = - p\tfrac{\mu(x)}{x}   + \tfrac{p }{2x^2} $
finishes the proof of Lemma~\ref{l:lemma_inv_mom}.
\end{proof}

\subsection{Temporal H\"older regularity based on Bessel-type processes}
In the next two lemmas we present two elementary 
and essentially well-known estimates for 
SDEs with a globally one-sided Lipschitz continuous drift
coefficient and an at most linearly growing diffusion coefficient.

\begin{lemma}  
\label{l:lemma_positive_mom}
Assume the setting in 
Section~\ref{sec:setting_numeric},
let $ c \in [0,\infty) $
and assume
$ 
  x \mu(x) 
  \leq 
  c
  \left( 1 + x^2 \right)
$ 
for all 
$ x \in (0,\infty) $.
Then it holds for all $q \in (0,\infty)$, $t \in [0,T]$ that
\begin{align} 
\label{l:lemma_positive_mom:statement}
 &
 \E\big[ 
   ( X_t )^q 
 \big]
\leq
  \E\big[ 
    ( 1 + ( X_t )^2 
    )^{ \frac{ q }{ 2 } } 
  \big]  
\leq 
 \exp\!\Big( 
   t \left[ 1 + ( q - 2 )^+ + 2 c \right] 
 \Big) 
 \,
 \E\big[ 
   ( 
     1 + ( X_0 )^2 
   )^{ \frac{ q }{ 2 } } 
 \big]
 .
\end{align}

\end{lemma}
\begin{proof}[Proof 
of Lemma~\ref{l:lemma_positive_mom}]
Observe that the estimate
$
  x \mu(x) \leq c \left( 1 + x^2 \right)
$  
for all 
$
  x \in [0,\infty)
$ 
implies that for all 
$ x \in [0,\infty) $, 
$ q \in (0,\infty) $ 
it holds that
\begin{align}  
\label{l:lemma_positive_mom:3}
&  
  2 x \mu(x) + 1 
  + 
  \frac{ 2 \left( \frac{ q }{ 2 } - 1 \right) x^2 }{ 1 + x^2 }
\leq 
  1 + ( q - 2 )^+ 
  + 2 c ( 1 + x^2 )
\leq 
  \left( 
    1 + ( q - 2 )^+  
    + 2 c 
  \right) 
  \left( 1 + x ^2 \right)
  .
\end{align}
This and 
Corollary 2.6 in Cox et.~al~\cite{CoxHutzenthalerJentzen2013} 
prove that for all 
$ q \in (0, \infty) $, $ t \in [0,T] $ 
it holds that
\begin{equation} 
\label{l:lemma_positive_mom:4} 
\begin{split} 
  \E\big[ 
    ( X_t)^q 
  \big] 
& 
  \leq  
  \E\big[ 
    ( 1 + (X_t)^2 )^{ \frac{ q }{ 2 } } 
  \big]  
 \leq 
 \exp\!\Big( 
   t \left[ 1 + ( q - 2 )^+ + 2 c \right] 
 \Big) 
 \E\big[ 
   ( 
     1 + ( X_0 )^2 
   )^{ \frac{ q }{ 2 } } 
 \big]
  .
\end{split}
\end{equation}
This finishes the proof of Lemma~\ref{l:lemma_positive_mom}.
\end{proof}
\begin{lemma}  
\label{l:lemma_positive_mom_mit_sup}
Assume the setting in 
Section~\ref{sec:setting_numeric},
let $ c \in [0,\infty) $, $ q \in [2,\infty) $
and assume
$ 
  x \mu(x) 
  \leq 
  c
  \left( 1 + x^2 \right)
$ 
for all 
$ x \in (0,\infty) $.
Then 
\begin{equation} 
\begin{split}  
\label{l:lemma_positive_mom_mit_sup:state} 
 \left\| 
  \sup_{t \in [0,T]}( X_t )^2 
\right\|_{ 
  L^q( 
    \Omega ; 
  \R ) 
} 
\leq 
\left\|
( X_0)^2
\right\|_{ 
  L^q( 
    \Omega ; 
  \R ) 
} 
+
2cT
  \sup_{s \in [0,T]}
\left\|
     X_s 
\right\|_{ 
  L^{2q}( 
    \Omega ; 
  \R ) 
}^2 
+
\sqrt{\tfrac{2T q^3}{q-1}}
  \sup_{s \in [0,T]}
\left\|
X_s
\right\|_{L^{q}( \Omega; \R )}
+
T(2c+1)
  .
\end{split}
\end{equation}
\end{lemma}
\begin{proof}[Proof 
of Lemma~\ref{l:lemma_positive_mom_mit_sup}]
First of all, observe that It\^{o}'s lemma and 
the assumption that 
for all 
$ x \in [0,\infty) $ 
it holds that
$ 
  x \mu(x) 
  \leq 
  c
  \left( 1 + x^2 \right)
$ 
implies that 
for all 
$ t \in [0,T] $ 
it holds that
\begin{equation} 
\label{l:lemma_positive_mom_mit_sup:1} 
\begin{split}  
    ( X_t)^2
& 
=
( X_0)^2
+
2 \int_0^t X_s \mu(X_s) 
 ds
+
t 
+ 
2  \int_0^t X_s \, dW_s
\\
& 
\leq
( X_0)^2
+
2c \int_0^t 1+ (X_s)^2  
 \, ds 
+
t
+ 
2  \int_0^t X_s \, dW_s.
\end{split}
\end{equation}
%
%
%
Furthermore, note that the Burkholder-Davis-Gundy-type inequalities in Lemma 7.2 and Lemma 7.7 in 
Da Prato \& Zabczyk~\cite{dz92} show that 
\begin{equation} 
\begin{split}  
\label{l:lemma_positive_mom_mit_sup:3}
&\left\| 
  \sup_{t \in [0,T]} \int_0^t X_s \, dW_s 
\right\|_{ 
  L^q( 
    \Omega ; 
  \R ) 
} 
\leq 
\sqrt{\tfrac{q^3}{2(q-1)}
\int_{0}^{T}
\left\|
X_s
\right\|_{L^q( \Omega; \R )}^2   ds }
\leq
\sqrt{\tfrac{T q^3}{2(q-1)}}
  \sup_{s \in [0,T]}
\left\|
X_s
\right\|_{L^{q}( \Omega; \R )}
.
\end{split} 
\end{equation} 
Combining \eqref{l:lemma_positive_mom_mit_sup:1} and \eqref{l:lemma_positive_mom_mit_sup:3} 
proves that
\begin{equation} 
\begin{split}  
\label{l:lemma_positive_mom_mit_sup:4}
&\left\| 
  \sup_{t \in [0,T]}( X_t )^2 
\right\|_{ 
  L^q( 
    \Omega ; 
  \R ) 
} 
\\ &
\leq
\left\|
( X_0)^2
\right\|_{ 
  L^q( 
    \Omega ; 
  \R ) 
} 
+
2c 
\left\|
\int_0^T (X_s)^2  
 \, ds 
\right\|_{ 
  L^q( 
    \Omega ; 
  \R ) 
} 
+
2 
\left\|
  \sup_{t \in [0,T]} 
 \int_0^t X_s \, dW_s
\right\|_{ 
  L^{q}( 
    \Omega ; 
  \R ) 
} 
+
T(2c+1)
\\ &
\leq
\left\|
( X_0)^2
\right\|_{ 
  L^q( 
    \Omega ; 
  \R ) 
} 
+
2cT
  \sup_{s \in [0,T]}
\left\|
    ( X_s )^{ 
      2
    } 
\right\|_{ 
  L^q( 
    \Omega ; 
  \R ) 
} 
+
\sqrt{\tfrac{2T q^3}{q-1}}
  \sup_{s \in [0,T]}
\left\|
X_s
\right\|_{L^{q}( \Omega; \R )}
+
T(2c+1)
  .
\end{split} 
\end{equation} 
This finishes the proof of Lemma~\ref{l:lemma_positive_mom_mit_sup}.
\end{proof}
\begin{lemma}
\label{p:lemma_conv_order_0}
Let $ T \in [0,\infty) $, 
let
$
  ( 
    \Omega, \mathcal{F}, \P  
  )
$
be a probability space, 
let 
$
  Z \in \mathcal{L}^0( [0,T] \times \Omega; \R )
$
satisfy
$
  \int_0^T \left| Z_s \right| ds < \infty
$ $\P$-a.s.\
and let $ Y \colon [0,T] \times \Omega \to \R $
be a stochastic process with continuous sample paths
satisfying
$
  Y_t = \smallint_0^t Z_r \, dr $ $ \P $-a.s.\ for all $ t \in [0,T] 
$.
Then for all $ q \in ( 0, \infty) $, $ p \in [1, \infty) $ it holds that
\begin{equation} 
\label{p:lemma_conv_order_0:statement1}
  \left\|  
 \left\|
Y
\right\|_{ \mathcal{C}^{ ( 1 - \nicefrac{ 1 }{ p } ) }( [0,T], \R ) }
  \right\|_{ 
    L^q( 
      \Omega ; 
    \R ) 
    ) 
  } 
\leq
  \left\| 
    Z
  \right\|_{
    L^q( \Omega; L^p( [0,T]; \R ) )
  }
\end{equation}
and for all $p, q \in ( 0, \infty) $, $\delta \in [ 0, \infty] $, $\theta \in [0,1]$ satisfying
$
  q \left( p - \theta \right) \left( 1 + \delta \right) \geq p
$  it holds that
\begin{equation} 
\label{p:lemma_conv_order_0:statement2} 
\begin{split}
& 
  \left\|  
Y
  \right\|_{ 
    \mathcal{C}^{ \left( 1 - \nicefrac{ \theta }{ p } \right) }( [0,T], L^q( \Omega; \R) ) 
  }
\leq
    \left\|
      Z 
    \right\|_{
      L^{
        q \theta ( 1 + \nicefrac{ 1 }{ \delta } )
      }( 
        \Omega; 
        L^p( [0,T]; \R ) 
      )
    }^{ \theta }
    \left\|
      Z
    \right\|_{
      L^{ \infty }( [0,T]; 
        L^{ 
          q ( 1 - \theta ) ( 1 + \delta ) 
        }(
          \Omega; \R
        )
      )
    }^{
      ( 1 - \theta ) 
    }
  .
\end{split}
\end{equation}
\end{lemma}
\begin{proof} [Proof of Lemma~\ref{p:lemma_conv_order_0}]
H\"older's inequality 
shows that 
for all $ q \in ( 0, \infty) $, $ p \in [1, \infty) $ 
it holds that
\begin{equation}  
\label{p:lemma_conv_order_0:7}
\begin{split}
&
  \left\|
    \sup_{ s, t \in [0,T] } 
    \left[ 
      \left| t - s \right|^{
        - 
        \left[ 1 - \frac{ 1 }{ p } \right]
      }
 \left|\smallint_s^t  Z_r \, dr \right| \right] \right\|_{ L^{q}( \Omega; \R ) }
\\ & \leq
 \left\| 
\sup_{s,t \in [0,T]} \left[ 
  \left| t - s \right|^{ 
    - \left[ 1 - \frac{ 1 }{ p } \right] 
  }
\left(
\smallint_s^t | Z_r |^p \, dr 
\right)^{ \! \frac{ 1 }{ p } }
\left(
\smallint_s^t 1 \, dr 
\right)^{ 
  \! \left[ 1 - \frac{ 1 }{ p } \right]
} 
\right]
 \right\|_{ L^{q}( \Omega; \R ) } 
=
 \left\| 
\smallint_0^T | Z_r |^p \, dr 
 \right\|_{ L^{ \nicefrac{ q }{ p } }( \Omega; \R ) }^{ \frac{ 1 }{ p } }
 .
\end{split}
\end{equation}
This proves \eqref{p:lemma_conv_order_0:statement1}.
Again H\"older's inequality shows 
that for all 
$ q \in ( 0, \infty) $, 
$ \delta \in [0,\infty] $,
 $ \theta \in (0,1] $,
$ p \in (\theta, \infty) $ 
satisfying
$
  q \left( p - \theta \right) \left( 1 + \delta \right) \geq p
$  
it holds that  
{\allowdisplaybreaks  
\begin{align}  
\label{p:lemma_conv_order_0:8} \nonumber
&
  \sup_{
    \substack{
      s, t \in [0,T] ,
    \\
      s \neq t
    }
  } 
  \left[
    \left| t - s 
    \right|^{
      - \left[ 1 - \frac{ \theta }{ p } \right]
    }
  \left\|
    \smallint_s^t  Z_r \, dr   
  \right\|_{ L^q( \Omega; \R ) }
 \right]
=
  \sup_{
    \substack{
      s, t \in [0,T] ,
    \\
      s \neq t
    }
  } 
  \left[
    \left| t - s 
    \right|^{
      - \left[ 1 - \frac{ \theta }{ p } \right]
    }
   \left\|
     \smallint_s^t  
       \left| Z_r \right|^{ \theta }
       \left| Z_r \right|^{ (1 - \theta) } 
     dr   
   \right\|_{ 
     L^q( \Omega; \R ) 
   }
 \right]
\\ \nonumber & \leq
  \sup_{
    \substack{
      s, t \in [0,T] ,
    \\
      s \neq t
    }
  } 
  \left[
    \left| t - s 
    \right|^{
      - \left[ 1 - \frac{ \theta }{ p } \right]
    }
 \left\| 
  \left(
    \smallint_s^t | Z_r |^p \, dr 
  \right)^{ \! \frac{ \theta }{ p } }
  \left(
    \smallint_s^t 
    | Z_r |^{
      \frac{ p ( 1 - \theta ) }{ p - \theta }
    }  
    dr 
  \right)^{ 
    \! 1 - \frac{ \theta }{ p } 
  }
 \right\|_{ L^{q}( \Omega; \R ) }  \right]
\\ & \leq
  \sup_{
    \substack{
      s, t \in [0,T] ,
    \\
      s \neq t
    }
  } 
  \left[
    \left| t - s 
    \right|^{
      - \left[ 1 - \frac{ \theta }{ p } \right]
    }
    \left\| 
      \left(
        \smallint_s^t | Z_r |^p \, dr 
      \right)^{ \! \frac{ \theta }{ p } }
    \right\|_{ 
      L^{ q ( 1 + \nicefrac{ 1 }{ \delta } ) }( \Omega; \R ) 
    }
    \left\|
      \left(
        \smallint_s^t 
        | Z_r |^{ \frac{ p ( 1 - \theta ) }{ p - \theta } } 
        dr 
      \right)^{ 
        \! 1 - \frac{ \theta }{ p } 
      }
   \right\|_{ 
     L^{ 
       q ( 1 + \delta ) 
     }( \Omega; \R ) 
   }  
 \right]
\\ \nonumber & =
  \sup_{
    \substack{
      s, t \in [0,T] ,
    \\
      s \neq t
    }
  } 
  \left[
    \left| 
      t - s 
    \right|^{
      - \left[ 1 - \frac{ \theta }{ p } \right]
    }
    \left\|
      Z 
    \right\|_{
      L^{
        q \theta ( 1 + \nicefrac{ 1 }{ \delta } )
      }( 
        \Omega; 
        L^p( [s,t]; \R ) 
      )
    }^{ \theta }
%
    \left\| 
      \smallint_s^t 
        | Z_r |^{ \frac{ p ( 1 - \theta ) }{ p - \theta } } 
      dr 
    \right\|_{ 
      L^{ \frac{q(p-\theta)(1+\delta)}{p}
    }( \Omega; \R ) }^{ 1 - \frac{ \theta }{ p } }  
  \right]
\\ \nonumber & \leq
  \sup_{
    \substack{
      s, t \in [0,T] ,
    \\
      s \neq t
    }
  } 
  \left[
    \left| 
      t - s 
    \right|^{
      - \left[ 1 - \frac{ \theta }{ p } \right]
    }
    \left\|
      Z 
    \right\|_{
      L^{
        q \theta ( 1 + \nicefrac{ 1 }{ \delta } )
      }( 
        \Omega; 
        L^p( [s,t]; \R ) 
      )
    }^{ \theta }
    \left(
      \smallint_s^t 
      \left\| 
        Z_r
      \right\|_{ L^{q(1-\theta)(1+\delta)}( \Omega; \R ) }^{\frac{p(1-\theta)}{p-\theta}}  
      dr 
    \right)^{ \!\! 1 - \frac{ \theta }{ p } }
  \right]
\\ \nonumber & \leq
    \left\|
      Z 
    \right\|_{
      L^{
        q \theta ( 1 + \nicefrac{ 1 }{ \delta } )
      }( 
        \Omega; 
        L^p( [0,T]; \R ) 
      )
    }^{ \theta }
    \left\|
      Z
    \right\|_{
      L^{ \infty }( [0,T]; 
        L^{ 
          q ( 1 - \theta ) ( 1 + \delta ) 
        }(
          \Omega; \R
        )
      )
    }^{
      ( 1 - \theta ) 
    }
  .
\end{align} }%
This finishes the proof of Lemma~\ref{p:lemma_conv_order_0}.
\end{proof}
\begin{proposition}[Temporal H\"older regularity for Bessel-type processes]
\label{p:lemma_conv_order_1}
Assume the setting in Section~\ref{sec:setting_numeric},
let
$ \alpha, q \in ( 0, \infty) $, 
$
  \eps \in 
  \big( 
    \frac{ [ \alpha - 1/2 ]^+ }{ 1 + 2 \alpha }     
  , 
    \frac{ 2 \alpha }{ 1 + 2 \alpha }
  \big)
$,  
$
  p = 
  \frac{ 
    1 + 2 \alpha 
  }{
    \eps \, ( 1 + 2 \alpha ) + 1 
  }
$,
$ c \in [p, \infty)$, $\tilde{c} \in [0, \infty) $
and assume
that for all 
$ x \in (0,\infty)
$ 
it holds that
$
  \alpha - 
  \tilde{c} 
  \, x^c  
  \leq 
  x \mu(x) 
  \leq 
  \tilde{c}
  \left( 1 + x^2 \right)
$ 
and
$
  \E\big[ 
    ( X_0 )^{
      \max\{ c - 1, 1 \}
      \max\!\left\{ 
        p , 
        q , 
        \nicefrac{ 
          ( q - p ) 
          \,
          ( 1 + 2 \alpha + p ) 
        }{ 
          (1 + 2 \alpha - p ) 
        }
      \right\}
    } 
  \big] < \infty
$.
Then
\begin{equation} 
\label{p:lemma_conv_order_1:statement1}
  \left\| 
      \left\|  
     X - W  
       \right\|_{ \mathcal{C}^{ (\nicefrac{2\alpha}{(1+2\alpha)}) -\eps }( [0,T],  \R) }
   \right\|_{ 
    L^q( 
      \Omega ; 
      \R
    ) 
  } 
\leq
  \big\|
    \!
    \left\| \mu( X ) \right\|_{
      L^p( [0,T]; \R )
    }
    \!
  \big\|_{
    L^q( \Omega; \R )
  }
  < \infty
  .
\end{equation}
If, in addition to the above assumptions, 
$
  \sup_{ t \in [0,T] } 
  \E\!\left[ 
    ( X_t )^{ - r } 
  \right] 
  < \infty
$ 
for all 
$ 
  r \in 
  \big[ 
    0,
    ( \nicefrac{ q }{ 2p } \wedge \nicefrac{ 1 }{ 2 } ) 
    ( 1 + 2 \alpha + p ) 
  \big]
$,
then 
\begin{equation} 
\label{p:lemma_conv_order_1:statement2} 
\begin{split}
& 
  \left\|  
    X - W  
  \right\|_{ 
    \mathcal{C}^{\left(\left(\nicefrac{2 \alpha(q+1) + 1}{q(1+2 \alpha)}\right)  - \eps\right) \wedge 1 }( [0,T], L^q( \Omega; \R) ) 
  }
\\ & \leq
\sup_{u \in [0,T]}
 \Big\| 
\mu(X_u)
\Big\|_{ L^{\left( \nicefrac{q}{2p} \wedge  \nicefrac{1}{2} \right) (1 + 2 \alpha + p) }( \Omega; \R ) }^{ \nicefrac{p}{q} \wedge 1 } 
 \left\| 
\smallint_0^T | \mu(X_r) |^p \, dr 
 \right\|_{ L^{  \left(\nicefrac{q}{p} -1 \right) \left(\nicefrac{(1 + 2 \alpha + p)}{(1 + 2 \alpha - p)} \right)}( \Omega; \R ) }^{ \nicefrac{1}{p}- \nicefrac{1}{q} }
  \in [0,\infty).
\end{split}
\end{equation}
\end{proposition}
\begin{proof} [Proof of Proposition~\ref{p:lemma_conv_order_1}]
%
%
%
Lemma~\ref{l:lemma_positive_mom} 
and the assumption that
\begin{equation}
  \E\!\left[ 
    (X_0)^{
      \max\{ c - 1, 1 \}
      \max\!\left\{ 
        p , 
        q , 
        \frac{ 
          ( q - p ) 
          \,
          ( 1 + 2 \alpha + p ) 
        }{ 
          (1 + 2 \alpha - p ) 
        }
      \right\}
    } 
  \right] < \infty
\end{equation}
imply that 
\begin{equation}
\label{eq:exact_sol_moments}
  \sup_{ t \in [0,T] }
  \E\!\left[ 
    ( X_t )^{
      \max\{ c - 1, 1 \}
      \max\!\left\{ 
        p , 
        q , 
        \frac{ 
          ( q - p ) 
          \,
          ( 1 + 2 \alpha + p ) 
        }{ 
          (1 + 2 \alpha - p ) 
        }
      \right\}
    } 
  \right] 
   < \infty
   .
\end{equation}
Next we note that
Jensen's inequality ensures 
that for all $ r \in (0,\infty) $ 
it holds that
\begin{equation} 
\begin{split}  
\label{p:lemma_conv_order_1:5}
  & 
    \E\Bigg[ 
    \left|
      \smallint_0^T  
      ( X_u )^{ p \max\{ c - 1 , 1 \} } 
      \, du 
     \right|^r 
   \Bigg]
\leq
  \E\!\left[
    1 + 
    \left|
      \smallint_0^T  
      ( X_u )^{ 
        p \left( ( c - 1 ) \vee 1 \right) 
      } \, du 
    \right|^{ r \vee 1 } 
  \right]
\\ & =
  1 + 
  \E\!\left[
    \left| 
      \tfrac{ 1 }{ T }
      \smallint_0^T  
        T \,
        ( X_u )^{ 
          p \left( ( c - 1 ) \vee 1 \right)
        } 
      \, du 
     \right|^{ r \vee 1 } 
   \right]
\leq
  1 +  
  T^{ [ r - 1 ]^+ } 
  \left[
  \smallint_0^T 
  \E\!\left[ 
    ( X_u )^{ 
      p \left( ( c - 1 ) \vee 1 \right) \left( r \vee 1 \right) 
    } 
  \right]
  du
  \right] 
  .
\end{split} 
\end{equation} 
In the next step we observe that
the assumption that for all $ x \in (0,\infty) $
it holds that
$
  \alpha -\tilde{c} \, x^c \leq x \mu(x) \leq \tilde{c} \left( 1 + x^2 \right)
$ 
implies that 
for all $ x \in [0,\infty) $ it holds that
$
  | \mu(x)| 
  \leq 
  \tfrac{ 2 \, ( \alpha + \tilde{c} ) }{ x } 
  \left( 1 +  x^{ c \vee 2 } \right) 
$.
This, in turn, ensures that
for all $ x \in [0,\infty) $ it holds that
\begin{equation}
\label{eq:mu_estimate}
  | \mu(x)|^p  
  \leq 
  4^p
  \left[ 
    \alpha + \tilde{c} 
  \right]^p 
  \left( 
    \left[ 
      \tfrac{ 1 }{ x }
    \right]^p 
    + 
    x^{
      p \left( ( c - 1 ) \vee 1 \right)
    } 
  \right)
  .
\end{equation}
Lemma~\ref{l:lemma_inv_bound_new}
together with the
estimates~\eqref{eq:exact_sol_moments},
\eqref{p:lemma_conv_order_1:5} 
and \eqref{eq:mu_estimate}
implies that
\begin{equation} 
\begin{split}  
\label{p:lemma_conv_order_1:2}  
  & \left\| 
    \smallint_0^T 
    | \mu( X_r ) |^p 
    \, dr 
  \right\|_{ 
    L^{ 
      \frac{ q }{ p } \vee 
      \left[ 
        \frac{ ( q - p ) }{ p } 
        \frac{ 1 + 2 \alpha +p }{ 1 + 2 \alpha - p }
      \right]
    }( \Omega; \R ) 
  }
\leq
   4^p 
   \left[ 
     \alpha + \tilde{c}
   \right]^p
 \left\| 
 \smallint_0^T  
   \big[ 
     \tfrac{ 1 }{ X_r }
   \big]^p  
   \,
   dr 
 \right\|_{ 
    L^{ 
      \frac{ q }{ p } \vee 
      \left[ 
        \frac{ ( q - p ) }{ p } 
        \frac{ 1 + 2 \alpha +p }{ 1 + 2 \alpha - p }
      \right]
    }( \Omega; \R ) 
 }
\\ & 
\qquad
+
 4^p 
 \left[ \alpha + \tilde{c} \right]^p
 \left\| 
   \smallint_0^T  
   ( X_r )^{
     p \left( ( c - 1 ) \vee 1 \right)
   } 
   \, dr 
 \right\|_{ 
    L^{ 
      \frac{ q }{ p } \vee 
      \left[ 
        \frac{ ( q - p ) }{ p } 
        \frac{ 1 + 2 \alpha +p }{ 1 + 2 \alpha - p }
      \right]
    }( \Omega; \R ) 
 }
 < \infty
 .
\end{split} 
\end{equation}  
Lemma~\ref{p:lemma_conv_order_0} and \eqref{p:lemma_conv_order_1:2} 
show that 
\begin{equation} 
\begin{split}   \label{p:lemma_conv_order_1:7} 
&\left\|\sup_{s,t \in [0,T]} \left[ 
\left| t-s \right| ^{- \frac{p-1}{p}
}
 |X_t - X_s - W_t + W_s| \right] \right\|_{ L^{q}( \Omega; \R ) }
\\ 
&
=
\left\|\sup_{s,t \in [0,T]} \left[ 
\left| t-s \right| ^{- \frac{p-1}{p}
}
 \left|\smallint_s^t  \mu(X_r) \, dr \right| \right] \right\|_{ L^{q}( \Omega; \R ) }
\leq
 \left\| 
\smallint_0^T | \mu(X_r) |^p \, dr 
 \right\|_{ L^{\frac{q}{p}}( \Omega; \R ) }^{\frac{1}{p}} 
 < \infty
 .
\end{split} 
\end{equation} 
This together with the identity
that 
$
  \tfrac{p-1}{p} = \frac{2 \alpha}{1+2 \alpha} - \eps   
  \in \left(0, \tfrac{1}{2}\right)
$ 
proves 
\eqref{p:lemma_conv_order_1:statement1}.
%
%
%
We define $\theta := \left( 1-\tfrac{p}{q} \right) \vee 0  \in [0,1]$ and $\delta:=  \tfrac{1+ 2 \alpha}{2p} -\tfrac{1}{2}   \in (0,\infty)$. 
Lemma~\ref{p:lemma_conv_order_0} and 
$
  q \left( p - \theta \right) \left( 1 + \delta \right) \geq p
$ 
yield  
{\allowdisplaybreaks  
\begin{align}  
\label{p:lemma_conv_order_1:8} 
\nonumber
&
  \sup_{
    s, t \in [0,T]
  } 
  \left| t - s \right|^{ 
    - \frac{ p - \theta }{ p } 
  }
  \left[ \left\|
 X_t - X_s - W_t + W_s  \right\|_{ L^{q}( \Omega; \R ) }\right]
=
\sup_{s,t \in [0,T]} \left[\left| t-s \right| ^{- \frac{p-\theta}{p}
}
\left\|
 \smallint_s^t  \mu(X_r) \, dr   \right\|_{ L^{q}( \Omega; \R ) }\right]
\\ 
&
\leq
 \left\| 
\smallint_0^T | \mu(X_r) |^p \, dr 
 \right\|_{ L^{\frac{q \theta (1+\delta)}{p \delta}}( \Omega; \R ) }^{\frac{\theta}{p}}
\sup_{u \in [0,T]}
 \left\| 
\mu(X_u)
\right\|_{ L^{q(1-\theta)(1+\delta)}( \Omega; \R ) }^{1-\theta}.
\end{align} }%
Note that $\tfrac{\theta}{p} = \left( \tfrac{1}{p}- \tfrac{1}{q} \right) \vee 0$, $1 - \theta = \tfrac{p}{q} \wedge 1$,
$1+\delta =
 \tfrac{1 + 2 \alpha + p}{2p}
$,
$q(1-\theta)(1+\delta)
=   \left( \tfrac{q}{p} \wedge 1 \right)\tfrac{1 + 2 \alpha + p}{2}$.
This together with \eqref{eq:exact_sol_moments} and $\sup_{t \in [0,T]} \E\!  \left[ (X_t)^{-r} \right] < \infty $ for all $r \in \left[0,\left( \tfrac{q}{p} \wedge 1 \right)\tfrac{1 + 2 \alpha + p}{2} \right]$ implies that
\begin{equation} 
\begin{split}  \label{p:lemma_conv_order_1:12bb}
\sup_{u \in [0,T]}
 \left\| 
\mu(X_u) 
\right\|_{ L^{q(1-\theta)(1+\delta)}( \Omega; \R ) }^{1-\theta}
\leq
2(\alpha +\tilde{c})
\sup_{u \in [0,T]}
 \left\| 
\tfrac{1}{X_u} + X_u^{(c-1) \vee 1}
\right\|_{ L^{\left( \frac{q}{p} \wedge 1 \right)\frac{1 + 2 \alpha + p}{2}}( \Omega; \R ) }^{1-\theta} \in [0,\infty).
\end{split}
\end{equation} 
Furthermore, observe that  
$\tfrac{1+\delta}{\delta} =
1 + \tfrac{2p}{1 + 2 \alpha - p}
=
\tfrac{1 + 2 \alpha + p}{1 + 2 \alpha - p}
$,
$
\tfrac{q \theta (1+\delta)}{p \delta} = \left( \left(\tfrac{q}{p} -1 \right) \vee 0 \right) \tfrac{1 + 2 \alpha + p}{1 + 2 \alpha - p}$ and that  $\tfrac{p - \theta}{p} =1- \left( \left(\tfrac{1}{p} -\tfrac{1}{q}\right) \vee 0 \right) 
= \left(\tfrac{p-1}{p} +\tfrac{1}{q}\right) \wedge 1  
 =  \left( \tfrac{2 \alpha}{1+2 \alpha}  + \tfrac{1}{q} - \eps \right) \wedge 1$.
This, \eqref{p:lemma_conv_order_1:8}, \eqref{p:lemma_conv_order_1:2} and \eqref{p:lemma_conv_order_1:12bb} imply that 
\begin{equation} 
\begin{split}  \label{p:lemma_conv_order_1:12}
& \|  X - W  \|_{ 
    \mathcal{C}^{\left(\frac{2 \alpha}{1+2 \alpha}  + \frac{1}{q} - \eps\right) \wedge 1 }( [0,T], L^q( \Omega; \R) ) 
  }
  \\ 
  & 
\leq
 \left\| 
\smallint_0^T | \mu(X_r) |^p \, dr 
 \right\|_{ L^{ \left(\frac{q}{p} -1 \right)  \frac{1 + 2 \alpha + p}{1 + 2 \alpha - p}}( \Omega; \R ) }^{  \frac{1}{p}- \frac{1}{q} }
\sup_{u \in [0,T]}
 \Big\| 
\mu(X_u)
\Big\|_{ L^{\left( \frac{q}{p} \wedge 1 \right)\frac{1 + 2 \alpha + p}{2} }( \Omega; \R ) }^{ \frac{p}{q} \wedge 1 } 
\in [0,\infty).
\end{split} 
\end{equation} 
This finishes the proof of Proposition~\ref{p:lemma_conv_order_1}.
\end{proof}
%
%
%
%
%
%
%
%
%
%
%
%
%
%
%
%
%
%
%
%
%
%
%
%
%
%
%
%

%
%
%
%

The next lemma presents an estimate that allows us to 
obtain H\"older-continuity of the mapping
$
  [0,T] \ni t \mapsto \tfrac{ 1 }{ X_t } \in L^1(\Omega;\R)
$
from H\"older-continuity of the mappings
$
  [0,T] \ni t \mapsto X_t \in L^q(\Omega;\R)
$,
$ q \in (0,\infty) $.
In particular, in the setting of Section~\ref{sec:setting_numeric}, this implies H\"older-continuity
of the mapping $[0,T] \ni t \to \mu(X_t) \in L^1(\Omega;\R)$ under certain assumptions.
We also refer the reader to the literature on strong approximations
of SDEs with discontinuous coefficients where different but related difficulties
arise; see, e.g., \cite{Yan2002}, \cite{Etore2006}, \cite{ForoushTahmasebi2012} and \cite{Przybylowicz2013}.

\begin{lemma}  \label{l:lemma_conv_order}
Let 
$ 
  T \in [0,\infty)
$,
$ 
  c \in (0,\infty)
$,
$
  f \in \mathcal{L}^0( [0,\infty) ; \R )
$,
 let
$
  ( 
    \Omega, \mathcal{F}, \P
  )
$
be a probability space, 
let 
$
  X \colon [0,T] \times \Omega\to[0,\infty)
$
be a stochastic process and assume 
$
  \sup_{t \in [0,T]} \P\!\left[X_t = 0\right] = 0
$
and
$
  |
   f(x) - f(y)
  | 
  \leq c \left| x - y \right| 
  \big(
    \frac{ 1 }{ x } + 
    \frac{ 1 }{ y } + 
    \frac{ 1 }{ x y } + x^c + y^c 
  \big)
$ 
for all 
$ x, y \in (0,\infty) $.
Then it holds for all $\theta, \delta, q \in (0,\infty)$, $ \gamma \in (0, 1]$
 that
\begin{equation} 
\begin{split} 
  \label{l:lemma_conv_order:statement}
 & 
  \| f( X ) \|_{ 
    \mathcal{C}^{ \gamma \theta 
    }( [0,T], L^q( \Omega; \R) ) 
  }
\leq
  c
  \,
  7^{
    \max\left\{ 
      1 ,
      \frac{ 1 }{ q ( 1 + \delta ) } 
    \right\}  
  }
  \| X  \|_{ 
    \mathcal{C}^{ \theta  
    }( 
      [0,T], 
      L^{
        q \gamma (1 / \delta + 1 )
      }( \Omega; \R) 
    ) 
  }^{ \gamma }  
\\ & \quad
  \cdot
\sup_{ u \in [0,T] }
  \left[ 
    1 +
    \left\| 
      \tfrac{ 1 }{ X_u }
    \right\|_{ 
      L^{ q ( 1 + \gamma ) ( 1 + \delta ) 
      }( \Omega; \R ) 
    }^{ 1 + \gamma }
    + 
    \Big\| 
      X_u
    \Big\|_{ 
      L^{
        q ( 1 + \delta ) 
        \left( \max\left\{ \nicefrac{ 1 }{ \gamma } , c + 1 \right\} - \gamma \right)
      }( \Omega; \R )
    }^{
      \max\left\{ \nicefrac{ 1 }{ \gamma } , c + 1 \right\} - \gamma 
    }   
  \right] . 
\end{split} 
\end{equation} 

\end{lemma}
\begin{proof} [Proof of Lemma~\ref{l:lemma_conv_order}]
First of all, 
observe that 
Young's inequality implies 
that for all 
$
  x, y \in (0,\infty)
$, 
$
  \gamma \in (0, 1]
$ 
it holds that
 {\allowdisplaybreaks 
\begin{align}   
\label{l:lemma_conv_order:1} 
\nonumber
 &
  \left( x + y \right)^{ (1 - \gamma) } 
  \left( 
    \tfrac{ 1 }{ x } + 
    \tfrac{ 1 }{ y } + 
    \tfrac{ 1 }{ x y }
  \right)
=
 \tfrac{x+y}{x(x+y)^{\gamma}} + \tfrac{x+y}{y(x+y)^{\gamma}} + \tfrac{x+y}{xy(x+y)^{\gamma}} 
=
\tfrac{2 + \frac{x}{y} + \frac{y}{x}}{(x+y)^{\gamma}} 
+
\tfrac{ \frac{1}{y} + \frac{1}{x}}{(x+y)^{\gamma}} 
\\ \nonumber
&
\leq
\tfrac{1}{x^{\gamma}} + \tfrac{1}{y^{\gamma}} 
+
\tfrac{x^{1-\gamma}}{y} + \tfrac{y^{1-\gamma}}{x}
+
\tfrac{1}{y^{1+\gamma}} + \tfrac{1}{x^{1+\gamma}}
\\ \nonumber
&
\leq
\tfrac{\gamma}{1+\gamma} \tfrac{1}{x^{1+\gamma}} + \tfrac{1}{1+\gamma}
+
\tfrac{\gamma}{1+\gamma} \tfrac{1}{y^{1+\gamma}} + \tfrac{1}{1+\gamma}
+
\tfrac{1}{1+\gamma} \tfrac{1}{y^{1+\gamma}} + \tfrac{\gamma}{1+\gamma} x^{\frac{1}{\gamma} -\gamma}
+
\tfrac{1}{1+\gamma} \tfrac{1}{x^{1+\gamma}} + \tfrac{\gamma}{1+\gamma} y^{\frac{1}{\gamma} -\gamma}
+
 \tfrac{1}{y^{1+\gamma}} +  \tfrac{1}{x^{1+\gamma}} 
\\
&
=
\tfrac{2}{x^{1+\gamma}} + \tfrac{2}{y^{1+\gamma}} +   \tfrac{\gamma}{1+ \gamma} x^{\frac{1}{\gamma} -\gamma} +   \tfrac{\gamma}{1+ \gamma}  y^{\frac{1}{\gamma} -\gamma}  + \tfrac{2}{1+\gamma}.
\end{align}}%
%
%
Next note that the
assumption that
$ 
  | f(x) - f(y) | \leq 
  c \left| x - y \right|
  \big(
    \frac{ 1 }{ x } + 
    \frac{ 1 }{ y } + 
    \frac{ 1 }{ x y } + x^c + y^c 
  \big)
$ 
for all 
$ x, y \in (0,\infty) $, 
the assumption 
$
  \sup_{t \in [0,T]} \P\!\left[X_t = 0\right] = 0
$, 
H\"older's inequality
and 
\eqref{l:lemma_conv_order:1} show 
that 
for all 
$ t, s \in [0,T] $, 
$ \delta, q \in (0,\infty) $, 
$ \gamma \in (0, 1] $ 
it holds that
{\allowdisplaybreaks 
\begin{align}  \nonumber
\label{l:lemma_conv_order:2} 
 & 
  \left\|  
    f(X_t) - f(X_s) 
  \right\|_{ L^q( \Omega; \R ) } 
\\   \nonumber & \leq 
  c \left\| |X_t - X_s| 
\left(\tfrac{1}{X_t} + \tfrac{1}{X_s}+ \tfrac{1}{X_t X_s} +(X_t)^{c} + (X_s)^{c} \right) \1_{\{X_t, X_s \in (0,\infty) \} }
 \right\|_{ L^q( \Omega; \R )} 
\\  \nonumber &= 
  c 
  \left\| 
    \left| X_t - X_s \right|^{ \gamma } 
    \left| X_t - X_s \right|^{ (1 - \gamma) } 
    \left(
      \tfrac{ 1 }{ X_t } + 
      \tfrac{ 1 }{ X_s } + 
      \tfrac{ 1 }{ X_t X_s } +
      ( X_t )^c + 
      ( X_s )^c 
    \right) 
    \1_{ 
      \{ X_t, X_s \in (0,\infty) \} 
    }
 \right\|_{ L^q( \Omega; \R ) }  
\\ & \leq 
  c
  \left\| 
    \left| X_t - X_s \right|^{ \gamma }  
  \right\|_{ 
    L^{ q \left( 1 + \frac{1}{ \delta} \right) }( \Omega; \R )
  }
\\ &   \nonumber
\cdot
 \left\|  \left| X_t - X_s \right|^{1-\gamma} \hspace{-0.05cm}
\left(\tfrac{1}{X_t} + \tfrac{1}{X_s}+ \tfrac{1}{X_t X_s} +
      ( X_t )^c + 
      ( X_s )^c 
  \right) \hspace{-0.1cm}
\1_{\{X_t, X_s \in (0,\infty) \} } 
 \right\|_{ L^{q(1+\delta)}( \Omega; \R )}
%
%
\\  \nonumber & \leq 
 c
 \left\| 
   X_t - X_s 
 \right\|_{  
   L^{ q \gamma \left( 1 + \frac{ 1 }{ \delta } \right)
   }( \Omega; \R )
 }^{ \gamma }
\\ & \nonumber
 \cdot 
\left\| 
  \tfrac{2}{(X_t)^{1+\gamma}} + \tfrac{2}{(X_s)^{1+\gamma}} +  \tfrac{\gamma}{1+\gamma}   (X_t)^{\frac{1}{\gamma} -\gamma} + \tfrac{\gamma}{1+\gamma} (X_s)^{\frac{1}{\gamma} -\gamma}  + \tfrac{2}{1+\gamma} 
  + 2 \left( X_t \right)^{ c + 1 - \gamma }
  + 2 \left( X_s \right)^{ c + 1 - \gamma }
 \right\|_{ L^{ q ( 1 + \delta ) }( \Omega; \R ) }
 .
\end{align}}
Jensen's inequality
hence shows that
for all 
$ t, s \in [0,T] $, 
$ \delta, q \in (0,\infty) $, 
$ \gamma \in (0, 1] $ 
it holds that
{\allowdisplaybreaks 
\begin{align}
 &  \nonumber
  \left\|  
    f(X_t) - f(X_s) 
  \right\|_{ L^q( \Omega; \R ) } 
\\ & \leq 
  c
  \,
  7^{ 
    \left[ 
      \frac{ 1 }{ q ( 1 + \delta ) } - 1 
    \right]^+
  }
  \left\| X_t - X_s \right\|_{ 
    L^{ 
      q \gamma \left( 1 + \frac{ 1 }{ \delta } \right)
    }
  }^{ \gamma }
\\  \nonumber &  
  \cdot
\sup_{u \in [0,T]}
\left[
  \tfrac{ 2 }{ 1 + \gamma }
  +
  4 
  \left\| 
    \tfrac{ 1 }{ X_u }
  \right\|_{ 
    L^{ q ( 1 + \gamma ) ( 1 + \delta ) 
    }( \Omega; \R )
  }^{ 1 + \gamma }
  +
  \tfrac{ 2 \gamma }{ 1 + \gamma }
  \left\|
    ( X_u )^{ ( \nicefrac{ 1 }{ \gamma } - \gamma ) }
  \right\|_{ 
    L^{ q ( 1 + \delta ) }( \Omega; \R )
  }
 + 4
 \left\|  
   X_u 
 \right\|_{ L^{q(1+\delta)(c+1-\gamma)}( \Omega; \R )}^{ c + 1 - \gamma }
\right] 
  .
\end{align}}%
%
%
%
%
%
%
This implies that for all $\theta, \delta, q \in (0,\infty)$, $\gamma \in (0, 1]$
 it holds that
 {\allowdisplaybreaks 
\begin{equation} 
\begin{split} 
\label{l:lemma_conv_order:5}  
&\sup_{s,t \in [0,T]} \left[ 
\left| t-s \right|^{- \gamma \theta 
}
\left\|  f(X_t) - f(X_s) \right\|_{ L^{q}( \Omega; \R ) }
\right] 
\\ & \leq 
  c
  \cdot
 7^{ 
   \left[
     \frac{ 1 }{ q ( 1 + \delta ) } 
     - 1 
   \right]^+ 
 }
 \cdot
  \sup_{ s, t \in [0,T]
  } 
\Big[ 
  \left| t-s \right| ^{
    - \theta  
  }
  \left\| X_t - X_s  \right\|_{ 
    L^{ q \gamma \left( 1 + \frac{ 1 }{ \delta } \right) }( \Omega; \R )
  }
\Big]^{ \gamma }
\\
& 
 \cdot
\sup_{u \in [0,T]}
\left[  
  \tfrac{ 2 }{ 1 + \gamma }
  +
  4
  \left\| 
    \tfrac{ 1 }{ X_u }
  \right\|_{ L^{q(1+\gamma)(1+\delta)}( \Omega; \R )}^{1+\gamma}
  +
  \left[ 
    \tfrac{ 2 \gamma }{ 1 + \gamma }
    +
    4
  \right]
  \max\!\left\{
    1,
    \big\| 
      X_u
    \big\|_{ 
      L^{ 
        q ( 1 + \delta ) \left( \left(\frac{1}{\gamma} \vee (c+1)\right) - \gamma \right)
      }( \Omega; \R )
    }^{
      \left( \frac{ 1 }{ \gamma } \vee ( c + 1 ) \right) - \gamma
    }  
  \right\}
\right]
  .
\end{split} 
\end{equation} }%
This finishes the proof of Lemma~\ref{l:lemma_conv_order}.
\end{proof}

\subsection{Temporal H\"older regularity based on Cox-Ingersoll-Ross type processes}

\begin{lemma}
\label{l:reg_CIR}
Let $T, x \in [0,\infty)$, $( c_i )_{ i \in \{ 1, 2, \dots, 7 \} } \subseteq [0,\infty)$, $\mu \in \mathcal{L}^0([0,\infty); \R)$, $\sigma \in  \mathcal{L}^0([0,\infty); [0,\infty))$
satisfy 
$\mu(z) \leq c_1+ c_2 z$ and 
$|\sigma(z)|^2 \leq 2 (c_3 z + c_4 z^2)$ for all $z \in [0,\infty)$, 
let
$
  ( 
    \Omega, \mathcal{F}, \P,
       ( \mathcal{F}_t )_{ t \in [0,T] } 
  )
$
be a stochastic basis,
let
$
  W \colon [0,T] \times \Omega \to \R
$
be a standard $ ( \mathcal{F}_t )_{ t \in [0,T] } $-Brownian motion,
let
  $X \colon [0,T]\times \Omega\to [0,\infty)$
  be an adapted stochastic process with continuous sample paths satisfying 
  $\int_0^t |\mu(X_s)| \,ds < \infty$ $\P$-a.s.\
  and
  \begin{equation}
    X_t=x+\int_0^t \mu(X_s)\,ds+\int_0^t\sigma(X_s)\,dW_s
  \end{equation}
  $\P$-a.s.\
  for all $t\in [0,T]$.
  Then it holds for all $t\in [0,T]$, $p \in (0,\infty]$ that
\begin{equation}
\begin{split} 
 \label{l:reg_CIR:stateNEW} 
 \big\| X_t \big\|_{L^p( \Omega; \R )} 
\leq e^{ t(c_2 + c_4(p-1)^+ )}    \Big[x+ t \big(c_1+c_3 (p-1)^+  \big)  \Big].
\end{split} 
\end{equation}  
If, in addition to the above assumptions,  $|\mu(z)| \leq c_5+ c_6 z^{c_7}$  for all $z \in [0,\infty)$, then  it holds  for all $t\in [0,T]$, $p \in [1,\infty]$, $q \in [p\vee 2,\infty]$ that 
\begin{equation} 
\begin{split} 
&\left\|X_t - x   \right\|_{ 
    L^p( \Omega; \R ) }
%
%
%
\leq
c_5 t +  e^{(c_7 \vee 1) t(c_2 + c_4(q(c_7 \vee 1)-1) )}  
\Bigg(
t
c_6 
  \Big[x+ t \big(c_1+c_3 (pc_7-1)^+  \big)  \Big]^{c_7}  
\\ &  \label{l:reg_CIR:state3} 
+
\sqrt{q(q-1)}
\sqrt{t
c_3 
x
+
\tfrac{1}{2}t^2 c_3
  (c_1+c_3 (\tfrac{q}{2}-1)) 
+
 t c_4
   \Big[x+ t \big(c_1+c_3 (q-1)   \big)  \Big]^2
 }
\Bigg).
\end{split} 
\end{equation}
\end{lemma}

\begin{proof} [Proof of Lemma~\ref{l:reg_CIR}]

First of all, we
show that for all 
$ \eps \in (0,1) $, 
$ k \in \N_0 $,
$ p \in ( k, k + 1 ] $,
$ t \in [0,T] $
it holds that
\begin{equation} 
\begin{split} 
 \label{l:reg_CIR:state} 
 \big\| X_t + \eps \big\|_{L^p( \Omega; \R )} 
\leq e^{ t(c_2 + c_4(p-1)^+ )}    \Big[(x+ \eps)+ t \big(c_1+c_3 (p-1)^+ \big)  \Big]
\end{split} 
\end{equation}
 by induction on $k \in \N_0$.
For this observe that It\^{o}'s lemma implies that for all $\eps \in (0,1)$, $t\in [0,T]$ it holds that
\begin{equation} 
\begin{split} 
 \label{l:reg_CIR:2} 
 (X_t+ \eps) e^{-c_2 t} &= (x+ \eps) + \smallint_0^t (-c_2) (X_s+ \eps) e^{-c_2s} + \mu(X_s) e^{-c_2 s} \, ds
+ \smallint_0^t \sigma(X_s) e^{-c_2 s} \, dW_s
\\ &
\leq (x+ \eps) + \smallint_0^t (-c_2) (X_s+ \eps) e^{-c_2s} + (c_1 + c_2 X_s) e^{-c_2 s} \, ds
+ \smallint_0^t \sigma(X_s) e^{-c_2 s} \, dW_s
\\ &
\leq (x+ \eps) + \smallint_0^t  c_1 e^{-c_2s} \, ds
+ \smallint_0^t \sigma(X_s) e^{-c_2 s} \, dW_s
\\ &
\leq (x+ \eps) +   c_1 t 
+ \smallint_0^t \sigma(X_s) e^{-c_2 s} \, dW_s.
\end{split} 
\end{equation}  
This 
implies that
 for all $\eps \in (0,1)$, $t\in [0,T]$ it holds that
\begin{equation} 
\begin{split} 
 \label{l:reg_CIR:3} 
 \E\big[ X_t+ \eps \big] e^{-c_2 t} 
\leq (x+ \eps) + c_1 t.
\end{split} 
\end{equation}  
Note that inequality \eqref{l:reg_CIR:state} in the case 
$ p \in (0,1] $ 
follows immediately from \eqref{l:reg_CIR:3}. 
This proves the base case $k=0$ of \eqref{l:reg_CIR:state}. 
For the induction step we assume that \eqref{l:reg_CIR:state} holds for some $k \in \N_0$. Then note with It\^{o}'s lemma that for all $\eps \in (0,1)$, $t\in [0,T]$, $p \in (k+1, k+2]$ it holds that
\begin{align}
 \label{l:reg_CIR:4}  \nonumber
  &(X_t + \eps)^p e^{-pt(c_2 + c_4 (p-1))}
\\ & \nonumber
= (x + \eps)^p + \smallint_0^t -p(c_2 + c_4 (p-1)) (X_s + \eps)^p e^{-ps(c_2 + c_4 (p-1))} + p \mu(X_s) (X_s + \eps)^{p-1} e^{-ps(c_2 + c_4 (p-1))} 
\\ \nonumber & \quad + \tfrac{p(p-1)}{2} |\sigma(X_s)|^2 (X_s + \eps)^{p-2} e^{-ps(c_2 + c_4 (p-1))}
 \, ds
+ p \smallint_0^t \sigma(X_s) (X_s + \eps)^{p-1} e^{-ps(c_2 + c_4 (p-1))} \, dW_s
\\ &
\leq
 (x + \eps)^p + \smallint_0^t e^{-ps(c_2 + c_4 (p-1))} \Big[ -p(c_2 + c_4 (p-1))  (X_s + \eps)^p 
 + p (X_s + \eps)^{p-1} (c_1+ c_2 X_s) 
\\ \nonumber& \quad +p(p-1) \big (c_3 X_s + c_4 (X_s)^2 \big)  (X_s + \eps)^{p-2} \Big]
 \, ds
+ p \smallint_0^t \sigma(X_s) (X_s + \eps)^{p-1} e^{-ps(c_2 + c_4 (p-1))} \, dW_s
\\ & \nonumber
\leq
 (x + \eps)^p + \smallint_0^t e^{-ps(c_2 + c_4 (p-1))}  p \big[ c_1  + c_3 (p-1) \big] (X_s + \eps)^{p-1} 
 \, ds
 \\ \nonumber & \quad 
+ p \smallint_0^t \sigma(X_s) (X_s + \eps)^{p-1} e^{-ps(c_2 + c_4 (p-1))} \, dW_s.
\end{align} 
This
implies that
 for all $\eps \in (0,1)$, $t\in [0,T]$, $p \in (k+1, k+2]$ it holds that
  \begin{align}
 \label{l:reg_CIR:5}  \nonumber
 &\E\big[ (X_t + \eps)^p \big] e^{-pt(c_2 + c_4 (p-1))}
\leq (x + \eps)^p + \smallint_0^t  e^{-ps(c_2 + c_4 (p-1))}  p \big[ c_1  + c_3 (p-1) \big]  \E\big[ (X_s + \eps)^{p-1} \big]
 \, ds
 \\ \nonumber
&
\leq (x + \eps)^p 
+ \smallint_0^t e^{-ps(c_2 + c_4 (p-1))}  p \big[ c_1  + c_3 (p-1) \big] 
 e^{(p-1) s(c_2 + c_4 (p-2)^+)} 
 \\
 &
 \quad
 \cdot
 \Big[(x + \eps)+ s \big(c_1+c_3 (p-2)^+)  \big)  \Big]^{p-1} 
 \, ds
\\ \nonumber
&
\leq (x + \eps)^p + p \big[ c_1  + c_3 (p-1) \big]   \smallint_0^t  
  \Big[(x + \eps)+ s \big(c_1+c_3 (p-1)  \big)  \Big]^{p-1}   \, ds
\\ \nonumber
&
= (x + \eps)^p +  
 \Big[  
\Big((x + \eps)+t \big(c_1+c_3 (p-1)  \big)  \Big)^{p} - (x + \eps)^p \Big]
=
\Big((x + \eps)+t \big(c_1+c_3 (p-1)  \big)  \Big)^{p}.
  \end{align} 
Induction hence shows that inequality \eqref{l:reg_CIR:state} holds for all $k \in \N_0$. Taking the infinum over all $\eps \in (0,1)$ on the right-hand side of inequality \eqref{l:reg_CIR:state} shows \eqref{l:reg_CIR:stateNEW}  in the case $p \in (0,\infty)$. The case $p=\infty$ follows from letting $(0,\infty) \ni p \to \infty$.
Now we assume in addition that for all $z \in [0,\infty)$ it holds that $|\mu(z)| \leq c_5+c_6 z^{c_7}$. 
Note that the Burkholder-Davis-Gundy-type inequality in Lemma 7.2 in Da Prato \& Zabczyk~\cite{dz92},
Jensen's inequality and \eqref{l:reg_CIR:stateNEW}  ensure that for all $t\in [0,T]$, $p \in [1,\infty)$, $q \in [p\vee 2,\infty)$ it holds that
  \begin{equation} 
\begin{split} 
 \label{l:reg_CIR:6} 
&\left\|X_t - x   \right\|_{ 
    L^p( \Omega; \R ) }
\leq
\left\|
\smallint_0^t |\mu(X_s)| \, ds 
 \right\|_{ 
    L^p( \Omega; \R ) }
+
\left\|
\smallint_0^t \sigma(X_s) dW_s
 \right\|_{ 
    L^p( \Omega; \R ) }
\\ &
\leq
\left\|
\smallint_0^t |\mu(X_s)| \, ds 
 \right\|_{ 
    L^p( \Omega; \R ) }
+
\left\|
\smallint_0^t \sigma(X_s) dW_s
 \right\|_{ 
    L^q( \Omega; \R ) }
\\ &
\leq
\smallint_0^tc_5  +  c_6 \left\|(X_s)^{c_7} \right\|_{ 
    L^p( \Omega; \R ) }  ds 
+
\sqrt{\tfrac{q(q-1)}{2}}
\sqrt{
\smallint_0^t \left\||\sigma(X_s)|^2  \right\|_{ 
    L^{\frac{q}{2}}( \Omega; \R ) }  ds }
\\ &
\leq
\smallint_0^t c_5  + c_6\left\|(X_s)^{c_7} \right\|_{ 
    L^p( \Omega; \R ) }  ds 
+
\sqrt{\tfrac{q(q-1)}{2}}
\sqrt{2
\smallint_0^t c_3 \left\|X_s \right\|_{ 
    L^{\frac{q}{2}}( \Omega; \R ) } 
+
c_4\left\| X_s \right\|_{ 
    L^{q}( \Omega; \R ) }^2
 ds }
\\ &
\leq
c_5 t + c_6 
\smallint_0^t 
 e^{c_7 s(c_2 + c_4   (pc_7-1)^+  )}    \Big[x+ s \big(c_1+c_3 (pc_7-1)^+  \big)  \Big]^{c_7}  ds 
+
\sqrt{q(q-1)}
\\ & \ \cdot
\sqrt{
\smallint_0^t 
c_3
e^{s(c_2 + c_4(q-1))}    \Big[x+ s \big(c_1+c_3 (\tfrac{q}{2}-1)   \big)  \Big]
+
c_4
e^{2 s (c_2 + c_4 (q-1) )}    \Big[x+ s \big(c_1+c_3 (q-1)   \big)  \Big]^2
 ds }.
\end{split} 
\end{equation} 
This shows  that for all $t\in [0,T]$, $p \in [1,\infty]$, $q \in [p\vee 2,\infty)$ it holds that
\begin{equation} 
\begin{split}  \label{l:reg_CIR:6b} 
&\left\|X_t - x   \right\|_{ 
    L^p( \Omega; \R ) }
\leq
c_5 t + c_6 
\smallint_0^t 
 e^{c_7 s(c_2 + c_4   (pc_7-1)^+  )}    \Big[x+ s \big(c_1+c_3 (pc_7-1)^+  \big)  \Big]^{c_7} \, ds 
+
\sqrt{q(q-1)}
\\ &   \ \cdot
\sqrt{
\smallint_0^t 
c_3
e^{s(c_2 + c_4(q-1))}    \Big[x+ s \big(c_1+c_3 (\tfrac{q}{2}-1)   \big)  \Big]
+
c_4
e^{2 s (c_2 + c_4 (q-1) )}    \Big[x+ s \big(c_1+c_3 (q-1)   \big)  \Big]^2
\, ds }.
\end{split} 
\end{equation}
Finally,  \eqref{l:reg_CIR:6b}   implies that for all $t\in [0,T]$, $p \in [1,\infty]$, $q \in [p\vee 2,\infty]$ it holds that
  \begin{equation} 
\begin{split} 
 \label{l:reg_CIR:7} 
&\left\|X_t - x   \right\|_{ 
    L^p( \Omega; \R ) }
\leq
c_5 t +  e^{(c_7 \vee 1) t(c_2 + c_4(q(c_7 \vee 1)-1) )}  
\Bigg(
c_6 
\smallint_0^t 
  \Big[x+ s \big(c_1+c_3 (pc_7-1)^+  \big)  \Big]^{c_7}  ds 
\\ & \quad 
+
\sqrt{q(q-1)}
\sqrt{
\smallint_0^t 
c_3
\Big[x+ s \big(c_1+c_3 (\tfrac{q}{2}-1)   \big)  \Big]
+
c_4
   \Big[x+ s \big(c_1+c_3 (q-1)   \big)  \Big]^2
 ds }
\Bigg)
\\ &
\leq
c_5 t +  e^{(c_7 \vee 1) t(c_2 + c_4(q(c_7 \vee 1)-1) )}  
\Bigg(
t
c_6 
  \Big[x+ t \big(c_1+c_3 (pc_7-1)^+  \big)  \Big]^{c_7}  
\\ & \quad 
+
\sqrt{q(q-1)}
\sqrt{t
c_3 
x
+
\tfrac{1}{2}t^2 c_3
  (c_1+c_3 (\tfrac{q}{2}-1)) 
+
 t c_4
   \Big[x+ t \big(c_1+c_3 (q-1)   \big)  \Big]^2
 }
\Bigg).
\end{split} 
\end{equation}  
This 
finishes the proof of Lemma~\ref{l:reg_CIR}.
\end{proof}

\begin{lemma}
\label{l:cir_min}
Let $T \in [0,\infty)$,  $( c_i )_{ i \in \{ 1, 2, 3 \} } \subseteq [0,\infty)$, $ \sigma \in \mathcal{L}^0([0,\infty); [0,\infty))$
satisfy 
$\sigma(z) \geq \tfrac{c_1 \sqrt{z} } {1+c_3 z^{c_2}}$ for all $z \in [0,\infty)$, 
let
$
  ( 
    \Omega, \mathcal{F}, \P
  )
$
be a probability space and 
let
  $X, Y \colon  \Omega\to [0,\infty)$
  be random variables.
  Then it holds for all  $p \in (0,\infty]$, $r \in [0,\infty]$  that
\begin{align}
 \label{l:cir_min:state} 
 &  \left\| \int_Y^X \tfrac{1}{\sigma(z)} \, dz \right\|_{L^p( \Omega; \R )} 
 \\ & \nonumber
\leq
\tfrac{2}{c_1}  \left\| 1+ c_3 X^{c_2} + c_3 Y^{c_2}  \right\|_{L^{p(1+\nicefrac{1}{r})}( \Omega; \R )}
\min\!\left\{
\left\|  | Y-X |^{\nicefrac{1}{2}}  \right\|_{L^{p(1+r)} ( \Omega; \R )},
\left\|  \tfrac{ \left| Y-X \right| }{ \sqrt{Y} + \sqrt{X}}\right\|_{L^{p(1+r)}( \Omega; \R )} 
\right\}.
\end{align}
\end{lemma}

\begin{proof} [Proof of Lemma~\ref{l:cir_min}]

First of all, observe that for all $v,w \in [0,\infty)$ it holds that
\begin{equation} 
\begin{split} 
 \label{l:cir_min:1} 
 \left| \int_w^v \tfrac{1}{\sigma(z)} \, dz \right|
 &\leq
 \tfrac{1}{c_1} \left| \int_w^v \tfrac{1}{\sqrt{z}} \, |1+ c_3 z^{c_2}| \, dz \right| 
  \leq
 \tfrac{1}{c_1} \sup_{u \in [v,w] \cup [w,v]} |1+ c_3 u^{c_2}| \left| \int_w^v \tfrac{1}{\sqrt{z}} \, dz \right| 
 \\ &
 \leq
 \tfrac{2}{c_1} \left|1+ c_3 v^{c_2} + c_3 w^{c_2} \right| \left| \sqrt{v} - \sqrt{w} \right| 
 =
  \tfrac{2}{c_1} \tfrac{ \left| v-w \right| }{ \sqrt{v} + \sqrt{w}}\left|1+ c_3 v^{c_2} + c_3 w^{c_2} \right|.
\end{split} 
\end{equation}  
From \eqref{l:cir_min:1} and from the estimate $\left| v-w \right| \leq  \left(\sqrt{v}+\sqrt{w}\right)^2$ for all $v,w \in [0,\infty)$ we get that for all $v,w \in [0,\infty)$ it holds that
\begin{equation} 
\begin{split} 
 \label{l:cir_min:2} 
 &\left| \int_w^v \tfrac{1}{\sigma(z)} \, dz \right|^2 
\leq
  \tfrac{4}{(c_1)^2}  \left| v-w \right|   \left|1+ c_3 v^{c_2} + c_3 w^{c_2} \right|^2.
\end{split} 
\end{equation}  
Inequality \eqref{l:cir_min:1} and H\"older's inequality ensure that for all $p \in (0,\infty]$, $q \in [1,\infty]$ it holds that
\begin{equation} 
\begin{split} 
 \label{l:cir_min:3} 
 \left \|\int_X^{Y} \tfrac{1}{\sigma(z)} \, dz \right\|_{L^p( \Omega; \R )} 
& 
 \leq
  \tfrac{2}{c_1} \left\| \tfrac{  Y-X  }{ \sqrt{Y} + \sqrt{X}}\left|1+ c_3 X^{c_2} + c_3 Y^{c_2}\right| \right\|_{L^p( \Omega; \R )}  
\\ &
 \leq
  \tfrac{2}{c_1} \left\|\tfrac{  Y-X  }{ \sqrt{Y} + \sqrt{X}} \right\|_{L^{pq}( \Omega; \R )}  
  \left\|1+ c_3 X^{c_2} + c_3 Y^{c_2}  \right\|_{L^{\nicefrac{p}{(1-\nicefrac{1}{q})}}( \Omega; \R )}.
\end{split} 
\end{equation} 
Similarly, inequality \eqref{l:cir_min:2} and H\"older's inequality imply that for all $p \in (0,\infty]$, $q \in [1,\infty]$ it holds that
\begin{align}
 \label{l:cir_min:4}  \nonumber
  \left\| \int_Y^X \tfrac{1}{\sigma(z)} \, dz \right\|_{L^p( \Omega; \R )} 
 & =  \left\| \left| \int_Y^X \tfrac{1}{\sigma(z)} \, dz \right|^2 \right\|_{L^{\nicefrac{p}{2}} ( \Omega; \R )}^{\nicefrac{1}{2}}
 \leq
\left\| \tfrac{4}{(c_1)^2}  \left| Y-X \right|   \left|1+ c_3 X^{c_2} + c_3 Y^{c_2} \right|^2 \right\|_{L^{\nicefrac{p}{2}} ( \Omega; \R )}^{\nicefrac{1}{2}}
\\ 
& 
 \leq
\tfrac{2}{c_1}  
\left\|  Y-X  \right\|_{L^{\nicefrac{pq}{2}} ( \Omega; \R )}^{\nicefrac{1}{2}}  
\left\| \left|1+ c_3 X^{c_2} + c_3 Y^{c_2} \right|^2 \right\|_{L^{\nicefrac{p}{(2-\nicefrac{2}{q})}}( \Omega; \R )}^{\nicefrac{1}{2}}
\\ \nonumber
& 
=
\tfrac{2}{c_1}  
\left\|  Y-X \right\|_{L^{\nicefrac{pq}{2}} ( \Omega; \R )}^{\nicefrac{1}{2}}  
\left\| 1+ c_3 X^{c_2} + c_3 Y^{c_2}  \right\|_{L^{\nicefrac{p}{(1-\nicefrac{1}{q})}}( \Omega; \R )}.
\end{align}
This and \eqref{l:cir_min:3} show that for all   $p \in (0,\infty]$, $q \in [1,\infty]$ it holds that
\begin{equation} 
\begin{split} 
 \label{l:cir_min:4bbb} 
 &  \left\| \int_Y^X \tfrac{1}{\sigma(z)} \, dz \right\|_{L^p( \Omega; \R )} 
 \\ &
\leq
\tfrac{2}{c_1}  \left\| 1+ c_3 X^{c_2} + c_3 Y^{c_2}  \right\|_{L^{\nicefrac{p}{(1-\nicefrac{1}{q})}}( \Omega; \R )}
\min \left\{
\left\|  Y-X  \right\|_{L^{\nicefrac{pq}{2}} ( \Omega; \R )}^{\nicefrac{1}{2}},
\left\|  \tfrac{  Y-X  }{ \sqrt{Y} + \sqrt{X}}\right\|_{L^{pq}( \Omega; \R )} 
\right\}.
\end{split} 
\end{equation} 
This 
finishes the proof of Lemma~\ref{l:cir_min}.
\end{proof}

\begin{theorem}
\label{l:cir_hr}
Let $T \in (0,\infty)$,  $( c_i )_{ i \in \{ 1, 2, \dots, 10 \} } \subseteq [0,\infty)$, $\mu \in \mathcal{L}^0([0,\infty); \R)$,
$ \sigma \in \mathcal{L}^{0}([0,\infty), [0,\infty))$
satisfy 
$ c_8 > 0 $,
$\mu(z) \leq c_1+ c_2 z$,
$|\sigma(z)|^2 \leq 2 (c_3 z + c_4 z^2)$,
$|\mu(z)| \leq c_5+ c_6 z^{c_7}$,
$\sigma(z) \geq \tfrac{c_8 \sqrt{z} } {1+c_{10} z^{c_9}}$ for all $z \in [0,\infty)$,
let
$
  ( 
    \Omega, \mathcal{F}, \P,
       ( \mathcal{F}_t )_{ t \in [0,T] } 
  )
$
be a stochastic basis,
let
$
  W \colon [0,T] \times \Omega \to \R
$
be a standard $ ( \mathcal{F}_t )_{ t \in [0,T] } $-Brownian motion,
let
  $X^x \colon [0,T]\times \Omega\to [0,\infty)$, $x \in [0,\infty)$,
  be a family of adapted stochastic processes with continuous sample paths satisfying 
  \begin{equation} \label{l:cir_hr:sde} 
    X_t^x=x+\int_0^t \mu(X_s^x)\,ds+\int_0^t\sigma(X_s^x)\,dW_s
  \end{equation}
  $\P$-a.s.\
  for all $(t,x)\in[0,T]\times [0,\infty)$
  and assume that for every compact set $K \subseteq [0,\infty)$ it holds that 
  \begin{align} \label{l:cir_hr:ass1} 
\sup_{
  \substack{
    x,y \in K, \,
   \\ 
    x \neq y
  }
}
\tfrac{(x-y)(\mu(x)-\mu(y))}{|x-y|^2} + 
\sup_{
  \substack{
    x,y \in K, \,
  \\ 
    x \neq y
  }
}
\tfrac{(\sigma(x)-\sigma(y))^2}{|x-y|} < \infty.
  \end{align}
Then it holds for all 
$ s, t \in [0,T] $, 
$ x \in [0,\infty) $, 
$ p \in (0,\infty) $, 
$ q \in [1,\infty) $ 
with $ p q \geq 2 $ that
\begin{equation} 
\begin{split} 
 \label{l:cir_hr:state} 
  &
  \left\| 
    \int_{ X_s^x }^{ X_t^x } 
    \tfrac{ 1 }{ \sigma(z) } \, dz 
  \right\|_{
    L^p( \Omega; \R )
  } 
\leq
  \sqrt{ | t - s | } 
\\ & 
  \cdot 
  \frac{ 2 }{ c_8 } 
  \Bigg[ 
    1 + 
    e^{ 
      ( t \wedge s ) 
      ( 
        c_2 + 
        c_4 
        [ 
          p 
          ( c_9 + \max\{ c_7 , \nicefrac{ 1 }{ 2 } \} ) - 1 
        ]^+ 
      )
    }    
    \Big[
    x+ (t \wedge s) \big(c_1+c_3 (p(c_9 + (c_7 \vee \tfrac{1}{2}))-1)^+  \big)  \Big] \Bigg]^{c_9 + (c_7 \vee \nicefrac{1}{2})}
\\ &
\cdot
\Bigg[
  1 + c_{ 10 }  
  + c_{ 10 } \,
  e^{
    | t - s | \, 
    c_9 
    \big( 
      c_2 + c_4 
      \left[
        \frac{ c_9 p q }{ q - 1 } - 1 
      \right]^+ 
    \big)
  }   
    \left[
      1 + | t - s | 
      \left(
        c_1 + 
        c_3 
        \left[ \tfrac{ c_9 p q }{ q - 1 } - 1 \right]^+ 
      \right)  
    \right]^{c_9} 
  \Bigg]
\\ & \cdot
  \max\Bigg\{ 
    1,
    c_5 + 
    e^{ 
      ( c_7 \vee 1 ) \left| t - s \right|
      \left[
        c_2 + c_4
        \left(pq(c_7 \vee 1)-1\right) 
      \right]
    }  
\Bigg(
c_6 
  \Big[1+ |t-s| \big(c_1+c_3 \left( pq c_7-1\right)^+  \big)  \Big]^{c_7}  
\\ &   
\quad+
\sqrt{pq(pq-1)}
 \sqrt{c_3
+
   \tfrac{c_1c_3}{2}+\tfrac{c_3^2}{2} \left( \tfrac{pq}{2} -1\right)
+
  c_4
   \Big[1+ \sqrt{|t-s|} \big(c_1+c_3 \left( pq-1 \right)   \big)  \Big]^2
 }
\Bigg)
\Bigg\} 
\end{split} 
\end{equation} 
and
it holds for all $p, \eps \in (0,\infty)$ that
\begin{equation} 
\begin{split} 
 \label{l:cir_hr:state2} 
 &
  \sup_{
    x \in [0,\infty)
  }
  \left[
    \frac{ 1 }{
      1 + 
      x^{ 
        c_9 + \max\{ c_7 , \nicefrac{ 1 }{ 2 } \} 
      }
    }
    \left\|
      \sup_{ 
       \substack{ 
         s, t \in [0,T] , \,
         s \neq t 
       }
     }
     \left[
       \left| 
         t - s 
       \right|^{ ( \eps - \nicefrac{1}{2} ) } 
     \left|
       \int_{
         X_s^x
       }^{
         X_t^x
       } 
       \tfrac{ 1 }{ \sigma(z) } 
       \, dz 
     \right|
    \right]
   \right\|_{
     L^p( \Omega; \R )
   } 
   \right]
  < \infty
  .
\end{split} 
\end{equation}
\end{theorem}

\begin{proof} [Proof of Theorem~\ref{l:cir_hr}]
%
%
%
%
From Theorem V.40.1 in~\cite{RogersWilliams2000b}
together with \eqref{l:cir_hr:ass1}
(the proof of
Theorem V.40.1 in~\cite{RogersWilliams2000b} only uses global one-sided Lipschitz continuity of $\mu$)
and from a stopping argument,
we obtain that the SDE with coefficients $(\mu, \sigma)$ is pathwise unique.
Then we fix $t_1 \in [0,T)$ and define $W^{t_1} \colon [0,T-t_1] \times \Omega \to \R$ by $W^{t_1}_t := W_{t_1+t} - W_{t_1}$ for all $t \in [0,T-t_1]$. 
Theorem 3.4 in Cox et.~al~\cite{CoxHutzenthalerJentzen2013} shows that there exists a measurable mapping 
$Y\colon[0,T-t_1]\times [0,\infty) \times \Omega\to[0,\infty)$
such that for all $x \in [0,\infty)$ it holds that $Y^x$ is an adapted stochastic process
with continuous sample paths satisfying 
\begin{equation}
  Y_t^x=x+\int_0^t \mu(Y_s^x)\,ds+\int_0^t\sigma(Y_s^x)\,dW_s^{t_1}
\end{equation}
$\P$-a.s.\
for all $t\in[0,T-t_1]$.
Pathwise uniqueness ensures that for all $x \in [0,\infty)$ it holds that
\begin{equation} 
\begin{split} 
 \label{l:cir_hr:2} 
\P \big[ \big(X_{t_1+t}^x\big)_{t \in [0,T-t_1]} \in \cdot \, \big]
=
\P \big[ \big(Y_{t}^{X_{t_1}^x}\big)_{t \in [0,T-t_1]} \in \cdot \,  \big].
\end{split} 
\end{equation}
 Lemma~\ref{l:reg_CIR} shows that for all $t \in [0,T-t_1]$, $x \in [0,t]$,
 $p, q \in (0,\infty)$ with $pq\geq 2$ it holds that
\begin{equation} 
\begin{split} 
 \label{l:cir_hr:4} 
&\left\|  Y_t^x - x  \right\|_{L^{\nicefrac{pq}{2}} ( \Omega; \R )}
\leq\left\|  Y_t^x - x  \right\|_{L^{pq} ( \Omega; \R )}
\\
&
\leq
c_5 t +  e^{(c_7 \vee 1) t\left(c_2 + c_4\left(pq (c_7 \vee 1)-1\right) \right)}  
\Bigg(
t
c_6 
  \Big[x+ t \big(c_1+c_3 \left( pq  c_7-1\right)^+  \big)  \Big]^{c_7}  
\\ &  \quad
+
\sqrt{pq(pq-1)}
\sqrt{
tc_3 
x
+
\tfrac{1}{2}t^2 c_3
  \left(c_1+c_3 \left( \tfrac{pq}{2} -1\right) \right) 
+
  tc_4
   \Big[x+ t \big(c_1+c_3 \left( pq -1 \right)   \big)  \Big]^2
 }
\Bigg)
\\ &
\leq
c_5 t +  e^{(c_7 \vee 1) t\left(c_2 + c_4\left(pq (c_7 \vee 1)-1\right) \right)}  
\Bigg(
t
c_6 
  \Big[x+ t \big(c_1+c_3 \left( pq  c_7-1\right)^+  \big)  \Big]^{c_7}  
\\ &  \quad
+
\sqrt{pq(pq-1)}
\sqrt{
c_3 
t^2
+
\tfrac{1}{2}t^2 c_3
  \left(c_1+c_3 \left(\tfrac{pq}{2} -1\right) \right) 
+
  tc_4
   \Big[\sqrt{t}\sqrt{x}+ t \big(c_1+c_3 \left( pq-1 \right)   \big)  \Big]^2
 }
\Bigg)
\\ &
=
t\Bigg[ c_5  +  e^{(c_7 \vee 1) t\left(c_2 + c_4\left(pq(c_7 \vee 1)-1\right) \right)}  
\Bigg(
c_6 
  \Big[ x+  t\Big(c_1+c_3 \left( pq c_7-1\right)^+   \Big)  \Big]^{c_7}  
\\ &  \quad
+
\sqrt{pq(pq-1)}
 \sqrt{
c_3 
+
 \tfrac{c_3}{2}
  \left(c_1+c_3 \left( \tfrac{pq}{2} -1\right) \right) 
+
  c_4
   \Big[\sqrt{x}+  \sqrt{t} \big(c_1+c_3 \left(  pq-1 \right)   \big)  \Big]^2
 }
\Bigg) \Bigg].
\end{split} 
\end{equation}
Similarly, Lemma~\ref{l:reg_CIR} shows that for all $t \in [0,T-t_1]$, $x \in [t,\infty)$, $p, q \in (0,\infty)$ with $pq\geq 2$ it holds that
\begin{equation} 
\begin{split} 
  \label{l:cir_hr:5} 
& \tfrac{1}{\sqrt{x}} \left\|  Y_t^x - x  \right\|_{L^{pq} ( \Omega; \R )}
\\ &
\leq
c_5 \tfrac{t}{\sqrt{x}} +  e^{(c_7 \vee 1) t\left(c_2 + c_4\left(pq(c_7 \vee 1)-1\right) \right)}  
\Bigg(
\tfrac{t}{\sqrt{x}}
c_6 
  \Big[x+ t \big(c_1+c_3 \left( pq c_7-1\right)^+  \big)  \Big]^{c_7}  
\\ &  \qquad
+
\sqrt{pq(pq-1)}
\sqrt{
tc_3 
+
\tfrac{t^2c_3}{2x}
  \left(c_1+c_3 \left( \tfrac{pq}{2} -1\right) \right) 
+
\tfrac{  tc_4 }{x}
   \Big[x+ t \big(c_1+c_3 \left( pq-1 \right)   \big)  \Big]^2
 }
\Bigg)
\\ &
\leq
c_5 \sqrt{t} +  e^{(c_7 \vee 1) t\left(c_2 + c_4\left(pq(c_7 \vee 1)-1\right) \right)}  
\Bigg(
\sqrt{t}
c_6 
  \Big[x+ t \big(c_1+c_3 \left( pq c_7-1\right)^+  \big)  \Big]^{c_7}  
\\ &  \qquad
+
\sqrt{pq(pq-1)}
\sqrt{
tc_3 
+
 \tfrac{tc_3}{2}
  \left(c_1+c_3 \left( \tfrac{pq}{2} -1\right) \right) 
+
  tc_4
   \Big[\sqrt{x}+ \sqrt{t} \big(c_1+c_3 \left( pq-1 \right)   \big)  \Big]^2
 }
\Bigg)
\\ &
=
 \sqrt{t} \Bigg[
c_5 +  e^{(c_7 \vee 1) t\left(c_2 + c_4\left(pq(c_7 \vee 1)-1\right) \right)}  
\Bigg(
c_6 
  \Big[x+ t \big(c_1+c_3 \left( pq c_7-1\right)^+  \big)  \Big]^{c_7}  
\\ &  \qquad
+
\sqrt{pq(pq-1)}
 \sqrt{
c_3 
+
 \tfrac{c_3}{2}
  \left(c_1+c_3 \left( \tfrac{pq}{2} -1\right) \right) 
+
  c_4
   \Big[\sqrt{x}+ \sqrt{t} \big(c_1+c_3 \left( pq-1 \right)   \big)  \Big]^2
 }
\Bigg)
\Bigg].
\end{split} 
\end{equation}
Next we define a function 
$
  \phi\colon [0,\infty) \to [0,\infty)
$ 
through 
$
  \phi(y) := \int_0^y \tfrac{ 1 }{ \sigma(z) } \, dz
$
for all $ y \in [0,\infty) $. 
The function $ \phi $ is well defined due to
the estimate that
$\sigma(z) \geq \tfrac{c_8 \sqrt{z} } {1+c_{10} z^{c_9}}$ for all $z \in [0,\infty)$.
In the next step we observe that
Lemma~\ref{l:cir_min},  
Lemma~\ref{l:reg_CIR}, 
\eqref{l:cir_hr:4} and 
\eqref{l:cir_hr:5} imply 
that for all 
$ t \in [0, T - t_1] $, 
$ x \in [0, \infty) $, 
$ p \in (0,\infty) $,
$ q \in [1,\infty) $ with $ p q \geq 2 $ 
it holds that
\begin{equation} 
\begin{split} 
 \label{l:cir_hr:6}  
 &\|\phi(Y_t^x) - \phi(x)\|_{L^p( \Omega; \R )} 
 \\ &
\leq
\tfrac{2}{c_8}  \left(1 + c_{10}  x^{c_9}+ c_{10} \left\| Y_t^x \right\|_{L^{\nicefrac{c_9pq}{q-1}}( \Omega; \R )}^{c_9}  \right)
\min \left\{
\left\|  Y_t^x - x  \right\|_{L^{\nicefrac{pq}{2}} ( \Omega; \R )}^{\nicefrac{1}{2}},
 \tfrac{1}{\sqrt{x}} \left\|  Y_t^x - x  \right\|_{L^{pq}( \Omega; \R )} 
\right\}
 \\ &
\leq
\tfrac{2}{c_8}   \left(1 + c_{10}  x^{c_9}
+ c_{10}  
e^{ t c_9 \left(c_2 + c_4 \left(\frac{c_9pq}{q-1}-1 \right)^+ \right)}    \left[x+ t \left(c_1+c_3 \left(\tfrac{c_9pq}{q-1}-1\right)^+  \right)  \right]^{c_9} 
\right)
\\ & \quad \cdot
\min \left\{
\left\|  Y_t^x - x  \right\|_{L^{\nicefrac{pq}{2}} ( \Omega; \R )}^{\nicefrac{1}{2}},
 \tfrac{1}{\sqrt{x}} \left\|  Y_t^x - x  \right\|_{L^{pq}( \Omega; \R )} 
\right\}
%
%
%
%
 \\ &
\leq
\tfrac{2}{c_8}  \sqrt{t} 
\left(1 + c_{10}  x^{c_9}
+ c_{10}  
e^{ t c_9 \left(c_2 + c_4 \left(\frac{c_9pq}{q-1}-1 \right)^+ \right)}    \left[x+ t \left(c_1+c_3 \left(\tfrac{c_9pq}{q-1}-1\right)^+  \right)  \right]^{c_9} 
\right)
\\ & \quad \cdot
 \Bigg\{
\Bigg[
c_5 +  e^{(c_7 \vee 1) t\left(c_2 + c_4\left(pq(c_7 \vee 1)-1\right) \right)}  
\Bigg(
c_6 
  \Big[x+ t \big(c_1+c_3 \left( pq c_7-1\right)^+  \big)  \Big]^{c_7}  
\\ &   \quad 
+
\sqrt{pq(pq-1)}
 \sqrt{
c_3 
+
 \tfrac{c_3}{2}
  \left(c_1+c_3 \left( \tfrac{pq}{2} -1\right) \right) 
+
  c_4
   \Big[\sqrt{x}+ \sqrt{t} \big(c_1+c_3 \left( pq-1 \right)   \big)  \Big]^2
 }
\Bigg)
\Bigg] \vee 1
\Bigg\}.
\end{split} 
\end{equation}
Then \eqref{l:cir_hr:2}  and \eqref{l:cir_hr:6}  yield that for all $h \in [0,T-t_1]$, $x \in [0, \infty)$, $p \in (0,\infty)$, $q \in [1,\infty) $ 
with $pq\geq 2$ it holds that
\begin{align}
 \label{l:cir_hr:7} 
 &\|\phi(X_{t_1+h}^x) - \phi(X_{t_1}^x)\|_{L^p( \Omega; \R )} 
= \left( \E \left[ \E \left[ \left( \phi(Y_{h}^{X_{t_1}^x}) - \phi(X_{t_1}^x) \right)^p \big| X_{t_1}^x \right] \right] \right)^{\nicefrac{1}{p}}
 \\ & \nonumber
\leq
\Bigg\|
\tfrac{2}{c_8}  \sqrt{h} 
\left(1 + c_{10}  (X_{t_1}^x)^{c_9}
+ c_{10}  
e^{h c_9 \left(c_2 + c_4 \left(\frac{c_9pq}{q-1}-1 \right)^+ \right)}    \left[X_{t_1}^x+ h \left(c_1+c_3 \left(\tfrac{c_9pq}{q-1}-1\right)^+  \right)  \right]^{c_9} 
\right)
\\ & \quad  \cdot \nonumber
\Bigg[
c_5 +  e^{(c_7 \vee 1) h\left(c_2 + c_4\left(pq(c_7 \vee 1)-1\right) \right)}  
\Bigg(
c_6 
  \Big[X_{t_1}^x+ h \big(c_1+c_3 \left( pq c_7-1\right)^+  \big)  \Big]^{c_7}  
\\ &  \qquad \nonumber
+
\sqrt{pq(pq-1)}
 \sqrt{c_3
+
  \tfrac{ c_1c_3}{2} + \tfrac{c_3^2}{2}  \! \left( \tfrac{pq}{2} -1\right)
+
  c_4
   \Big[\sqrt{X_{t_1}^x}+ \sqrt{h} \big(c_1+c_3 pq-c_3    \big)  \Big]^2
 }
\Bigg)
\Bigg] \! \vee \! 1
\Bigg\|_{L^p( \Omega; \R )} 
\end{align}
and
\begin{equation}  \begin{split}
 &\|\phi(X_{t_1+h}^x) - \phi(X_{t_1}^x)\|_{L^p( \Omega; \R )} 
\leq
\Bigg\|
\tfrac{2}{c_8}  \sqrt{h} \left(1 + X_{t_1}^x \right)^{c_9 + (c_7 \vee \frac{1}{2})}
\\ &  \quad 
\cdot
\Bigg(1 + c_{10}  
+ c_{10}  
e^{h c_9 \left(c_2 + c_4 \left(\frac{c_9pq}{q-1}-1 \right)^+ \right)}   
 \left[1+ h \left(c_1+c_3 \left(\tfrac{c_9pq}{q-1}-1\right)^+  \right)  \right]^{c_9} 
\Bigg)
\\ & \quad  \cdot
\Bigg[
c_5 +  e^{(c_7 \vee 1) h\left(c_2 + c_4\left(pq(c_7 \vee 1)-1\right) \right)}  
\Bigg(
c_6 
  \Big[1+ h \big(c_1+c_3 \left( pq c_7-1\right)^+  \big)  \Big]^{c_7}  
\\ &  \qquad  
+
\sqrt{pq(pq-1)}
 \sqrt{c_3
+
   \tfrac{c_1c_3}{2}+\tfrac{c_3^2}{2} \left( \tfrac{pq}{2} -1\right)
+
  c_4
   \Big[1+ \sqrt{h} \big(c_1+c_3  pq-c_3   \big)  \Big]^2
 }
\Bigg)
\Bigg] \! \vee \! 1
\Bigg\|_{L^p( \Omega; \R )} 
 \\ & 
=
\tfrac{2}{c_8}  \sqrt{h} \Big\|1 + X_{t_1}^x  \Big\|_{L^{p\left(c_9 + (c_7 \vee \nicefrac{1}{2})\right)} ( \Omega; \R )}^{c_9 + (c_7 \vee \frac{1}{2})} 
\\ &  \quad 
\cdot
\Bigg(1 + c_{10}  
+ c_{10}  
e^{h c_9 \left(c_2 + c_4 \left(\frac{c_9pq}{q-1}-1 \right)^+ \right)}   
 \left[1+ h \left(c_1+c_3 \left(\tfrac{c_9pq}{q-1}-1\right)^+  \right)  \right]^{c_9} 
\Bigg)
\\ & \quad  \cdot
\Bigg[
c_5 +  e^{(c_7 \vee 1) h\left(c_2 + c_4\left(pq(c_7 \vee 1)-1\right) \right)}  
\Bigg(
c_6 
  \Big[1+ h \big(c_1+c_3 \left( pq c_7-1\right)^+  \big)  \Big]^{c_7}  
\\ &  \qquad  
+
\sqrt{pq(pq-1)}
 \sqrt{c_3
+
   \tfrac{c_1c_3}{2}+\tfrac{c_3^2}{2} \left( \tfrac{pq}{2} -1\right)
+
  c_4
   \Big[1+ \sqrt{h} \big(c_1+c_3  pq-c_3   \big)  \Big]^2
 }
\Bigg)
\Bigg] \! \vee \! 1.
\end{split} 
\end{equation}
This and 
Lemma~\ref{l:reg_CIR} yield
that for all 
$
  h \in [ 0, T - t_1 ] 
$,  
$ 
  x \in [0, \infty)
$, 
$ p \in (0,\infty) $,
$ q \in [1,\infty) $ 
with $ p q \geq 2 $
it holds that
\begin{equation} 
\begin{split} 
 \label{l:cir_hr:8} 
&
  \big\|
    \phi( X_{ t_1 + h }^x ) - 
    \phi( X_{ t_1 }^x )
  \big\|_{
    L^p( \Omega; \R )
  } 
\\ & \leq
  \frac{ 2 \sqrt{ h } }{ c_8 }  
\Bigg[ 1 + e^{ t_1(c_2 + c_4(p(c_9 + (c_7 \vee \nicefrac{1}{2}))-1)^+ )}    \Big[x+ t_1 \big(c_1+c_3 (p(c_9 + (c_7 \vee \tfrac{1}{2}))-1)^+  \big)  \Big] \Bigg]^{c_9 + (c_7 \vee \frac{1}{2})}
\\ &
\quad\cdot
\Bigg(1 + c_{10}  
+ c_{10}  
e^{h c_9 \left(c_2 + c_4 \left(\frac{c_9pq}{q-1}-1 \right)^+ \right)}   
 \left[1+ h \left(c_1+c_3 \left(\tfrac{c_9pq}{q-1}-1\right)^+  \right)  \right]^{c_9} 
\Bigg)
\\ & \quad  \cdot
\Bigg[
c_5 +  e^{(c_7 \vee 1) h\left(c_2 + c_4\left(pq(c_7 \vee 1)-1\right) \right)}  
\Bigg(
c_6 
  \Big[1+ h \big(c_1+c_3 \left( pq c_7-1\right)^+  \big)  \Big]^{c_7}  
\\ &  \qquad  
+
\sqrt{pq(pq-1)}
 \sqrt{c_3
+
   \tfrac{c_1c_3}{2}+\tfrac{c_3^2}{2} \left( \tfrac{pq}{2} -1\right)
+
  c_4
   \Big[1+ \sqrt{h} \big(c_1+c_3  pq-c_3   \big)  \Big]^2
 }
\Bigg)
\Bigg] \! \vee \! 1.
\end{split} 
\end{equation}
This, in particular, proves
that for all 
$ p \in (0,\infty) $,
$ q \in [1,\infty) $ 
with $ p q \geq 2 $
it holds that
\begin{equation}
  \sup_{
    x \in [0,\infty)
  }
  \left[
  \frac{ 1 }{
    \left[
      1 + 
      x^{
        c_9 + \max\{ c_7 , \nicefrac{ 1 }{ 2 } \}
      }
    \right]
  }
  \left[
  \sup_{ 
    t_1, t_2 \in [0,T] , \, t_1 \neq t_2
  }
  \frac{
    \big\|
      \phi( X_{ t_2 }^x ) - 
      \phi( X_{ t_1 }^x )
    \big\|_{
      L^p( \Omega; \R )
    } 
  }{
    \left| t_2 - t_1 \right|^{ \nicefrac{ 1 }{ 2 } }
  }
  \right]
  \right]
  < \infty
\end{equation}
Combining this with
the Sobolev embedding that for all
$ p \in (1,\infty) $,
$ \alpha \in ( \nicefrac{ 1 }{ p }, 1 ] $,
$ \eps \in ( 0, \alpha - \nicefrac{ 1 }{ p } ) $
it holds that
$
  C^{ \alpha }( 
    [0,T]
    ,
    L^p( \Omega; \R )
  )
  \subseteq
  L^p( 
    \Omega ;
    C^{ \alpha - 1 / p - \eps }( [0,T] , \R )
  )
$
continuously
implies
that for all $p, \eps  \in (0,\infty)$ it holds that
\begin{equation} 
\begin{split} 
 \label{l:cir_hr:100} 
 &
  \sup_{ 
    x \in [0,\infty)
  }
  \left[
    \tfrac{ 1 }{
      1 + 
      x^{ 
        c_9 + \max\{ c_7 , \nicefrac{ 1 }{ 2 } \} 
      }
    }
    \,
    \Big\|
      \sup\nolimits_{ s, t \in [0,T] , \, s \neq t } 
      \tfrac{
        | \phi( X_t^x ) - \phi( X_s^x ) |
      }{
        | t - s |^{ \nicefrac{ 1 }{ 2 } - \eps } 
      } 
    \Big\|_{
      L^p( \Omega; \R ) 
    } 
  \right]
  < \infty .
\end{split} 
\end{equation}
This finishes the proof of Theorem~\ref{l:cir_hr}.
\end{proof}

The following corollary, Corollary~\ref{cor_hr}, specialises Theorem~\ref{l:cir_hr} (applied with $c_2=c_4=c_9=c_{10}=0, c_1=\delta, c_3=\tfrac{\beta^2}{2}, c_5=\delta, c_6=\gamma^+, c_7=1, c_8=\beta, q=1$) for Cox-Ingersoll-Ross  processes
and generalizes Lemma 3.2 in Dereich, Neuenkirch \& Szpruch~\cite{DereichNeuenkirchSzpruch2012}
which assumes
$\tfrac{2\delta}{\beta^2}\in(1,\infty)$.

\begin{corollary}
\label{cor_hr}
Let $\delta  \in [0,\infty)$,  $\gamma\in\R$, $T,\beta \in (0,\infty)$, 
let
$
  ( 
    \Omega, \mathcal{F}, \P,
       ( \mathcal{F}_t )_{ t \in [0,T] } 
  )
$
be a stochastic basis,
let
$
  W \colon [0,T] \times \Omega \to \R
$
be a standard $ ( \mathcal{F}_t )_{ t \in [0,T] } $-Brownian motion and
let
  $X^x \colon [0,T]\times \Omega\to [0,\infty)$, $x \in [0,\infty)$,
  be a family of adapted stochastic processes with continuous sample paths satisfying
  \begin{equation} \label{cor_hr:sde}
    X_t^x=x+\int_0^t \delta - \gamma X_s^x\,ds+ \int_0^t \beta \sqrt{X_s^x} dW_s
  \end{equation}
  $\P$-a.s.\
  for all $t \in [0,T]$, $x \in [0,\infty)$.
  Then it holds for all  $s, t \in [0,T]$, $x \in [0,\infty)$, $p \in (0,\infty)$ that
\begin{align} 
\label{cor_hr:state} 
&
  \left\| 
    \sqrt{ X_t^x } - 
    \sqrt{ X_s^x } 
  \right\|_{
    L^p( \Omega ; \R )
  } 
\leq
  \sqrt{ | t - s | }
  \,
  \Big[ 
    1 + x +
    \min\{ s , t \} 
    \big[
      \delta  
      + 
      \tfrac{ \beta^2 }{ 2 } 
      ( p - 1 )^+ 
    \big]
  \Big] 
\\ & \cdot 
  \max\Big\{ 
    1 ,
    \delta 
    + 
  \gamma^+
  \Big[ 
    1 + \left| t - s \right| 
    \big[
      \delta + \tfrac{ \beta^2 }{ 2 } 
      [ \tfrac{ p }{ 2 } - 1 ]^+  
    \big]
  \Big]   
+
  \beta
  \sqrt{
    \max( p ( p - 1 ) , 2 ) 
  }
  \sqrt{
    1
    +
     \tfrac{ 1 }{ 2 }
     \big[
       \delta + 
       \tfrac{ \beta^2 }{ 2 } 
       \max\{ p - 1 , 1 \} 
     \big] 
 }
  \Big\}
\nonumber
\end{align} 
and it holds for all $p, \eps  \in (0,\infty)$ that
\begin{equation} 
\begin{split} 
 \label{cor_hr:state2} 
 &
  \sup_{ x \in [0,\infty) }
  \left[
    \frac{ 1 }{ 1 + x }
    \left\|
      \sup_{ 
       \substack{ 
         s, t \in [0,T] , \,
         s \neq t 
       }
     }
   \tfrac{\left| \sqrt{X_{t}^x}  - \sqrt{X_{s}^x} \right| }{ |t-s|^{\nicefrac{1}{2} - \eps} }
 \right\|_{L^p( \Omega; \R )} 
 \right]
< \infty.
\end{split} 
\end{equation}

\end{corollary}

\section{Strong convergence rates for drift-implicit 
(square-root) 
Euler approximations of Bessel- and 
Cox-Ingersoll-Ross-type processes}

Strong convergence rates are established in 
Corollary~\ref{l:num_corollary}
and 
Theorem~\ref{thm:num_theorem} below.

\subsection{On the relation between 
Bessel- and Cox-Ingersoll-Ross-type processes}

In Lemma~\ref{l:transformierte_gleichung} below we establish, roughly speaking, 
that a transformation of a solution process of
an SDE with non-additive noise is a solution process
of a suitable SDE with additive noise (cf., e.g., 
Exercise XI.1.26 in~Revuz \& Yor~\cite{RevuzYor1994} for the case of CIR processes).
In the proof of Lemma~\ref{l:transformierte_gleichung} the following 
result, Lemma~\ref{l:PX0=0}, is used.

\begin{lemma}
[Time at non-trap boundaries has Lebesgue measure zero]
\label{l:PX0=0}
Let $I\subseteq\R $ be an open interval, let $T \in [0,\infty)$,
$J\subseteq\overline{I}$,
$\mu \in \mathcal{L}^0(J; \R)$, $\sigma \in \mathcal{L}^0(J, [0,\infty))$
satisfy $I\subseteq J$
and $\sup_{z\in K}|\sigma(z)|<\infty$ for all compact sets $K\subseteq J$,
let
$
  ( 
    \Omega, \mathcal{F}, \P,
$ 
$ 
       ( \mathcal{F}_t )_{ t \in [0,T] } 
  )
$
be a stochastic basis,
let
$
  W \colon [0,T] \times \Omega \to \R
$
be a standard $ ( \mathcal{F}_t )_{ t \in [0,T] } $-Brownian motion,
 let $X\colon[0,T]\times \Omega\to J$
  be an adapted stochastic process with continuous sample paths satisfying 
  $\int_0^T |\mu(X_s)| \,ds < \infty$ $\P$-a.s.\
  and
  \begin{equation}
    X_t=X_0+\int_0^t \mu(X_s)\,ds+\int_0^t\sigma(X_s)\,dW_s
  \end{equation}
  $\P$-a.s.\
  for all $t\in[0,T]$. 
  Then it holds for all 
  $ 
    b \in 
    \{ 
      z \in  J \colon |\mu(z)| > 0 , 
      \limsup_{J\setminus\{z\} \ni x \to z} \tfrac{(\sigma \cdot \sigma) (x)}{|x-z|} < \infty
    \} 
  $  
that
\begin{align} \label{lemma2.0:result}
  \P\!\left[ \int_0^T \1_{\{ X_s = b\}} \,ds = 0 \right]
  = 1 
  .
\end{align}
\end{lemma}

\begin{proof} [Proof of Lemma~\ref{l:PX0=0}]
Define a set 
$
  \Gamma := 
  \{ 
    z \in  J \colon |\mu(z)| > 0 , 
    \limsup_{ J\setminus\{z\} \ni x \to z } \tfrac{ ( \sigma \cdot \sigma ) ( x ) }{ | x - z | } < \infty
  \}
$ 
and define a family $f_n \colon \R \to \R$, $n \in \N$, 
of functions by
\begin{align} \label{lemma2.0:func_def}
f_n(x) :=\begin{cases} 
 x \exp\! \left( 1- \tfrac{1}{1-n^2 x^2}\right)   &\colon x \in (-\tfrac{1}{n}, \tfrac{1}{n})\\
  0  &\colon x \in \R \backslash (-\tfrac{1}{n}, \tfrac{1}{n})
\end{cases}
\end{align}
for all $x \in \R$, $ n \in \N$. 
Observe that for all $ n \in \N$ it holds 
that $f_n \in C^2\!\left(\R, \R \right)$
and note that for all $ n \in \N $, $ x \in \R$ it holds that
\begin{align} \label{lemma2.0:first_diff}
f'_n(x) =\begin{cases} 
 \exp\!\left(1- \tfrac{1}{1-n^2 x^2}\right) - 2 n^2 x^2  \tfrac{ \exp\left( 1- \tfrac{1}{1-n^2 x^2}\right) }{(1-n^2 x^2)^2}   &\colon  x \in (-\tfrac{1}{n}, \tfrac{1}{n})\\
  0  & \colon  x \in \R \backslash (-\tfrac{1}{n}, \tfrac{1}{n})
\end{cases}
\qquad 
  \text{and}
\end{align}
\begin{align} \label{lemma2.0:sec_diff}
f''_n(x) =\begin{cases} 
  - \tfrac{6 n^2 x \exp\left( 1- \tfrac{1}{1-n^2 x^2}\right) }{(1-n^2 x^2)^2}
   - \tfrac{8 n^4 x^3 \exp\left( 1- \tfrac{1}{1-n^2 x^2}\right) }{(1-n^2 x^2)^3} 
   + \tfrac{4 n^4 x^3 \exp\left( 1- \tfrac{1}{1-n^2 x^2}\right) }{(1-n^2 x^2)^4} 
     &\colon  x \in (-\tfrac{1}{n}, \tfrac{1}{n})\\
  0  &\colon  x \in \R \backslash (-\tfrac{1}{n}, \tfrac{1}{n}).
\end{cases}
\end{align}
Next observe that \eqref{lemma2.0:func_def} and  \eqref{lemma2.0:first_diff} 
ensure that for all $z \in \R$ it holds that
\begin{align} \label{lemma2.0:def_boundary}
 \lim_{n \to \infty} \sup_{x \in \R} |f_n(x)| = 0, 
\qquad  
  \sup_{n\in \N} \sup_{x \in \R} \left| f'_n(x) \right|  \leq 3 \exp(1)
\qquad \text{and} 
\qquad  
  \lim_{n \to \infty} f'_n(z) = \1_{\{0\}} (z).
\end{align} 
Furthermore, note that \eqref{lemma2.0:sec_diff} shows that
\begin{align} \label{lemma2.0:dom1}
  \sup_{n  \in \N} \sup_{x  \in \R} | f''_n(x)x | 
\leq   \sup_{x  \in (-1,1)}  \tfrac{18 \exp \left( 1- \tfrac{1}{1-x^2}\right) }{(1-x^2)^4} < \infty.
\end{align} 
Next define
a family 
$ 
  \phi_{b} \colon  J \to [0, \infty)
$, 
$
  b \in \Gamma
$,  
of functions by
\begin{align} \label{lemma2.0:upper_semi}
\phi_{b}(z):=\begin{cases} 
 \tfrac{(\sigma \cdot \sigma)(z)}{|z-b|}   & \colon z \in J\setminus\{b\} \\ 
   \limsup_{J\setminus\{b\} \ni x \to b} \tfrac{(\sigma \cdot \sigma)(x)}{|x-b|}
   &\colon z \in \{b\}
\end{cases}
\end{align}
for all $z \in  J $, $b \in \Gamma$. 
The boundedness of $ \sigma $ on compact subsets of $ J $ and the fact 
that for all $ b \in \Gamma$ it holds that 
$ \limsup_{J\setminus\{b\} \ni x \to b}
\tfrac{(\sigma \cdot \sigma)(x)}{|x-b|} < \infty $ 
implies that for every $ b \in \Gamma $ and every compact set $K\subseteq J$
it holds
that $\phi_{b}|_{K}$ is an upper semi-continuous function
on the compact set $  K $ and, therefore,
that $\phi_{b}|_{K}$ is bounded from above.
Furthermore, note  that 
\eqref{lemma2.0:dom1}, \eqref{lemma2.0:upper_semi} and 
the fact that for all $ n \in \N $ it holds that $ f_{n}''(0) = 0 $ 
show that for all $ b \in \Gamma $ and all compact sets $ K \subseteq J $
it holds that
\begin{equation} 
\begin{split} 
 \label{lemma2.0:majorante} 
\sup_{n  \in \N} \sup_{z \in  K } | f''_n(z-b) \, (\sigma \cdot \sigma)(z) | 
 & = 
 \sup_{n  \in \N} \sup_{ z \in K \backslash \{ b \} } 
 \left| 
   \tfrac{ ( z - b ) f''_n( z - b ) \, ( \sigma \cdot \sigma )( z ) }{ z - b } 
 \right|
\\ & \leq   
\left[
  \sup_{ z \in  K }  
  \phi_{b}( z )  
\right]
\left[
  \sup_{ z \in (-1,1) }  
  \tfrac{ 18 \exp\left( 1 - \tfrac{ 1 }{ 1 - z^2 } \right) 
  }{
    ( 1 - z^2 )^4
  } 
\right]
  < \infty
  .
\end{split} 
\end{equation} 
%
%
%
In the next step we define stopping times 
$
  \tau_k \colon \Omega \to [0,T]
$, 
$ k \in \N $, by
\begin{align}  
  \tau_k
  := 
  \inf\!\left( 
    \left\{ T \right\} 
    \cup 
    \left\{
      t \in [0,T] \colon 
      \smallint\limits_0^t
      \left| \mu( X_s ) \right| ds  
      +
      |X_t| > k
    \right\}
    \cup
    \Big\{
      t \in [0,T] \colon 
      \big(
        \exists \,
        b \in ( \partial I ) \backslash J
        \colon
        | X_t - b | < \tfrac{ 1 }{ k }
      \big)
    \Big\}
  \right)
\end{align}
for all $ k \in \N $.
It\^{o}'s lemma implies that
\begin{equation} 
\begin{split} 
\label{lemma2.0:ito} 
  f_n(X_{t \wedge \tau_k} - b) - f_n( X_0 - b ) 
& = 
  \smallint_0^{ t \wedge \tau_k } 
    f'_n(X_s - b) \, \mu(X_s) + \tfrac{ 1 }{ 2 } \, f''_n( X_s - b ) 
    \, ( \sigma \cdot \sigma )( X_s ) 
  \, ds 
\\ &  
  + \smallint_0^{t \wedge \tau_k} f'_n(X_s - b) \, \sigma(X_s)  \, dW_s
\end{split} 
\end{equation} 
$\P$-a.s.\ for all $b \in \Gamma$, $t \in [0,T]$, $n, k \in \N$. 
Next note that the boundedness of $ \sigma $ on compact subsets of $ J $,
the compactness of the sets 
$ 
  K_k 
  := 
  [ - k , k ] 
  \cap
  \big\{
    z \in J \colon 
    \big(
      \forall \, 
      b \in ( \partial I ) \setminus J 
      \colon
      |z - b| \geq \tfrac{ 1 }{ k } 
    \big)
  \big\}
  \subseteq J
$, 
$
  k \in \N
$, 
and \eqref{lemma2.0:def_boundary} imply that 
for all $ b \in \Gamma $, $ k \in \N $ it holds that
\begin{equation}
\begin{split}
  \sup_{ n \in \N }  
  \sup_{ x \in K_k } 
  \big[ 
    f'_n(x-b) \, \sigma(x) 
  \big]^2
\leq
  \left[ 
    \sup_{ n \in \N }  
    \sup_{ x \in \R }
    \left| f'_n( x ) \right|
  \right]^2
  \left[ 
    \sup_{ x \in K_k } 
    \left|
      \sigma(x)
    \right|
  \right]^2
  < \infty .
\end{split}
\end{equation}
This implies that the expectation of the stochastic integral 
on the right-hand side of \eqref{lemma2.0:ito} vanishes. 
Hence, it holds for all $ b \in \Gamma $, $ t \in [0,T] $, $ n, k \in \N $ 
that
\begin{align} \label{lemma2.0:expectations}
  \E\big[
    f_n(X_{t \wedge \tau_k} - b) - f_n(X_0 - b ) 
  \big] 
& = 
  \E\!\left[ 
    \smallint_0^{ t \wedge \tau_k } 
      f'_n( X_s - b ) \, \mu(X_s) + 
      \tfrac{ 1 }{ 2 } \, f''_n( X_s - b )  
      \, ( \sigma \cdot \sigma )( X_s ) \, ds 
  \right]
  .
\end{align}
The dominated convergence theorem together with \eqref{lemma2.0:def_boundary},
\eqref{lemma2.0:majorante} and 
\eqref{lemma2.0:sec_diff} 
shows that for all 
$ b \in \Gamma $, $ t \in [0,T] $, $ k \in \N $ 
it holds that
\begin{equation}
\label{lemma2.0:new1}
\begin{split}
   0 & = 
   \lim_{ n \to \infty } 
   \E\big[
     f_n( X_{ t \wedge \tau_k } - b ) - f_n( X_0 - b ) 
   \big] 
 \\ & = 
 \lim_{ n \to \infty} 
 \E \Bigg[ \smallint_0^{t \wedge \tau_k} 
 \Big(
   f'_n( X_s - b ) \, \mu(X_s) 
   + 
   \tfrac{ 1 }{ 2 } \,
   f''_n( X_s - b ) \, 
   ( \sigma \cdot \sigma )( X_s ) 
 \Big) \, ds 
 \Bigg]
 \\ & =  \E \Bigg[ \smallint_0^{t \wedge \tau_k} 
 \lim_{ n \to \infty} 
 \Big(
   f'_n( X_s - b ) \, \mu(X_s) 
   + 
   \tfrac{ 1 }{ 2 } \,
   f''_n( X_s - b ) \, 
   ( \sigma \cdot \sigma )( X_s ) 
 \Big) \, ds 
 \Bigg]
=  \E \!\left[ \smallint_0^{t \wedge \tau_k} \1_{\{b \}} (X_s) \, \mu(b) \, ds \right].
\end{split}
\end{equation}
Next we obtain from  
$
  \int_0^T |\mu(X_s)| \,ds < \infty
$ 
$ \P $-a.s.\ and from the assumption that $ X $ has continuous sample paths that 
$ \lim_{ k \to \infty } \tau_k = T $ $ \P $-a.s.
This, \eqref{lemma2.0:new1}, the fact 
that for all $ b \in \Gamma $ it holds that 
$ |\mu(b)| > 0 $ and Fatou's lemma show
that for all $b \in \Gamma$, $t \in [0,T]$ it holds that
\begin{align} 
  \E \! \left[ \smallint_0^t \1_{ \{ X_s = b \} } \, ds \right] 
  =
  \E \! \left[ \smallint_0^t \1_{\{b \}} (X_s) \, ds \right] 
  \leq \liminf_{k \to \infty}  \E \! \left[ \smallint_0^{t \wedge \tau_k} \1_{\{b \}} (X_s) \, ds \right]  = 0.
\end{align} 
%
%
%
This finishes the proof of Lemma~\ref{l:PX0=0}.
\end{proof}

\begin{lemma}  \label{l:transformierte_gleichung}
Let $ I \subseteq \R $ be an open interval,
let 
$ T \in [0,\infty)$, $x_0 \in \overline{I}$,  
$ J \subseteq \overline{I} $,
$
  \mu \in C(J, \R)
$, 
$
  \sigma \in C(J, [0,\infty))
$ 
satisfy  
$
  I \subseteq J
$,
$
  \sigma(I) \subseteq (0,\infty)
$, 
$
  \sigma \cdot \sigma
  \in
  C^1\!\left(J, [0,\infty) \right)$, 
let
$
  ( 
    \Omega, \mathcal{F}, \P, ( \mathcal{F}_t )_{ t \in [0,T] } 
  )
$
be a stochastic basis, let
$
  W \colon [0,T] \times \Omega \to \R
$
be a standard $ ( \mathcal{F}_t )_{ t \in [0,T] } $-Brownian motion,
  let $X\colon[0,T]\times \Omega\to J$
  be an adapted stochastic process with continuous sample paths satisfying 
  \begin{equation} \label{l:transformierte_gleichung:ass1}
    X_t=X_0+\int_0^t \mu(X_s)\,ds+\int_0^t\sigma(X_s)\,dW_s
  \end{equation}
   $\P$-a.s.\ for all $t \in [0,T]$,
   assume that 
   $
     \big|
       \!
       \int_{ x_0 }^y \tfrac{ 1 }{ \sigma(z) } \, dz 
     \big| < \infty
   $
  for all $y \in J$,
    let $\phi\colon J \to \R$ be a function defined by
  $
    \phi(y) := \int_{x_0}^y \tfrac{1}{\sigma(z)} \, dz
  $
  for all $ y \in J $
  and
  assume that
  for all $b \in J \cap \partial I$ it holds that
  $
    \sigma(b) = 0
  $
  and
  %
  \begin{align}
  \big(
    \mu(b) - \tfrac{ 1 }{ 4 } ( \sigma \cdot \sigma )'( b ) 
  \big)
  \left( 
    \1_{
      \{ ( \sigma \cdot \sigma )'( b ) \geq 0 \}
    } 
    -
    \1_{ 
      \{
        ( \sigma \cdot \sigma )'( b ) < 0 
      \}
    } 
  \right) 
  > 0 ,
  \quad
\label{l:transformierte_gleichung:ass4} 
     \limsup_{I \ni x \to b} \tfrac{|\mu(x) - \mu(b)| + | (\sigma \cdot \sigma)'(x) - (\sigma \cdot \sigma)'(b)|}{\sigma(x)} &< \infty.
\end{align} 
  Then it holds that 
  $ \phi $ is injective and absolutely continuous,
  it holds that
\begin{align}
\label{eq:SDE_transformation_integral_finite}
    \int_{0}^{T}
      \tfrac{
        1
          +
        \left|
          \mu(X_s)
        \right|
          + 
        |
          (\sigma \cdot \sigma)'(X_s)
        |
      }{
        \sigma( X_s )
      }  
      \,
      \1_{\{ X_s \in I\} } 
   \, ds < \infty
\end{align}
$\P$-a.s.\ and it holds that
the mapping $Y\colon[0,T]\times \Omega\to \phi(J)$ defined by
   $Y_t := \phi(X_t)$  
  for all $t \in [0,T]$
  is an adapted stochastic process with continuous sample paths which
  satisfies
  \begin{equation} \label{l:transformierte_gleichung:ass5}
   Y_t = Y_0 + \int_0^t \left( \tfrac{\mu - \frac{1}{4} (\sigma \cdot \sigma)'}{\sigma} \right) \! \left( \phi^{-1}(Y_s) \right) \1_{\{Y_s \in \phi(I)\}} \, ds + W_t
  \end{equation}
  $\P$-a.s.\
  for all $t\in[0,T]$.
\end{lemma}

\begin{proof}[Proof 
of Lemma~\ref{l:transformierte_gleichung}]
First of all, note 
that the assumption that for all $ b \in J $
it holds that
$
  \big| 
    \int_{x_0}^b \tfrac{1}{\sigma(z)} \, dz 
  \big| < \infty
$
ensures that $ \phi $ is well-defined
and absolutely continuous. 
Next observe that for all $x,y \in J$ it holds that $\int_{y}^x \tfrac{1}{\sigma(z)} \, dz = 0$ if and only if $\phi(x) = \phi(y)$. 
This together with the assumption that $\sigma \geq 0$ 
implies that for all $x,y \in J$ it holds that $ \phi(x) = \phi(y) $ 
if and only if $ x = y $, 
which proofs the injectivity of $ \phi $.
In the next step we define families 
$ \sigma_{ \eps } \colon J \to [0,\infty) $, 
$ \eps \in (0,1) $, 
and
$
  \phi_{\eps}\colon J \to \R
$, 
$
  \eps \in (0,1)
$,
of functions by
\begin{equation} \label{l:transformierte_gleichung:phi_eps}
  \sigma_{ \eps }( z ) := 
  \sqrt{ \eps + ( \sigma \cdot \sigma )( z ) }
\qquad 
  \text{and}
\qquad
  \phi_{\eps}(y) := \int_{x_0}^y \tfrac{1}{\sigma_{\eps}(z)} \, dz
\end{equation}
for all $ y \in J $, $ \eps \in (0,1) $. 
Note that for all 
$ \eps \in (0,1) $ 
it holds that 
$
  \sigma_{\eps}  \in
  C^1\!\left(J, [0,\infty) \right)
$  
and 
$
  \phi_{ \eps }  
  \in
  C^2\!\left( J, \R \right)
$. 
It\^{o}'s lemma
and the fact that 
for all $ x \in J $, $ \eps \in (0,1) $
it holds that
$
  (\sigma_{\eps} \cdot \sigma_{\eps})'(x) =  (\sigma \cdot \sigma)'(x)
$ 
and 
$
  (
    \sigma_{ \eps } 
  )'(x) 
  = 
  \frac{
    ( \sigma_{ \eps } \cdot \sigma_{ \eps } )'(x) 
  }{ 
    2 \sigma_{ \eps }( x ) 
  } 
$ 
hence implies 
that for all $ \eps \in (0,1) $, $ t \in [0,T] $ 
it holds that
\begin{equation} 
\begin{split} 
 \label{l:transformierte_gleichung:ito}
  \phi_{\eps}(X_t)& =  \phi_{\eps}(X_0) + \int_{0}^t \tfrac{\mu(X_s)}{\sigma_{\eps}(X_s)} - \tfrac{1}{2} \tfrac{\sigma'_{\eps}(X_s)}{(\sigma_{\eps} \cdot \sigma_{\eps}) (X_s)} (\sigma \cdot \sigma) (X_s) \, ds
  + \int_{0}^t \tfrac{\sigma(X_s)}{\sigma_{\eps} (X_s)} \, dW_s
\\ 
& 
=  \phi_{\eps}(X_0) +
 \int_{0}^t 
\tfrac{\mu(X_s) - \frac{1}{4} (\sigma \cdot \sigma)'(X_s)
\frac{(\sigma \cdot \sigma) (X_s) }{(\sigma_{\eps} \cdot \sigma_{\eps}) (X_s)}   }{\sigma_{\eps} (X_s)} 
 \, ds
  + \int_{0}^t \tfrac{\sigma(X_s)}{\sigma_{\eps} (X_s)}  \, dW_s.
\end{split} 
\end{equation} 
Lemma~\ref{l:PX0=0} together with the fact that 
for all $b \in J \cap \partial I$ it holds that
$ | \mu(b)| >0 = \sigma(b) $
shows that 
for Lebesgue almost all 
$ t \in [0,T] $
it holds that
\begin{equation}
\label{eq:X_ist_kaum_am_Rand}
  \P\big[ X_t \in I \big]
  = 1
  .
\end{equation}
Next we fix $ y_0 \in I $,
we define a sequence 
$
  \tau_k \colon \Omega \to [0,T]
$, 
$
  k \in \N
$, 
of stopping times 
by 
$
  \tau_0 := 0
$ 
and
\begin{align}  
  \tau_k := 
  \inf\!\big( 
    \left\{ T \right\} 
    \cup 
    \left\{ 
      t \in ( \tau_{ k - 1 } , T ] 
      \colon 
      X_t = y_0 \text{ and } 
      \exists \, s \in [ \tau_{ k - 1 } , t ] 
      \colon 
      X_s \in \partial I
    \right\} 
  \big)
\end{align}
for all $ k \in \N $ 
and we define a mapping $\kappa \colon \Omega \to \N_0 \cup \{\infty\}$
by $\kappa : = \min(\{k\in \N_0 \colon \tau_k =T \}\cup\{\infty\})$.
Pathwise continuity of $X$ implies that for all $ \omega \in \Omega $
it holds that $ \kappa(\omega) <  \infty$.
%
%
%
%
%
%
%
%
%
%
%
%
%
Now assumption
\eqref{l:transformierte_gleichung:ass4}
and the assumptions that
$
  \sigma \cdot \sigma
  \in C^1\!\left( J , [0,\infty) \right)
$, 
$
  \sigma( I ) \subseteq (0,\infty)
$
and
$
  \mu \in C( J, \R )
$
imply 
that for every 
$
  b \in J \cap \partial I
$ 
and every compact set
$
  K \subseteq I \cup \{ b \}
$ 
it holds that
\begin{align} \label{l:transformierte_gleichung:dom1}
  \sup_{ \eps \in (0,1) } 
  \sup_{ z \in K } 
  \left| 
    \tfrac{ 
      \mu(z) - \mu(b) - 
      \frac{ 1 }{ 4 } 
      \big( 
        ( \sigma \cdot \sigma )'(z) - 
        ( \sigma \cdot \sigma )'(b) 
      \big)
      \frac{
        ( \sigma \cdot \sigma )(z) 
      }{
        ( \sigma_{ \eps } \cdot \sigma_{ \eps } )(z)
      }   
    }{
      \sigma_{ \eps }(z)
    } 
  \right|
\leq
  \sup_{
    z \in K \backslash \{ b \}
  } 
  \tfrac{
    | \mu(z) - \mu(b) | + 
    | ( \sigma \cdot \sigma )'(z) - ( \sigma \cdot \sigma )'(b) |
  }{
    \sigma(z)
  } 
  < \infty 
  .
\end{align}
This proves that
for all $ b \in \partial I \cap J $, $ k \in \N $
it holds $ \P $-a.s.\ that
\begin{align}
\label{l:domination_integral}
&
  \mathbbm{1}_{
    \{ 
      \exists \, u \in [ \tau_{ k - 1 }, \tau_k ] \colon
      X_u = b
    \}
  }
  \int_{ \tau_{ k - 1 } }^{
    \tau_k
  }
    \tfrac{ 
      \left| \mu(X_s) - \mu(b) \right|
      +
      \frac{ 1 }{ 4 } 
      | 
        ( \sigma \cdot \sigma )'( X_s ) - 
        ( \sigma \cdot \sigma )'( b ) 
      |
    }{
      \sigma( X_s )
    }
    \,
    \1_{ 
      \{ X_s \in I \} 
    } 
    \,
  ds
\\ & \leq
  \mathbbm{1}_{
    \{ 
      \exists \, u \in [ \tau_{ k - 1 }, \tau_k ] \colon
      X_u = b
    \}
  }
  \int_{ \tau_{ k - 1 } }^{
    \tau_k
  }
  \sup\!\left\{
    \tfrac{
      | \mu(z) - \mu(b) | + 
      | ( \sigma \cdot \sigma )'(z) - ( \sigma \cdot \sigma )'(b) |
    }{
      \sigma(z)
    } 
  \colon
    z \in 
    \left[
      \cup_{ u \in [ \tau_{ k - 1 }, \tau_k ] } \{ X_u \}
    \right]
    \backslash \{ b \}
  \right\}
  ds
  < \infty
  .
\nonumber
\end{align}
Next, we define a function 
$
  \sgn \colon \R \to \{-1,1\}
$ 
by
$
  \sgn(x) := 
  \1_{
    \{ x \geq 0 \}
  } 
  -
  \1_{
    \{ x < 0 \}
  } 
$
for all $ x \in \R $.
Then observe that 
the fact that 
for all $ b \in J \cap \partial I $  
it holds that
$
  \big(\mu(b) - \tfrac{1}{4}  (\sigma \cdot \sigma)'(b) \big)
  \sgn\left( (\sigma \cdot \sigma)'(b) \right)  > 0
$
implies that for all $ b \in J \cap \partial I $ 
it holds that 
$
  \sgn\!\big( (\sigma \cdot \sigma)'(b) \big) \mu(b) = |\mu(b)| 
$.
This, 
\eqref{l:transformierte_gleichung:ito} 
and 
\eqref{eq:X_ist_kaum_am_Rand}
imply that
\begin{align}  \nonumber  \label{l:transformierte_gleichung:dom4b}
& \sgn\!\big( (\sigma \cdot \sigma)'(b) \big)    \int_{\tau_{k-1}}^{\tau_k} \tfrac{\mu(b) - \frac{1}{4} (\sigma \cdot \sigma)'(b)
   }{\sigma_{\eps} (X_s)} \left(\tfrac{(\sigma \cdot \sigma) (X_s) }{(\sigma_{\eps} \cdot \sigma_{\eps}) (X_s)}\right)  \,  ds
\leq
  \sgn\!\big( (\sigma \cdot \sigma)'(b) \big)   \int_{\tau_{k-1}}^{\tau_k} \tfrac{\mu(b) - \frac{1}{4} (\sigma \cdot \sigma)'(b)
\frac{(\sigma \cdot \sigma) (X_s) }{(\sigma_{\eps} \cdot \sigma_{\eps}) (X_s)}   }{\sigma_{\eps} (X_s)} \,  ds
\\ & =  
\nonumber
  \sgn\!\big( 
    ( \sigma \cdot \sigma )'(b) 
  \big)   
  \Bigg(
    \phi_{ \eps }( X_{ \tau_k } ) 
    - \phi_{ \eps }( X_{ \tau_{ k - 1 } } ) 
    - 
    \int_{ \tau_{ k - 1 } }^{ \tau_k } 
      \tfrac{ \sigma( X_s ) }{ \sigma_{ \eps }( X_s ) } 
    \, dW_s 
\\ & \quad -
  \int_{
    \tau_{ k - 1 } 
  }^{ \tau_k } 
  \tfrac{
    \mu( X_s ) - \mu( b ) 
    - 
    \frac{ 1 }{ 4 } 
    \big( ( \sigma \cdot \sigma )'( X_s ) - ( \sigma \cdot \sigma )'( b ) 
    \big)
    \frac{
      ( \sigma \cdot \sigma )( X_s ) 
    }{
      ( \sigma_{ \eps } \cdot \sigma_{ \eps } )( X_s )
    }   
  }{
    \sigma_{ \eps }( X_s ) 
  }
  \, 
  ds
\Bigg)  
\\ & =  
\nonumber
  \sgn\!\big( 
    ( \sigma \cdot \sigma )'(b) 
  \big)   
  \Bigg(
    \phi_{ \eps }( X_{ \tau_k } ) 
    - \phi_{ \eps }( X_{ \tau_{ k - 1 } } ) 
    - 
    \int_{ \tau_{ k - 1 } }^{ \tau_k } 
      \tfrac{ \sigma( X_s ) }{ \sigma_{ \eps }( X_s ) } 
    \, dW_s 
\\ \nonumber & \quad -
  \int_{
    \tau_{ k - 1 } 
  }^{ \tau_k } 
  \tfrac{
    \mu( X_s ) - \mu( b ) 
    - 
    \frac{ 1 }{ 4 } 
    \big( ( \sigma \cdot \sigma )'( X_s ) - ( \sigma \cdot \sigma )'( b ) 
    \big)
    \frac{
      ( \sigma \cdot \sigma )( X_s ) 
    }{
      ( \sigma_{ \eps } \cdot \sigma_{ \eps } )( X_s )
    }   
  }{
    \sigma_{ \eps }( X_s ) 
  }
  \,
  \1_{
    \{ X_s \in I \} 
  } 
  \, 
  ds
\Bigg)  
\end{align}
$\P$-a.s.\ for all $ b \in  \partial I \cap J $, $ \eps \in (0,1) $, $ k \in \N $. 
The monotone convergence theorem shows that 
for all $ t \in [0,T] $ it holds that
$
  \lim_{ (0,1) \ni \eps \to 0} 
  \phi_{\eps}(X_t) = \phi(X_t)
$.
Moreover,
Doob's martingale inequality,
the dominated convergence theorem,
\eqref{eq:X_ist_kaum_am_Rand}
and $\sigma^{-1}(\{0\})=J\cap\partial I$
imply that
\begin{equation} 
\label{l:transformierte_gleichung:ito_iso}
\begin{split}
&
 \lim_{(0,1) \ni \eps \to 0} 
 \left\|
 \sup_{
    t\in[0,T]
 }
 \left|
   W_t
   - 
   \int_{ 0 }^{ t } 
   \tfrac{ \sigma(X_s) }{ \sigma_{ \eps }( X_s ) } 
   \, dW_s 
 \right|
 \right\|_{ 
   L^2( \Omega; \R )
 }^2
 =
 \lim_{(0,1) \ni \eps \to 0} 
 \left\|
 \sup_{
    t\in[0,T]
 }
 \left|
   \int_{ 0 }^{ t } 
   1-
   \tfrac{ \sigma(X_s) }{ \sigma_{ \eps }( X_s ) } 
   \, dW_s 
 \right|
 \right\|_{ 
   L^2( \Omega; \R )
 }^2
\\ & 
\leq
 2
  \lim_{ (0,1) \ni \eps \to 0 } 
  \E\!\left[
    \int_0^T 
    \left( 
      1 - \tfrac{ \sigma(X_s) }{ \sigma_{ \eps }( X_s ) } 
    \right)^2 
    ds 
  \right] 
=
 2 \, 
   \E\!\left[
    \int_0^T 
    \lim_{ (0,1) \ni \eps \to 0 } 
    \left(
      1 - \tfrac{\sigma(X_s)}{\sigma_{\eps}(X_s)} 
    \right)^2 
    ds 
  \right] 
\\ & =
 2 \,
 \E\!\left[ 
   \int_0^T  \1_{\{\sigma(X_s) = 0 \}} \, ds 
 \right]
= 
 2 \int_0^T 
  \P\big[ 
    \sigma(X_s) = 0
  \big] \, ds
=
 2 \int_0^T 
  \P\big[ 
    X_s \in J \cap \partial I 
  \big] \, ds
= 0
  .
\end{split}
\end{equation}
Consequently,
there exists 
a strictly decreasing sequence
$ \eps_n \in (0,\infty) $, $n \in \N$,
such that 
it holds $ \P $-a.s.\ that
$
 \lim_{\N \ni n \to \infty}
 \sup_{t\in[0,T]}\big|W_t-\int_0^t \tfrac{\sigma(X_s)}{\sigma_{\eps_n}(X_s)}  \, dW_s
                 \big|=0
$.
This, the monotone convergence theorem, 
\eqref{l:transformierte_gleichung:dom4b}
and
the dominated convergence theorem 
together with
\eqref{l:domination_integral}
prove that
{\allowdisplaybreaks 
\begin{align}   
\label{l:transformierte_gleichung:dom3a} \nonumber
& 
  \mathbbm{1}_{
    \{ 
      \exists \, u \in [ \tau_{ k - 1 }, \tau_k ] \colon
      X_u = b
    \}
  }
 \int_{
   \tau_{k-1}
 }^{
   \tau_k
 } 
   \tfrac{
     | \mu(b) - \frac{ 1 }{ 4 } (\sigma \cdot \sigma)'(b) |
   }{
     \sigma( X_s ) 
   } 
   \, 
   \1_{\{ X_s \in I\} } 
 \, ds
\\ \nonumber
&=
  \mathbbm{1}_{
    \{ 
      \exists \, u \in [ \tau_{ k - 1 }, \tau_k ] \colon
      X_u = b
    \}
  }
  \lim_{\N \ni n \to \infty} 
  \sgn\!\big( 
    (\sigma \cdot \sigma)'(b) 
  \big) 
  \int_{\tau_{k-1}}^{\tau_k}  
    \tfrac{
      \mu(b) - \frac{1}{4} (\sigma \cdot \sigma)'(b)
    }{
      \sigma_{\eps_n}(X_s)} 
      \left(
        \tfrac{
          ( \sigma \cdot \sigma )(X_s) 
        }{
          ( \sigma_{ \eps_n } \cdot \sigma_{ \eps_n } )( X_s )
        }
      \right) 
      ds
\\ \nonumber
&\leq
  \mathbbm{1}_{
    \{ 
      \exists \, u \in [ \tau_{ k - 1 }, \tau_k ] \colon
      X_u = b
    \}
  }
\lim_{\N \ni n \to \infty}  \sgn\!\big( (\sigma \cdot \sigma)'(b) \big)   \Bigg(\phi_{\eps_n}(X_{\tau_{k}}) - \phi_{\eps_n}(X_{\tau_{k-1}}) -
 \int_{\tau_{k-1}}^{\tau_{k}} \tfrac{\sigma(X_s)}{\sigma_{\eps_n} (X_s)}  \, dW_s 
\\ 
& \qquad -
\int_{\tau_{k-1}}^{\tau_k} \tfrac{\mu(X_s) - \mu(b) - \frac{1}{4} \big( (\sigma \cdot \sigma)'(X_s) - (\sigma \cdot \sigma)'(b) \big)
\frac{(\sigma \cdot \sigma) (X_s) }{(\sigma_{\eps_n} \cdot \sigma_{\eps_n}) (X_s)}   }{\sigma_{\eps_n} (X_s)} \, \1_{\{ X_s \in I\} } \,  ds
\Bigg)  
\\ \nonumber
&=
  \mathbbm{1}_{
    \{ 
      \exists \, u \in [ \tau_{ k - 1 }, \tau_k ] \colon
      X_u = b
    \}
  }
\sgn\!\big( (\sigma \cdot \sigma)'(b) \big)   \Bigg(\phi(X_{\tau_{k}}) - \phi(X_{\tau_{k-1}}) -
(W_{\tau_k} - W_{\tau_{k-1}})
\\  \nonumber
& \qquad -
  \int_{
    \tau_{ k - 1 }
  }^{
    \tau_k
  } 
  \tfrac{ 
    \mu(X_s) - \mu(b) - \frac{ 1 }{ 4 } 
    \big( 
      ( \sigma \cdot \sigma )'( X_s ) - 
      ( \sigma \cdot \sigma )'( b ) 
    \big)
  }{
    \sigma( X_s )
  }
  \,
  \1_{ 
    \{ X_s \in I \} 
  } 
  \, 
  ds
\Bigg)  < \infty
\end{align}}%
$\P$-a.s.\ 
for all $ b \in \partial I \cap J $, $ k \in \N $. 
This and the fact that 
for all $ b \in \partial I \cap J $
it holds that
$
  | \mu(b) - \tfrac{ 1 }{ 4 } ( \sigma \cdot \sigma )'( b )| > 0 
$
imply that for all $ k \in \N $ it holds $ \P $-a.s.\ that
\begin{equation}
\begin{split}
&
  \int_{
    \tau_{ k - 1 }
  }^{
    \tau_k
  } 
    \tfrac{
      1
    }{
      \sigma( X_s ) 
    } 
    \, 
    \1_{
      \{ X_s \in I \} 
    } 
    \, 
  ds
\leq
  \mathbbm{1}_{
    \{ 
      \forall \, u \in [ \tau_{ k - 1 }, \tau_k ] \colon
      X_u \in I
    \}
  }
  \int_{
    \tau_{ k - 1 }
  }^{
    \tau_k
  } 
    \tfrac{
      1
    }{
      \sigma( X_s ) 
    } 
    \, 
    \1_{
      \{ X_s \in I \} 
    } 
    \, 
  ds
\\ &
  +
  \sum_{ b \in \partial I \cap J }
  \frac{ 1 }{
    | \mu(b) - \frac{ 1 }{ 4 } (\sigma \cdot \sigma)'(b) |
  }
  \left[
  \mathbbm{1}_{
    \{ 
      \exists \, u \in [ \tau_{ k - 1 }, \tau_k ] \colon
      X_u = b
    \}
  }
  \int_{
    \tau_{k-1}
  }^{
    \tau_k
  } 
    \tfrac{
      | \mu(b) - \frac{ 1 }{ 4 } (\sigma \cdot \sigma)'(b) |
    }{
      \sigma( X_s ) 
    } 
    \, 
    \1_{
      \{ X_s \in I \} 
    } 
    \, 
  ds
  \right]
  < \infty
  .
\end{split}
\end{equation}
The fact that for all $ \omega \in \Omega $ it holds that
$ \kappa( \omega ) < \infty $ hence proves that
\begin{equation}
  \int_0^T
  \tfrac{ 1 }{ \sigma( X_s ) } 
  \, \mathbbm{1}_{ \{ X_s \in I \} } \, ds
  < \infty
\end{equation}
$ \P $-a.s.
This and the fact that
$
  \P\big[
    \sup\{
      | \mu(X_s) | + | ( \sigma \cdot \sigma )'( X_s ) | 
      \colon s \in[0,T]
    \}
    < \infty
  \big]
  =
  1
$
imply 
\eqref{eq:SDE_transformation_integral_finite}.
Finally, 
the dominated convergence theorem together
with~\eqref{eq:SDE_transformation_integral_finite},
\eqref{eq:X_ist_kaum_am_Rand} and \eqref{l:transformierte_gleichung:ito}
shows that for every $ t \in [0,T] $ it holds $ \P $-a.s.\ that
 {\allowdisplaybreaks 
\begin{equation} 
\begin{split} 
 \label{l:transformierte_gleichung:ito_lim} 
& Y_t
  - Y_0
  - \int_0^t \left( \tfrac{\mu - \frac{1}{4} (\sigma \cdot \sigma)'}{\sigma} \right)
             \! \left( \phi^{-1}(Y_s) \right) \1_{\{Y_s \in \phi(I)\}}\,  ds
  - W_t
\\ 
& =
  \phi(X_t)
  - \phi(X_0)
  -  \int_{0}^t \tfrac{\mu(X_s) - \frac{1}{4} (\sigma \cdot \sigma)'(X_s)
                      }{\sigma (X_s)} \1_{\{ X_s \in I\} }\,  ds 
 - W_t
\\ 
& =
\lim_{n\to\infty}
\left(
  \phi_{\eps_n}(X_t) 
  - \phi_{\eps_n}(X_0)
  -  \int_{0}^t \tfrac{\mu(X_s) - \frac{1}{4} (\sigma_{\eps_n} \cdot \sigma_{\eps_n})'(X_s)
                      }{\sigma_{\eps_n} (X_s)} \, \1_{\{ X_s \in I\} } \, ds 
 - \int_{0}^t \tfrac{\sigma(X_s)}{\sigma_{\eps_n} (X_s)}  \, dW_s
\right)
\\ 
& =
\lim_{n\to\infty}
\left(
  \phi_{\eps_n}(X_t) 
  - \phi_{\eps_n}(X_0)
  -  \int_{0}^t \tfrac{\mu(X_s) - \frac{1}{4} (\sigma_{\eps_n} \cdot \sigma_{\eps_n})'(X_s)
                      }{\sigma_{\eps_n} (X_s)}  \, ds 
 - \int_{0}^t \tfrac{\sigma(X_s)}{\sigma_{\eps_n} (X_s)}  \, dW_s
\right)
=
0.
\end{split} 
\end{equation}  }%
This finishes the proof of Lemma~\ref{l:transformierte_gleichung}.
\end{proof}

\begin{remark}
  In the setting of Lemma~\ref{l:transformierte_gleichung} it holds that
  if
  $\big(
    \mu(b) - \tfrac{ 1 }{ 4 } ( \sigma \cdot \sigma )'( b ) 
  \big)
  \left( 
    \1_{
      \{ ( \sigma \cdot \sigma )'( b ) \geq 0 \}
    } 
    -
    \1_{ 
      \{
        ( \sigma \cdot \sigma )'( b ) < 0 
      \}
    } 
  \right) 
  \leq 0$ for some $b\in\partial I$, then the process $Y$ may not be a solution process of an It\^o SDE;
  see Section XI.1 in Revuz \& Yor~\cite{RevuzYor1994}
  for the example of the one-dimensional Bessel process.
\end{remark}

\subsection{Properties of drift-implicit Euler approximations for 
Bessel-type processes}

\begin{lemma}  \label{l:inverse_function}
  Let 
  $
    \left( H, \langle \cdot, \cdot \rangle_H, \left\| \cdot \right\|_H \right)
  $ 
  be an $ \R $-Hilbert space,
  let $ I \subseteq H $ be a non-empty set, 
  let $ T, L \in [0,\infty) $, 
  let $ \mu \colon I \to H $ 
  be a function
  satisfying
\begin{align} \label{l:inverse_function_ass1}
   \langle x-y, \mu(x) - \mu(y) \rangle_H \leq L \|x-y\|_H^2
\end{align}
for all $x, y \in I$,
let
  $F\colon(0,\tfrac{1}{L}) \times I \to H$ be a mapping defined by
   $F_t(x) := x - t \mu(x)$ 
for all $x \in I$, $t  \in (0,\tfrac{1}{L})$.
Then for every $ t \in (0,\tfrac{1}{L}) $ 
the function $ F_t \colon I \to H $ is injective and 
for all $ t \in ( 0, \frac{ 1 }{ L } ) $, $ x, y \in F_t(I) $ 
it holds that 
\begin{align} 
\label{eq:Finverse_estimate}
   \|
     F^{ - 1 }_t( x ) - F^{ - 1 }_t( y ) 
   \|_H 
   \leq  
   \tfrac{ 1 }{ (1 - t L) } 
   \left\| x - y \right\|_H
   .
\end{align}
\end{lemma}

\begin{proof} [Proof of Lemma~\ref{l:inverse_function}]
Observe that the Cauchy-Schwarz inequality 
and \eqref{l:inverse_function_ass1} 
imply that for all $ x, y \in I $, $ t \in ( 0, \frac{ 1 }{ L } ) $ 
it holds that
\begin{align} \label{l:inverse_function1} \nonumber
  \left\| x - y \right\|_H
  \left\| F_t(x) - F_t(y) \right\|_H
& \geq
  \langle x - y , F_t(x) - F_t(y) \rangle_H  
  = 
  \left\| x - y \right\|_H^2 
  - t \langle x - y , \mu(x) - \mu(y) \rangle_H 
\\ 
&
 \geq \|x-y\|_H^2  - tL \|x-y\|_H^2
 = \left( 1 - t L \right) \|x-y\|_H^2.
\end{align}
This ensures that for every $ t \in ( 0, \frac{ 1 }{ L } ) $ 
it holds that the mapping $ F_t $ is injective. 
%
%
%
In addition, \eqref{l:inverse_function1}
shows that
for all 
$ t \in (0,\frac{ 1 }{ L } ) $, $ x, y \in F_t(I) $ 
it holds that
\begin{equation} 
\label{eq:last_estimate}
  \| x - y \|_H
\geq
  \left( 1 - t L \right)
  \| 
    F_t^{ - 1 }(x) - F_t^{ - 1 }(y) 
  \|_H
  .
\end{equation}
Estimate~\eqref{eq:last_estimate}
implies \eqref{eq:Finverse_estimate}.
The proof of Lemma~\ref{l:inverse_function}
is thus completed.
\end{proof}
\begin{lemma}  \label{l:lemma_num}
Let $ I \subseteq \R $ be an open interval, let $T, L \in [0,\infty)$, 
$n \in \N_0$, $\theta=(t_0,\dots, t_n) \in [0,T]^{n+1}$ satisfy $0=t_0< t_1 < \dots < t_n=T$ and $L | \theta | < 1$,
let $\mu \in \mathcal{L}^0(\overline{I}; \R)$ satisfy
$( x-y) (\mu(x) - \mu(y)) \leq L (x-y)^2
$ for all $x, y \in I$,
let
$
  ( 
    \Omega, \mathcal{F}, \P, ( \mathcal{F}_t )_{ t \in [0,T] } 
  )
$
be a stochastic basis, let
$
  W \colon [0,T] \times \Omega \to \R
$
be a standard $ ( \mathcal{F}_t )_{ t \in [0,T] } $-Brownian motion, 
let $X\colon[0,T]\times \Omega\to\overline{I}$
  be an adapted stochastic process with continuous sample paths satisfying 
  $\int_0^t |\mu(X_s)| \,ds < \infty$ $\P$-a.s.\
  and
  \begin{align} 
  \label{l:lemma_num_X}
    X_t=X_0+\int_0^t \mu(X_s)\,ds+W_t
  \end{align}
 $\P$-a.s.\
  for all $t\in[0,T]$,
let 
$
  Y \colon \{ t_0, t_1, \ldots, t_n \} \times \Omega \to I
$ 
be a measurable mapping satisfying
  \begin{equation} \label{l:lemma_num_Y_umformung}
    Y_{t_{k+1}}= Y_{t_k} +  \mu( Y_{ t_{ k + 1 } } ) \left( t_{ k + 1 } - t_k \right) 
    + W_{ t_{ k + 1 } } - W_{ t_k }  
  \end{equation}
 $\P$-a.s.\ for all $k \in \{0,\dots,n-1 \}$ and
assume that 
 $
   \sup_{ i \in\{ 0, 1, \ldots, n \} }
   \P\!\left[ X_{ t_i } \in \partial I \right] 
   = 0
 $.
Then 
\begin{align} \label{l:lemma_num_state1}
  &
  \left(
    X_{ t_k } - Y_{ t_k }
  \right)
  \prod_{ i = 0 }^{ k - 1 }
  \left(
    1 - 
    ( t_{ i + 1 } - t_i ) 
    \tfrac{ 
      \mu( X_{ t_{ i + 1 } } ) - 
      \mu( Y_{ t_{ i + 1 } } ) 
    }{
      X_{ t_{ i + 1 } } - Y_{ t_{ i + 1 } } 
    }
  \right)
  \\& \nonumber
  =X_0-Y_0
  +\int_{0}^{t_k}\left(\mu(X_s)-\mu(X_{\lceil s\rceil_{\theta}})\right)
  \prod_{i\in\{0,\ldots,n-1\}\colon t_{i+1}\leq \lfloor s\rfloor_{\theta}}
   \left(1-(t_{i+1}-t_i)\tfrac{\mu(X_{t_{i+1}})-\mu(Y_{t_{i+1}})}{X_{t_{i+1}}-Y_{t_{i+1}}}\right)
   ds
\end{align}
$\P$-a.s.\
  for all  $k \in \{0,\dots,n \}$
and
\begin{align} 
\label{l:lemma_num:state2}
  \left| X_{t_k} - Y_{t_k} \right|
\leq  
  \left(
    \tfrac{ 1 }{ 1 - | \theta | L } 
  \right)^k 
    \left( 
      \left| X_0 - Y_0 \right|
      +  
      \int_0^{ t_k } 
      \left|
        \mu( X_s ) - \mu( X_{ \lceil s \rceil_{ \theta } } ) 
      \right|  
      ds 
    \right)
\end{align}
$\P$-a.s.\ for all $k\in\{0,\dots,n\}$.
\end{lemma}

\begin{proof} [Proof of Lemma~\ref{l:lemma_num}]
First of all, observe that
\eqref{l:lemma_num_X} and \eqref{l:lemma_num_Y_umformung} 
imply that
\begin{align} \nonumber \label{l:lemma_num2}
&\left(
    X_{t_{k+1}}-Y_{t_{k+1}} \right)
  \left(1-(t_{k+1}-t_k)\tfrac{\mu(X_{t_{k+1}})-\mu(Y_{t_{k+1}})}{X_{t_{k+1}}-Y_{t_{k+1}}}\right)
\\
&
=
X_{t_{k+1}}-Y_{t_{k+1}}
-
  \left(t_{k+1}-t_k \right) \left( \mu(X_{t_{k+1}})-\mu(Y_{t_{k+1}})\right)
\\  \nonumber
 & =
X_{t_{k}}-Y_{t_{k}}
+ \smallint_{t_k}^{t_{k+1}} \mu(X_s) \, ds  -   \mu(X_{t_{k+1}})\left(t_{k+1}-t_k \right)
 =
X_{t_{k}}-Y_{t_{k}}
+ \smallint_{t_k}^{t_{k+1}} \mu(X_s) - \mu(X_{\lceil s\rceil_{\theta}}) \, ds 
 \end{align}
$\P$-a.s.\ for all $k \in\{0,\dots,n\}$.  
Next we claim that
\begin{align} 
\label{l:lemma_num3}
  &\left(
    X_{t_j}-Y_{t_j}
  \right)
  \prod_{i=0}^{j-1}\left(1-(t_{i+1}-t_i)\tfrac{\mu(X_{t_{i+1}})-\mu(Y_{t_{i+1}})}{X_{t_{i+1}}-Y_{t_{i+1}}}\right)
  \\& \nonumber
  =X_0-Y_0
  +\int_{0}^{t_j}
  \left(\mu(X_s)-\mu(X_{\lceil s\rceil_{\theta}})\right)
  \prod_{i\in\{0,1,\ldots,n-1\}\colon t_{i+1}\leq \lfloor s\rfloor_{\theta}}
   \left(1-(t_{i+1}-t_i)\tfrac{\mu(X_{t_{i+1}})-\mu(Y_{t_{i+1}})}{X_{t_{i+1}}-Y_{t_{i+1}}}\right)
   ds
\end{align}
$\P$-a.s.\ for all $j \in \{1,\dots,n \}$. 
We now prove \eqref{l:lemma_num3} by induction on $ j \in \{1,2,\dots,n \} $.
The base case $j=1$ is \eqref{l:lemma_num2} with $k=0$. 
For the induction step we assume that \eqref{l:lemma_num3} holds 
for some $ j \in \{ 1, 2, \dots, n - 1 \} $.
Then \eqref{l:lemma_num2} shows that 
{\allowdisplaybreaks \begin{align} \label{l:lemma_num4} \nonumber
  &\left(
    X_{t_{j+1}}-Y_{t_{j+1}}
  \right)
  \prod_{i=0}^{j}\left(1-(t_{i+1}-t_i)\tfrac{\mu(X_{t_{i+1}})-\mu(Y_{t_{i+1}})}{X_{t_{i+1}}-Y_{t_{i+1}}}\right)
  \\& \nonumber
  =
 \left(
    X_{t_j}-Y_{t_j} + \smallint_{t_j}^{t_{j+1}} \mu(X_s) - \mu(X_{\lceil s\rceil_{\theta}}) \, ds 
  \right)
  \prod_{i=0}^{j-1}\left(1-(t_{i+1}-t_i)\tfrac{\mu(X_{t_{i+1}})-\mu(Y_{t_{i+1}})}{X_{t_{i+1}}-Y_{t_{i+1}}}\right)
  \\&
  =
\left(
  \smallint_{t_j}^{t_{j+1}} \mu(X_s) - \mu(X_{\lceil s\rceil_{\theta}}) \, ds 
  \right)
  \prod_{i=0}^{j-1}\left(1-(t_{i+1}-t_i)\tfrac{\mu(X_{t_{i+1}})-\mu(Y_{t_{i+1}})}{X_{t_{i+1}}-Y_{t_{i+1}}}\right)
\\ & \nonumber
\quad +
X_0-Y_0
  +\int_{0}^{t_j}\left(\mu(X_s)-\mu(X_{\lceil s\rceil_{\theta}})\right)
  \prod_{i\in\{0,\ldots,n-1\}\colon t_{i+1}\leq \lfloor s\rfloor_{\theta}}
   \left(1-(t_{i+1}-t_i)\tfrac{\mu(X_{t_{i+1}})-\mu(Y_{t_{i+1}})}{X_{t_{i+1}}-Y_{t_{i+1}}}\right)
   ds
\\ & \nonumber
  =X_0-Y_0
  +\int_{0}^{t_{j+1}}\left(\mu(X_s)-\mu(X_{\lceil s\rceil_{\theta}})\right)
  \prod_{i\in\{0,\ldots,n-1\}\colon t_{i+1}\leq \lfloor s\rfloor_{\theta}}
   \left(1-(t_{i+1}-t_i)\tfrac{\mu(X_{t_{i+1}})-\mu(Y_{t_{i+1}})}{X_{t_{i+1}}-Y_{t_{i+1}}}\right)
   ds
\end{align}}%
$\P$-a.s.\ 
Induction hence shows \eqref{l:lemma_num3}.
Identity~\eqref{l:lemma_num3} proves \eqref{l:lemma_num_state1}
and it thus remains to prove~\eqref{l:lemma_num:state2}.
For this note that
\eqref{l:lemma_num_state1}
and the estimate that
$
  \left( t_{ k + 1 } - t_k \right)
  \frac{
    \mu(x) - \mu(y)
  }{
    x - y 
  } 
  \leq 
  \left( t_{ k + 1 } - t_k \right) L 
  \leq \left| \theta \right| L < 1
$ 
for all 
$ k \in \{ 0, 1, \dots, n - 1 \} $,
$ x, y \in I $ 
together with the assumption that
$
  \sup_{ i \in \{ 0, 1, \ldots, n \} }
  \P\!\left[  
    X_{ t_i } \in \partial I
  \right] = 0
$
yield that
{\allowdisplaybreaks  \begin{align} \label{l:lemma_num5} \nonumber  
  &  
  \left| 
    X_{ t_k } - Y_{ t_k } 
  \right|
\leq 
  \left[
    \prod_{i=0}^{k-1} \left(1-(t_{i+1}-t_i)\tfrac{\mu(X_{t_{i+1}})-\mu(Y_{t_{i+1}})}{X_{t_{i+1}}-Y_{t_{i+1}}}\right)
  \right]^{ - 1 }  
\\ \nonumber & \cdot 
  \left( 
    \left| 
      X_0 - Y_0 
    \right| 
    +  
    \int_0^{ t_k } 
    \left| 
      \mu( X_s ) - 
      \mu( X_{ \lceil s \rceil_{\theta} } )
    \right|
    \prod_{
      i \in \{ 0, 1, \ldots, n - 1 \}
      \colon t_{ i + 1 }
      \leq 
      \lfloor s \rfloor_{\theta} 
    }
    \left(
      1 - ( t_{ i + 1 } - t_i ) 
      \tfrac{ 
        \mu( X_{ t_{ i + 1 } } ) - \mu( Y_{ t_{ i + 1 } } ) 
      }{
        X_{ t_{ i + 1 } } - Y_{ t_{ i + 1 } } 
      }
    \right)
    ds  
  \right)
\\ & = 
  \left[
    \prod_{ i = 0 }^{ k - 1 } 
    \left(
      1 - 
      ( t_{ i + 1 } - t_i )
      \tfrac{ 
        \mu( X_{ t_{ i + 1 } } )
        -
        \mu( Y_{ t_{ i + 1 } } )
      }{
        X_{ t_{ i + 1 } } - Y_{ t_{ i + 1 } } 
      }
    \right)
  \right]^{ - 1 }   
  \left| 
    X_0 - Y_0 
  \right| 
\\ \nonumber & +  
  \int_0^{ t_k } 
  \left| 
    \mu( X_s ) - \mu( X_{ \lceil s \rceil_{ \theta } } )
  \right|
  \left[
    \prod_{ 
      i \in \{ 0, 1, \ldots, k - 1 \} \colon 
      t_{ i + 1 } > \lfloor s \rfloor_{ \theta }
    }
    \left(
      1 - 
      ( t_{ i + 1 } - t_i )
      \tfrac{ 
        \mu( X_{ t_{ i + 1 } } ) - 
        \mu( Y_{ t_{ i + 1 } } ) 
      }{
        X_{ t_{ i + 1 } } - Y_{ t_{ i + 1 } }
      }
    \right)
  \right]^{ - 1 }
   ds  
\\ \nonumber & \leq 
  \left[
   \prod_{ i = 0 }^{ k - 1 }
   \left( 1 - | \theta | L \right) 
  \right]^{ - 1 }   
  \left| 
    X_0 - Y_0 
  \right|  
  +  
  \int_0^{ t_k } 
  \left| 
    \mu(X_s) -
    \mu( X_{\lceil s\rceil_{\theta}} )
  \right|
  \left[
    \prod_{i\in\{0,\ldots,k-1\}\colon t_{i+1}> \lfloor s\rfloor_{\theta}}
    \left(1-| \theta | L\right)
  \right]^{ - 1 }
  \, ds  
\\ \nonumber & \leq
  \left( 
    \tfrac{ 1 }{ 1 - | \theta | L }
  \right)^k 
  \left( 
    \left| X_0 - Y_0  \right|  
    +  
    \int_0^{ t_k } 
      \left| 
        \mu(X_s) - 
        \mu( X_{\lceil s\rceil_{\theta}} )
      \right|
    ds  
  \right)
\end{align} }
$\P$-a.s.\ for all $k\in\{0,\dots,n\}$. 
This finishes the proof of Lemma~\ref{l:lemma_num}.
\end{proof}

The proof of the following lemma is analogous
to the proof of Corollary 2.27  in~\cite{HutzenthalerJentzen2014Memoires}.
%
%
%
\begin{lemma}  
\label{l:phi}
Let $ d, m \in \N $, 
$ c, \tilde{c}, T \in [0,\infty) $, 
$ p \in [2,\infty) $, 
$
  O \in \mathcal{B}( \R^d )
$,
$ 
  \mu \in \mathcal{L}^0( O; \R^d )
$,
$
  \sigma \in \mathcal{L}^0( O; \R^{ d \times m } )
$,
let
$
  ( 
    \Omega, \mathcal{F}, \P, ( \mathcal{F}_t )_{ t \in [0,T] } 
  )
$
be a filtered probability space, 
let
$
  W \colon [0,T] \times \Omega \to \R^m
$
be a standard $ ( \mathcal{F}_t )_{ t \in [0,T] } $-Brownian motion,
let $ \xi \colon \Omega \to O $ be an
$ \mathcal{F}_0 \big/ \mathcal{B} (O) $-measurable function 
with
$
  \E \!\left[ \| \mu(\xi) \|^p_{ \R^d } \right] < \infty
$,
assume
\begin{equation} 
\begin{split} 
  \langle x - y , \mu(x) - \mu(y) \rangle_{ \R^d } 
  & \leq c \, \| x - y \|^2_{ \R^d }
  \\ 
    \langle x, \mu(x) \rangle_{ \R^d } + 
    \tfrac{ (p - 1) }{ 2 } 
    \| \sigma (x) \|_{ HS( \R^m, \R^d ) }^2 & \leq \tilde{c} \left( 1 + \|x\|^2_{ \R^d } \right)
\end{split} 
\end{equation} 
for all $x,y \in O$ and assume that 
for all $ t \in (0, \nicefrac{ 1 }{ c } ) $ 
it holds that 
the mapping $ O \ni x \to x - t \mu(x) \in \R^d $ is surjective.
Then there exists a unique family 
$
  Y^h \colon (\N_0 \cap [0, \nicefrac{ T }{ h } ]) \times \Omega\to O 
$, 
$
  h \in (0,T] \cap (0, \nicefrac{ 1 }{ c } )
$, 
of stochastic processes 
satisfying $ Y_0^h = \xi $ 
and
\begin{equation} \label{eq:def.Y}
  Y^h_n = Y^h_{ n - 1 } + \mu( Y^h_n ) \, h 
  + \sigma( Y_{ n - 1 }^h ) \left( W_{ n h } - W_{ ( n - 1 ) h } \right) 
\end{equation}
for all $ n \in \N \cap [0, \nicefrac{ T }{ h } ] $, $ h \in (0,T] \cap (0, \nicefrac{ 1 }{ c } ) $ 
and there exists a real number $\rho \in (0,\infty)$ such that for all 
$ q \in [ 0, \nicefrac{ p }{ 2 } ] $ 
it holds that
\begin{align} 
  \label{eq:moment.bound.implicit.Euler}
  \limsup_{ h \searrow 0 }
  \sup_{
    n \in \N_0 \cap [0, \nicefrac{ T }{ h } ] 
  } 
  \E\!\left[ 
    \left( 1 + \| Y^h_n \|^2_{ \R^d } \right)^q 
  \right] 
\leq 
  e^{ \rho T } 
  \,
  \E\!\left[ 
    \left( 1 + \|\xi\|^2_{ \R^d } \right)^q 
  \right] 
\end{align}
and that
  \begin{equation}  
    \sup_{h \in (0,T] \cap (0, \nicefrac{ 1 }{ 4c } ]}
    \sup_{ n \in (\N_0 \cap [0, \nicefrac{ T }{ h } ]) }
    \E\Big[
      \big\{
        1 + 
         \| Y_n^h \|_{\R^d}^2
      \big\}^{ q }
    \Big]
  \leq
    e^{ (4c+\rho) T }
    \cdot
    \E\Big[
      \big\{
        1 + \left(
        \| 
         \xi
        \|_{\R^d}
        + T
	\| 
         \mu(\xi)
        \|_{\R^d}
        \right)^2
      \big\}^{ q }
    \Big].
  \end{equation}
\end{lemma}

\begin{proof}[Proof 
  of 
  Lemma~\ref{l:phi}]
  Throughout this proof, let
  $
    F \colon (0,\nicefrac{1}{c}) \times O \to \R^d
  $ 
  be a function defined by
  $ F_h(x) := x - \mu(x) h $ 
  for all 
  $ (h,x) \in (0,\nicefrac{1}{c}) \times O $.
  Then 
  Lemma~\ref{l:inverse_function} and the surjectivity assumption
  imply
  for all $h\in(0,\nicefrac{1}{c})$ that the function $F_h$ is bijective.
  This ensures the unique 
  existence of stochastic
  processes
  $
    Y^h \colon (\N_0 \cap [0, \nicefrac{ T }{ h } ]) \times \Omega\to O 
  $, 
  $
    h \in (0,T] \cap (0, \nicefrac{ 1 }{ c } )
  $, 
  satisfying $ Y_0^h = \xi $ 
  and~\eqref{eq:def.Y}.
  In the next step, note that
  \begin{equation}
    \big\|
      F_{ h}( Y^h_{ n + 1 } )
    \big\|_{\R^d}^2 =
    \big\|
      Y_n^h
      +
      \sigma(Y_n^h )
      \big(
        W_{ ( n + 1 ) h }
        - W_{ n h }
      \big)
    \big\|_{\R^d}^2
  \end{equation}
  for all 
  $ 
    n \in (\N_0 \cap [0, \nicefrac{ T }{ h } ])
  $ and all
  $
    h \in (0,T] \cap (0, \nicefrac{ 1 }{ c } )
  $. 
  Lemma 2.26 in~\cite{HutzenthalerJentzen2014Memoires}
  hence implies the existence of 
  a real number $ \rho \in \mathbb{R} $
  such that
  \begin{equation}  \label{eq:cor.Lyapunov.implicit.Euler}
    \E\Big[
      \big\{
        1 + 
        \| 
          F_{ h }(
            Y_{n+1}^h
          )
        \|_{\R^d}^2
      \big\}^{ q } \,
      \big| \, Y_n^h
    \Big]
  \leq
    \exp\!\left(
      \rho h
    \right)
    \cdot
    \left\{
      1 + 
      \|
        F_{ h }(
          Y_n^h
        )
      \|_{\R^d}^2
    \right\}^{ q }
  \end{equation}
  $ \mathbb{P} $-a.s.\ for 
  all 
  $ 
    n \in (\N_0 \cap [0, \nicefrac{ T }{ h } ])
  $,
  $
    h \in (0,T] \cap (0, \nicefrac{ 1 }{ 4c } ]
  $ 
  and all $ q \in [0,\frac{p}{2}] $.
  Next fix a real number
  $ q \in [0,\frac{p}{2}] $
  and we now prove
  \eqref{eq:moment.bound.implicit.Euler}
  for this
  $ q \in [0,\frac{p}{2}] $.
  If 
  $ 
    \E\big[ 
      \| \xi \|_{\R^d}^{ 2 q } 
    \big] = \infty
  $,
  then \eqref{eq:moment.bound.implicit.Euler} 
  is trivial.
  So we assume
  $
    \E\big[
      \|\xi\|_{\R^d}^{ 2 q }
    \big]
    < \infty
  $ 
  for the rest of this proof.
  Hence, we obtain that
  $
    \E\big[
      \| \xi \|_{\R^d}^{ 2 q }
      +
      \| \mu(\xi) \|_{\R^d}^{ 2 q }
    \big]
    < \infty
  $.
  Now, for every 
  $
    h \in (0,T] \cap (0, \nicefrac{ 1 }{ 4c } ]
  $,
  we apply
  Corollary~2.2
  in~\cite{HutzenthalerJentzen2014Memoires}
  with the Lyapunov-type function
  $ 
    V \colon O \rightarrow [0,\infty) 
  $
  given by
  $
    V( x ) = 
    \big\{
      1 + \| F_{ h }( x ) \|_{\R^d}^2
    \big\}^{ q }
  $
  for all $ x \in O $,
  with the truncation function
  $
    \zeta \colon [0,\infty)
    \rightarrow (0,\infty]
  $
  given by
  $ 
    \zeta(t) = \infty
  $
  for all $ t \in [0,\infty) $
  and with the sequence
  $ t_n \in \R $, $ n \in \N_0 $,
  given by
  $ 
    t_n = \min( n h, T )
  $ 
  for all 
  $ n \in \N_0 $
  to obtain
  \begin{equation}    \label{eq:cor.Lyapunov.implicit.Euler2}
    \sup_{ n \in (\N_0 \cap [0, \nicefrac{ T }{ h } ]) }
    \E\Big[
      \big\{
        1 + 
        \| 
          F_{ h }( Y_{n}^h )
        \|_{\R^d}^2
      \big\}^{ q }
    \Big]
  \leq
    e^{ \rho T }
    \cdot
    \E\Big[
      \big\{
        1 +
        \| 
          F_{ h }( \xi  )
        \|_{\R^d}^2
      \big\}^{ q }
    \Big]
  \end{equation}
  for
  all 
  $
    h \in (0,T] \cap (0, \nicefrac{ 1 }{ 4c } ]
  $. 
  Lemma~2.25
  in~\cite{HutzenthalerJentzen2014Memoires}
  and the dominated convergence
  theorem  hence give
  \begin{equation}  
  \begin{split}
    \limsup_{ h \searrow 0 }
    \sup_{ n \in (\N_0 \cap [0, \nicefrac{ T }{ h } ]) }
    \E\Big[
      \big\{
        1 + \| Y_n^h \|_{\R^d}^2
      \big\}^{ q }
    \Big]
  & \leq
    \limsup_{ h \searrow 0 }
    \left(
    e^{ 4c h }
    \cdot
    e^{ \rho T }
    \cdot
    \E\Big[
      \big\{
        1 + 
        \| 
          F_{ h }( \xi )
        \|_{\R^d}^2
      \big\}^{ q }
    \Big]
    \right)
  \\ =
    e^{ \rho T }
    \cdot
    \E\Big[
      \lim_{(0,T]\ni h\to 0}
      \big\{
        1 + 
        \|
          F_{ h }( \xi )
        \|_{\R^d}^2
      \big\}^{ q }
    \Big]
  & =
    e^{ \rho T }
    \cdot
    \E\Big[
      \big\{ 
        1 +
        \| \xi \|_{\R^d}^2
      \big\}^{ q }
    \Big] 
  \end{split}.  
  \end{equation}
Furthermore, Lemma~2.25
  in~\cite{HutzenthalerJentzen2014Memoires} and \eqref{eq:cor.Lyapunov.implicit.Euler2} yield 
  \begin{equation}  
    \sup_{h \in (0,T] \cap (0, \nicefrac{ 1 }{ 4c } ]}
    \sup_{ n \in (\N_0 \cap [0, \nicefrac{ T }{ h } ]) }
    \E\Big[
      \big\{
        1 + 
         \| Y_n^h \|_{\R^d}^2
      \big\}^{ q }
    \Big]
  \leq
  e^{ 4c T } \cdot
    e^{ \rho T }
    \cdot
    \E\Big[
      \big\{
        1 + \left(
        \| 
         \xi
        \|_{\R^d}
        + T
	\| 
         \mu(\xi)
        \|_{\R^d}
        \right)^2
      \big\}^{ q }
    \Big]
  \end{equation}
  and this completes the proof of 
  Lemma~\ref{l:phi}.
\end{proof}

\begin{lemma}  
\label{l:mom_num}
Let $ d, m \in \N $, $n \in \N_0 $, 
$T, c \in [0,\infty) $,
 $\theta=(t_0,\dots, t_n) \in [0,T]^{n+1}$ satisfy $0=t_0< t_1 < \dots < t_n=T$ and $4c | \theta |  \leq 1$,
let
$
  O \in \mathcal{B}( \R^d )
$,
$ 
  \mu \in \mathcal{L}^0( O; \R^d )
$,
$
  \sigma \in \mathcal{L}^0( O; \R^{ d \times m } )
$,
let
$
  ( 
    \Omega, \mathcal{F}, \P, ( \mathcal{F}_t )_{ t \in [0,T] } 
  )
$
be a filtered probability space, 
let
$
  W \colon [0,T] \times \Omega \to \R^m
$
be a standard $ ( \mathcal{F}_t )_{ t \in [0,T] } $-Brownian motion,
let 
$
  Y \colon \{ t_0, t_1, \ldots, t_n \} \times \Omega \to O
$ 
be a measurable mapping satisfying 
  \begin{equation} \label{l:mom_num_ass1}
    Y_{t_{k+1}}= Y_{t_k} +  \mu( Y_{ t_{ k + 1 } } ) \left( t_{ k + 1 } - t_k \right) 
    + \sigma( Y_{ t_k } ) (W_{ t_{ k + 1 } } - W_{ t_k }  )
  \end{equation}
 $\P$-a.s.\ for all $k \in \{0,1,\dots,n-1 \}$ and
assume that for all $x \in O$ it holds that
$
  \langle x, \mu(x) \rangle_{\R^d} \leq c \left( 1 + \|x\|_{\R^d}^2 \right).
$
Then it holds for all $k \in \{0,\dots,n \}$, $p\in[2,\infty)$ that
\begin{equation} 
\begin{split}  \label{l:mom_num:state} 
 & \left\| \sup_{i \in \{0,1,\dots,k\} } \left[ 1 + \left\| Y_{t_i} \right\|_{\R^d}^2  \right] \right\|_{L^p( \Omega; \R )} 
\\ &
\leq 
e^{4ct_k} 
\left[ 
 1+ 
 \left\| Y_0 
\right\|_{L^{2p}( \Omega; \R^d )}^2 
\right]
+
\left[ \tfrac{2p^3}{p-1}
\sum_{i=0}^{k-1}
e^{8c(t_k - t_i)} 
\left\|
  \sigma(  Y_{t_i}  )^*  Y_{t_i}
\right\|_{L^{p}( \Omega; \R^m )}^2
  \left( t_{ i + 1 } - t_i \right)
\right]^{\nicefrac{1}{2}} 
\\ & \quad 
+
  p \left( 2 p - 1 \right)
  \left[
\sum_{i=0}^{k-1} 
 e^{4c(t_k - t_i)} 
\left\|
\sigma(  Y_{t_i} )
\right\|_{L^{2p}( \Omega; HS( \R^m, \R^d )  )}^2
  \left( t_{ i+1 } -  t_i \right) 
  \right]
  .
\end{split} 
\end{equation}
\end{lemma}

\begin{proof}[Proof of Lemma~\ref{l:mom_num}]

First of all, Lemma 2.25 in \cite{HutzenthalerJentzen2014Memoires} together with $4c | \theta |  \leq 1$ implies that for all $k \in \{1,\dots,n \}$ it holds that 
\begin{equation} 
\begin{split}  \label{l:mom_num:1} 
1 + \left\| Y_{t_k} \right\|_{\R^d}^2  
\leq
e^{4c(t_k - t_{k-1})} \left[ 1+ \left\| Y_{t_k} - \mu(Y_{t_k})(t_k - t_{k-1}) \right\|_{\R^d}^2 \right].
\end{split} 
\end{equation} 
This together with \eqref{l:mom_num_ass1} yields that for all $k \in \{1,\dots,n \}$ it holds that 
\begin{equation} 
\begin{split}  \label{l:mom_num:2} 
& 1 + \left\| Y_{t_k} \right\|_{\R^d}^2  
\leq
e^{4c(t_k - t_{k-1})} \left[ 1+ \left\|  Y_{t_{k-1}} + \sigma(  Y_{t_{k-1}} ) 
(W_{ t_{ k } } - W_{t_{k-1}}  )
 \right\|_{\R^d}^2 \right]
\\ 
&
=
e^{4c(t_k - t_{k-1})} 
\left[ 1+ \left\|  Y_{t_{k-1}} 
\right\|_{\R^d}^2 
+ 
2 \int_{t_{k-1}}^{t_k}
(Y_{t_{k-1}})^*  \sigma(  Y_{t_{k-1}} ) \, dW_s
+
\left\|
\sigma(  Y_{t_{k-1}} )
(W_{ t_{ k } } - W_{t_{k-1}}  )
 \right\|_{\R^d}^2 \right]
\end{split} 
\end{equation} 
$\P$-a.s.
A straight forward induction on  $k \in \{1,\dots,n-1 \}$ then shows that for all $k \in \{0,\dots,n \}$ it holds that 
\begin{equation} 
\begin{split}  \label{l:mom_num:3} 
1 + \left\| Y_{t_k} \right\|_{\R^d}^2  
\leq & 
e^{4ct_k} \left[ 1+ 
 \left\| Y_0 \right\|_{\R^d}^2 \right]
+
2 \int_{t_0}^{t_k}
e^{4c(t_k - \lfloor s \rfloor_{\theta})} 
(Y_{\lfloor s \rfloor_{\theta}} )^* \sigma(  Y_{\lfloor s \rfloor_{\theta}}  ) \, dW_s
\\ &
+
\sum_{i=0}^{k-1} 
 e^{4c(t_k - t_i)} 
\left\|
\sigma(  Y_{t_i} )
(W_{ t_{ i+1 } } - W_{t_i}  )
 \right\|_{\R^d}^2 
\end{split} 
\end{equation}
$\P$-a.s.
Estimate \eqref{l:mom_num:3} 
and the Burkholder-Davis-Gundy-type inequalities in 
Lemma~7.2 and Lemma~7.7 in Da Prato \& Zabczyk~\cite{dz92} ensure that for all $k \in \{0,\dots,n \}$, $p\in[2,\infty)$ it holds that
\begin{equation} 
\begin{split}  \label{l:mom_num:4} 
 & \left\| \sup_{i \in \{0,1,\dots,k\} } \left[ 1 + \left\| Y_{t_i} \right\|_{\R^d}^2  \right] \right\|_{L^p( \Omega; \R )} 
\\ &
\leq 
e^{4ct_k} 
\left\| 
 1+ 
 \left\| Y_0 \right\|_{\R^d}^2 
\right\|_{L^p( \Omega; \R )} 
+
2 
\left\|
\sup_{s \in [t_0, t_k]}
\left| 
\int_{t_0}^{s}
e^{4c(t_k - \lfloor s \rfloor_{\theta})} 
(Y_{\lfloor s \rfloor_{\theta}})^*  \sigma(  Y_{\lfloor s \rfloor_{\theta}}  ) \, dW_s \right|
\right\|_{L^p( \Omega; \R )} 
\\ & \quad 
+
 \left\|
\sum_{i=0}^{k-1} 
 e^{4c(t_k - t_i)} 
\left\|
 \sigma(  Y_{t_i} )
(W_{ t_{ i+1 } } - W_{t_i}  )
 \right\|_{\R^d}^2 
\right\|_{L^p( \Omega; \R )} 
\\ &
\leq 
e^{4ct_k} 
\left[ 
 1+ 
 \left\| Y_0 
\right\|_{L^{2p}( \Omega; \R^d )}^2 
\right]
+
2 \sqrt{\tfrac{p^3}{2(p-1)}
\int_{t_0}^{t_k}
e^{8c(t_k - \lfloor s \rfloor_{\theta})} 
\left\|
  \big(
    Y_{\lfloor s \rfloor_{\theta}} 
  \big)^*  
  \sigma(  Y_{\lfloor s \rfloor_{\theta}}  ) \right\|_{L^{p}( \Omega; HS( \R^m, \R )  )}^2 \, ds }
\\ & \quad 
+
\sum_{i=0}^{k-1} 
 e^{4c(t_k - t_i)} 
\left\|
\sigma(  Y_{t_i} )
(W_{ t_{ i+1 } } - W_{t_i}  )
\right\|_{L^{2p}( \Omega; \R^d )}^2
\\ &
\leq 
e^{4ct_k} 
\left[ 
 1+ 
 \left\| Y_0 
\right\|_{L^{2p}( \Omega; \R^d )}^2 
\right]
+
\left[ \tfrac{2p^3}{p-1} 
\sum_{i=0}^{k-1}
e^{8c(t_k - t_i)} 
\left\|
  \sigma(  Y_{t_i}  )^*  Y_{t_i}
\right\|_{L^{p}( \Omega; \R^m )}^2
(t_{i+1} - t_i)
\right]^{\nicefrac{1}{2}} 
\\ & \quad 
+
  p \left( 2 p - 1 \right)
  \left[
\sum_{i=0}^{k-1} 
 e^{4c(t_k - t_i)} 
\left\|
\sigma(  Y_{t_i} )
\right\|_{L^{2p}( \Omega; HS( \R^m, \R^d )  )}^2
  ( t_{ i+1 } -  t_i  ) 
  \right]
  .
\end{split} 
\end{equation}
  This finishes the proof
  of 
  Lemma~\ref{l:mom_num}.
\end{proof}
The next result, Corollary~\ref{l:uni_mom_num2}, proves, under suitable assumptions, uniform moment bounds
(see \eqref{l:mom_num:statebbb} below)
for a family of fully-drift implicit Euler approximations.
Corollary~\ref{l:uni_mom_num2} follows immediately from a combination of
Lemma~\ref{l:mom_num} and Lemma~\ref{l:phi}.
\begin{corollary}  
\label{l:uni_mom_num2}
Let $ d, m \in \N $, 
$ c, \bar{c}, \tilde{c}, T \in [0,\infty) $, 
$ p \in [2,\infty) $, 
$
  O \in \mathcal{B}( \R^d )
$,
$ 
  \mu \in \mathcal{L}^0( O; \R^d )
$,
$
  \sigma \in \mathcal{L}^0( O; \R^{ d \times m } )
$,
let
$
  ( 
    \Omega, \mathcal{F}, \P, ( \mathcal{F}_t )_{ t \in [0,T] } 
  )
$
be a filtered probability space, 
let
$
  W \colon [0,T] \times \Omega \to \R^m
$
be a standard $ ( \mathcal{F}_t )_{ t \in [0,T] } $-Brownian motion,
let $ \xi \colon \Omega \to O $ be an
$ \mathcal{F}_0 \big/ \mathcal{B} (O) $-measurable function 
with
$
   \big\| 
     \| \xi \|_{ \R^d }
     +
     \| \mu( \xi ) \|_{ \R^d }
   \big\|_{ L^{(\bar{c} + 1 \vee \bar{c})p}( \Omega; \R ) } 
  < \infty
$,
assume
\begin{equation} 
\begin{split} 
   \| \sigma (x) \|_{ HS( \R^m, \R^d ) }
  & \leq  \tilde{c} \left( 1 + \|x\|^{\bar{c}}_{ \R^d } \right)
  \\ 
  \langle x - y , \mu(x) - \mu(y) \rangle_{ \R^d } 
  & \leq c \, \| x - y \|^2_{ \R^d }
  \\ 
    \langle x, \mu(x) \rangle_{ \R^d } + 
    \tfrac{ ({(\bar{c} + 1 \vee \bar{c})p} - 1) }{ 2 } 
    \| \sigma (x) \|_{ HS( \R^m, \R^d ) }^2 & \leq \tilde{c} \left( 1 + \|x\|^2_{ \R^d } \right)
\end{split} 
\end{equation} 
for all $x,y \in O$ and assume that 
for all $ t \in (0, \nicefrac{ 1 }{ c } ) $ 
it holds that 
the mapping $ O \ni x \to x - t \mu(x) \in \R^d $ is surjective.
Then there exists a unique family 
$
  Y^h \colon (\N_0 \cap [0, \nicefrac{ T }{ h } ]) \times \Omega\to O 
$, 
$
  h \in (0,T] \cap (0, \nicefrac{ 1 }{ c } )
$, 
of stochastic processes 
satisfying $ Y_0^h = \xi $ 
and
\begin{equation} \label{l:uni_mom_num2.Y}
  Y^h_n = Y^h_{ n - 1 } + \mu( Y^h_n ) \, h 
  + \sigma( Y_{ n - 1 }^h ) \left( W_{ n h } - W_{ ( n - 1 ) h } \right) 
\end{equation}
for all $ n \in \N \cap [0, \nicefrac{ T }{ h } ] $, $ h \in (0,T] \cap (0, \nicefrac{ 1 }{ c } ) $ 
and 
it holds that
\begin{equation} 
\begin{split}  \label{l:mom_num:statebbb}  
\sup_{ h \in (0,T] \cap (0,\frac{1}{4c}]} 
\left\|
 \sup_{ n \in \N_0 \cap [0,T/h]}  \big\| Y^{h}_{n} \big\|_{ \R^d }  \right\|_{ L^{p}( \Omega; \R ) }
< \infty.
\end{split} 
\end{equation}
\end{corollary}
%
%
%
%
%
%
%
%
%

%
%
%

\subsection{Strong convergence rates for drift-implicit Euler approximations of Bessel-type processes}

%

%
%
The following corollary establishes strong convergence rates
for drift-implicit Euler approximations of Bessel-type processes.
\begin{corollary}   \label{l:num_corollary}
Let $ T, L, c \in [0,\infty) $, 
$ \varepsilon \in (0,\infty) $,
$ \gamma \in (0,1] $,
$
  \mu \in \mathcal{L}^0([0,\infty); \R)
$ ,
$
  n \in \N_0
$, 
$
  \theta=(t_0,\dots, t_n) \in [0,T]^{n+1}
$ 
satisfy 
$
  0 = t_0 < t_1 < \dots < t_n = T 
$, 
$
  L | \theta | < 1
$
and
$
  ( x-y) (\mu(x) - \mu(y)) \leq L (x-y)^2 
$ 
and
$
  |
    \mu(x) - \mu(y) 
  | 
  \leq 
  c
  \left| x - y \right|
  \big(
    \frac{ 1 }{ x } + \frac{ 1 }{ y } 
    + 
    \frac{ 1 }{ x y } + 
    x^{ c } + 
    y^{ c } 
  \big)
$ 
for all $x, y \in (0,\infty)$,
let
$
  ( 
    \Omega, \mathcal{F}, \P, ( \mathcal{F}_t )_{ t \in [0,T] } 
  )
$
be a stochastic basis, let
$
  W \colon [0,T] \times \Omega \to \R
$
be a standard $ ( \mathcal{F}_t )_{ t \in [0,T] } $-Brownian motion, 
let $X\colon[0,T]\times \Omega\to [0,\infty)$
  be an adapted stochastic process with continuous sample paths satisfying 
  $\int_0^T |\mu(X_s)| \,ds < \infty$ $ \P $-a.s., $\sup_{s \in [0,T]} \P [X_s=0]=0$
  and
  \begin{align} 
  \label{l:lemma_num_Xbbb}
    X_t=X_0+\int_0^t \mu(X_s)\,ds+W_t
  \end{align}
 $\P$-a.s.\
  for all $t\in[0,T]$,
let 
$
  Y \colon \{ t_0, t_1, \ldots, t_n \} \times \Omega \to [0,\infty)
$ 
be a measurable mapping satisfying $Y_0= X_0$ and
  \begin{equation} \label{l:lemma_num_Y_umformungbbb}
    Y_{t_{k+1}}= Y_{t_k} +  \mu( Y_{ t_{ k + 1 } } ) \left( t_{ k + 1 } - t_k \right) 
    + W_{ t_{ k + 1 } } - W_{ t_k }  
  \end{equation}
 $\P$-a.s.\ for all $k \in \{0,\dots,n-1 \}$.
Then 
\begin{equation} 
\begin{split}  \label{l:num_corollary:3bbb} 
&\left\| \sup_{k \in \{0,1,\dots,n\}  }  \left|X_{t_k} - Y_{t_k} \right| \right\|_{ L^{1}( \Omega; \R ) }
 \leq
 2 T c  \left( \tfrac{1}{1-|\theta|L} \right)^{n}
\Big\| 
\sup_{u \in [0,T]}   
\left| X_u - X_{\lceil u \rceil_{\theta}} \right|
 \Big\|_{ L^{\gamma (1 +\nicefrac{1}{\eps})}( \Omega; \R ) }^{\gamma}
\\ & \quad \cdot 
\sup_{u\in[0,T]} \left[
\big\| 
 \tfrac{2}{(X_u)^{1+\gamma}} 
\big\|_{ L^{1+\eps}( \Omega; \R ) } 
+  \tfrac{\gamma}{1+\gamma}  
\big\| (X_u)^{\nicefrac{1}{\gamma} -\gamma}  
\big\|_{ L^{1+\eps}( \Omega; \R ) } 
  + 2
\big\|
\left( X_u \right)^{ c + 1 - \gamma }
\big\|_{ L^{1+\eps}( \Omega; \R ) } 
+ \tfrac{1}{1+\gamma} 
 \right].
\end{split} 
\end{equation} 

\end{corollary}

\begin{proof} [Proof of Corollary~\ref{l:num_corollary}]
Observe that 
Young's inequality implies 
that for all 
$
  x, y \in (0,\infty)
$
it holds that
 {\allowdisplaybreaks 
\begin{equation} 
\begin{split}   
\label{l:num_corollary:1} 
 &
  \left( x + y \right)^{ (1 - \gamma) } 
  \left( 
    \tfrac{ 1 }{ x } + 
    \tfrac{ 1 }{ y } + 
    \tfrac{ 1 }{ x y }
  \right)
=
 \tfrac{x+y}{x(x+y)^{\gamma}} + \tfrac{x+y}{y(x+y)^{\gamma}} + \tfrac{x+y}{xy(x+y)^{\gamma}} 
=
\tfrac{2 + \frac{x}{y} + \frac{y}{x}}{(x+y)^{\gamma}} 
+
\tfrac{ \frac{1}{y} + \frac{1}{x}}{(x+y)^{\gamma}} 
\\ 
&
\leq
\tfrac{1}{x^{\gamma}} + \tfrac{1}{y^{\gamma}} 
+
\tfrac{x^{1-\gamma}}{y} + \tfrac{y^{1-\gamma}}{x}
+
\tfrac{1}{y^{1+\gamma}} + \tfrac{1}{x^{1+\gamma}}
\\ 
&
\leq
\tfrac{\gamma}{1+\gamma} \tfrac{1}{x^{1+\gamma}} + \tfrac{1}{1+\gamma}
+
\tfrac{\gamma}{1+\gamma} \tfrac{1}{y^{1+\gamma}} + \tfrac{1}{1+\gamma}
+
\tfrac{1}{1+\gamma} \tfrac{1}{y^{1+\gamma}} + \tfrac{\gamma}{1+\gamma} x^{\frac{(1-\gamma)(1+\gamma)}{\gamma} }
+
\tfrac{1}{1+\gamma} \tfrac{1}{x^{1+\gamma}} + \tfrac{\gamma}{1+\gamma} y^{\frac{(1-\gamma)(1+\gamma)}{\gamma} }
\\ & \quad 
+
 \tfrac{1}{y^{1+\gamma}} +  \tfrac{1}{x^{1+\gamma}} 
\\
&
=
\tfrac{2}{x^{1+\gamma}} + \tfrac{2}{y^{1+\gamma}} +   \tfrac{\gamma}{1+ \gamma} x^{\frac{1}{\gamma} -\gamma} +   \tfrac{\gamma}{1+ \gamma}  y^{\frac{1}{\gamma} -\gamma}  + \tfrac{2}{1+\gamma}
\end{split} 
\end{equation}}%
and
\begin{equation} 
\begin{split}  \label{l:num_corollary:1bb} 
&(x+y)^{1-\gamma} \left( x^c + y^c \right)
=
\tfrac{(x+y)x^c}{(x+y)^{\gamma}} + \tfrac{(x+y)y^c}{(x+y)^{\gamma}}
\leq
x^{c+1-\gamma} + x^{c} y^{1-\gamma} + y^{c+1-\gamma} + y^{c} x^{1-\gamma}
\\ &
\leq
x^{c+1-\gamma} 
+ y^{1-\gamma + c} \tfrac{1-\gamma}{1-\gamma + c} + x^{1-\gamma + c} \tfrac{c}{1-\gamma+c}
+y^{c+1-\gamma} 
+ x^{1-\gamma + c} \tfrac{1-\gamma}{1-\gamma + c} + y^{1-\gamma + c} \tfrac{c}{1-\gamma+c}
\\ &
=
2x^{c+1-\gamma} + 2y^{c+1-\gamma}.
\end{split} 
\end{equation}
Moreover, observe that 
Lemma~\ref{l:lemma_num}, 
the assumption that $ \sup_{ s \in [0,T] } \P\big[ X_s = 0 \big] = 0 $ 
and H\"older's inequality
show that 
{\allowdisplaybreaks  
\begin{align}  \label{l:num_corollary:2} \nonumber
&\left\| \sup_{k \in \{0,1,\dots,n\}  }  \left|X_{t_k} - Y_{t_k} \right| \right\|_{ L^{1}( \Omega; \R ) }
\leq  \left\|  \sup_{k \in \{0,1,\dots,n\}  } \left[ \left( \tfrac{1}{1-|\theta|L} \right)^{k}  \int_0^{t_k} \left| \mu(X_s) -  \mu(X_{\lceil s \rceil_{\theta}}) \right| \, ds \right] \right\|_{ L^{1}( \Omega; \R ) } 
\\ &  \nonumber
= \left\|   \left( \tfrac{1}{1-|\theta|L} \right)^{n}  \int_0^{T} \left| \mu(X_s) -  \mu(X_{\lceil s \rceil_{\theta}}) \right| \1_{ \{ X_s, X_{\lceil s \rceil_{\theta}}) \in (0,\infty) \}}  \, ds  \right\|_{ L^{1}( \Omega; \R ) } 
 \\ \nonumber
 &\leq
 c\left( \tfrac{1}{1-|\theta|L} \right)^{n}\left\|  \int_0^{T}
  \left| X_s - X_{\lceil s \rceil_{\theta}} \right|
  \left(
    \tfrac{ 1 }{ X_s } + \tfrac{ 1 }{ X_{\lceil s \rceil_{\theta}} } 
    + 
    \tfrac{ 1 }{ X_s X_{\lceil s \rceil_{\theta}} } + 
    (X_s)^{ c } + 
    (X_{\lceil s \rceil_{\theta}})^{ c } 
  \right)
 \, ds \right\|_{ L^{1}( \Omega; \R ) }
%
%
%
 \\ \nonumber
 &\leq
 c\left( \tfrac{1}{1-|\theta|L} \right)^{n}\Bigg\| 
\sup_{u \in [0,T]}   
\left| X_u - X_{\lceil u \rceil_{\theta}} \right|^{\gamma}
\int_0^{T} 
\left| X_s + X_{\lceil s \rceil_{\theta}} \right|^{1-\gamma}
\\ \nonumber & \quad \cdot  
\left(
    \tfrac{ 1 }{ X_s } + \tfrac{ 1 }{ X_{\lceil s \rceil_{\theta}} } 
    + 
    \tfrac{ 1 }{ X_s X_{\lceil s \rceil_{\theta}} } + 
    (X_s)^{ c } + 
    (X_{\lceil s \rceil_{\theta}})^{ c } 
  \right)
 \, ds \Bigg\|_{ L^{1}( \Omega; \R ) }
 \\ \nonumber
 &\leq
 c\left( \tfrac{1}{1-|\theta|L} \right)^{n}
\Big\| 
\sup_{u \in [0,T]}   
\left| X_u - X_{\lceil u \rceil_{\theta}} \right|^{\gamma}
 \Big\|_{ L^{1+\nicefrac{1}{\eps}}( \Omega; \R ) }
\\ \nonumber & \quad \cdot 
\Bigg\| 
\int_0^{T} 
\left| X_s + X_{\lceil s \rceil_{\theta}} \right|^{1-\gamma} 
\left(
    \tfrac{ 1 }{ X_s } + \tfrac{ 1 }{ X_{\lceil s \rceil_{\theta}} } 
    + 
    \tfrac{ 1 }{ X_s X_{\lceil s \rceil_{\theta}} } + 
    (X_s)^{ c } + 
    (X_{\lceil s \rceil_{\theta}})^{ c } 
  \right)
 \, ds \Bigg\|_{ L^{1+\eps}( \Omega; \R ) }
 \\ \nonumber
 &\leq
 c\left( \tfrac{1}{1-|\theta|L} \right)^{n}
\Big\| 
\sup_{u \in [0,T]}   
\left| X_u - X_{\lceil u \rceil_{\theta}} \right|
 \Big\|_{ L^{\gamma (1+ \nicefrac{1}{\eps})}( \Omega; \R ) }^{\gamma}
\\  & \quad \cdot 
\int_0^{T} \big\| 
\left| X_s + X_{\lceil s \rceil_{\theta}} \right|^{1-\gamma} 
\left(
    \tfrac{ 1 }{ X_s } + \tfrac{ 1 }{ X_{\lceil s \rceil_{\theta}} } 
    + 
    \tfrac{ 1 }{ X_s X_{\lceil s \rceil_{\theta}} } + 
    (X_s)^{ c } + 
    (X_{\lceil s \rceil_{\theta}})^{ c } 
  \right)
\big\|_{ L^{1+\eps}( \Omega; \R ) } \, ds.
\end{align}}%
Inequality~\eqref{l:num_corollary:2}, inequality~\eqref{l:num_corollary:1} and  
inequality~\eqref{l:num_corollary:1bb}
imply that 
\begin{equation} 
\begin{split}  \label{l:num_corollary:3ccc} 
&\left\| \sup_{k \in \{0,1,\dots,n\}  }  \left|X_{t_k} - Y_{t_k} \right| \right\|_{ L^{1}( \Omega; \R ) }
 \\
 &\leq
 c\left( \tfrac{1}{1-|\theta|L} \right)^{n}
\Big\| 
\sup_{u \in [0,T]}   
\left| X_u - X_{\lceil u \rceil_{\theta}} \right|
 \Big\|_{  L^{\gamma (1+ \nicefrac{1}{\eps})}( \Omega; \R )  }^{\gamma} 
\int_0^{T} \Big\| 
 \tfrac{2}{( X_{\lceil s \rceil_{\theta}})^{1+\gamma}} + \tfrac{2}{(X_s)^{1+\gamma}} 
\\ & \quad 
+  \tfrac{\gamma}{1+\gamma}   ( X_{\lceil s \rceil_{\theta}})^{\frac{1}{\gamma} -\gamma} 
+ \tfrac{\gamma}{1+\gamma} (X_s)^{\frac{1}{\gamma} -\gamma}  + \tfrac{2}{1+\gamma} 
  + 2 \left(  X_{\lceil s \rceil_{\theta}} \right)^{ c + 1 - \gamma }
  + 2 \left( X_s \right)^{ c + 1 - \gamma }
\Big\|_{ L^{1+\eps}( \Omega; \R ) } \, ds 
%
%
%
%
 \\
 &\leq
 2 T c  \left( \tfrac{1}{1-|\theta|L} \right)^{n}
\Big\| 
\sup_{u \in [0,T]}   
\left| X_u - X_{\lceil u \rceil_{\theta}} \right|
 \Big\|_{  L^{\gamma (1+ \nicefrac{1}{\eps})}( \Omega; \R )  }^{\gamma}
\\ & \quad \cdot 
\sup_{u\in[0,T]} \left[
\big\| 
 \tfrac{2}{(X_u)^{1+\gamma}} 
\big\|_{ L^{1+\eps}( \Omega; \R ) } 
+  \tfrac{\gamma}{1+\gamma}  
\big\| (X_u)^{\frac{1}{\gamma} -\gamma}  
\big\|_{ L^{1+\eps}( \Omega; \R ) } 
  + 2
\big\|
\left( X_u \right)^{ c + 1 - \gamma }
\big\|_{ L^{1+\eps}( \Omega; \R ) } 
+ \tfrac{1}{1+\gamma} 
 \right].
\end{split} 
\end{equation} 
This finishes the proof of Corollary~\ref{l:num_corollary}.
\end{proof}
\begin{remark} \label{remark.L1.to.Lp}
In the setting of Corollary~\ref{l:num_corollary}, H\"older's inequality implies that for all  
$p \in [1,\infty)$, $\kappa \in [0,\infty]$
it holds that
\begin{align} \label{l:num_corollary:22}  \nonumber
& \left\| \sup_{k \in \{0,1,\dots,n\}  }  \left|X_{t_k} - Y_{t_k} \right| \right\|_{ L^{p}( \Omega; \R ) }
=
\left\| \sup_{k \in \{0,1,\dots,n\}  } \Big[ \left|X_{t_k} - Y_{t_k} \right|^{\frac{1}{p(1+\kappa)}} |X_{t_k} - Y_{t_k}|^{1-\frac{1}{p(1+\kappa)}} \Big] \right\|_{ L^{p}( \Omega; \R ) }
\\ 
& \leq
\left\| \sup_{k \in \{0,1,\dots,n\}  } |X_{t_k} - Y_{t_k}|^{\frac{1}{p(1+\kappa)}}  \right\|_{ L^{p(1+\kappa)}( \Omega; \R ) }
\left\| \sup_{k \in \{0,1,\dots,n\}  }  |X_{t_k} - Y_{t_k}|^{1-\frac{1}{p(1+\kappa)}} \right\|_{ L^{p(1+\nicefrac{1}{\kappa})}( \Omega; \R ) }
\\ \nonumber
& \leq 
\left\| \sup_{k \in \{0,1,\dots,n\}  }  \left| X_{t_k} - Y_{t_k} \right| \right\|_{ L^{ 1}( \Omega; \R ) }^{\frac{1}{p(1+\kappa)}}
\left[ 
\left\| \sup_{k \in \{0,1,\dots,n\}  } \left| X_{t_k}  \right| \right\|_{ L^{p + \frac{p-1}{\kappa}}( \Omega; \R )}^{1-\frac{1}{p(1+\kappa)}} 
+
\left\| \sup_{k \in \{0,1,\dots,n\}  } \left| Y_{t_k} \right| \right\|_{ L^{p + \frac{p-1}{\kappa}}( \Omega; \R )}^{1-\frac{1}{p(1+\kappa)}}\right].
\end{align}
\end{remark}

\subsection{Strong convergence rates for drift-implicit square root
Euler approximations of Cox-Ingersoll-Ross-type processes}

In the next two lemmas we present elementary 
properties of the drift coefficient of the transformed SDE
derived in Lemma~\ref{l:transformierte_gleichung}.

\begin{lemma}   \label{l:num_bound}
Let  $\mu  \in
  C^1\!\left([0,\infty) , \R \right)$,
  $\sigma \in
  C\!\left([0,\infty) , [0,\infty) \right)$ satisfy $\sigma(0)=0$, $\sigma((0,\infty)) \subseteq (0,\infty)$, $\sigma \cdot \sigma  \in
  C^2\!\left([0,\infty) , [0,\infty) \right)$, $\mu(0) > \tfrac{(\sigma \cdot \sigma)'(0)}{4} > 0$ and $\smallint_1^{\infty} \tfrac{1}{\sigma(z)}\,  dz = \infty$,
  let $\phi\colon[0,\infty) \to [0,\infty)$ be the function defined by
   $ \phi(y) := \int_{0}^y \tfrac{1}{\sigma(z)} \, dz$
  for all $y \in [0,\infty)$ and
let $g \colon (0, \infty) \to \R $ be the function defined by 
 \begin{align} 
   g(x)
:=
\left( \tfrac{\mu - \frac{1}{4} (\sigma \cdot \sigma)'}{\sigma} \right)  \! \left( \phi^{-1}(x) \right)
\end{align} 
for all $x \in (0,\infty)$. 
Then $\phi$ is well-defined and bijective, $g$ is well-defined and continuously differentiable, $\lim_{(0,\infty) \ni x \to 0} \tfrac{ \sigma (x)}{\sqrt{x}}
=
\sqrt{(\sigma \cdot \sigma)'(0)} \in (0,\infty)$,
\begin{align}
 \lim_{(0,\infty) \ni x \to 0} x g(x) = \tfrac{2 \mu(0)}{(\sigma \cdot \sigma)'(0)} - \tfrac{1}{2}
 \quad \text{   and   } \quad
  \limsup_{(0,\infty) \ni x \to 0}  \left|g'(x) +  \tfrac{1}{x^2}\left( \tfrac{2 \mu(0)}{(\sigma \cdot \sigma)'(0)} - \tfrac{1}{2} \right) \right|
  < \infty.
\end{align}

%
%
%
%
%
%
  %
  %
  %
  %
  %
%
%
%

\end{lemma}

\begin{proof} [Proof of Lemma~\ref{l:num_bound}]
First of all, we define 
$
  \alpha :=\tfrac{2 \mu(0)}{(\sigma \cdot \sigma)'(0)} - \tfrac{1}{2} \in (0,\infty)
$.
Next note that 
the assumptions 
$
  \sigma \cdot \sigma  \in
  C^2\!\left([0,\infty) ,[0,\infty) \right)
$ 
and 
$
  \sigma((0,\infty)) \subseteq (0,\infty)
$ 
ensure that 
$
  \sigma|_{(0,\infty)}   
  \in
  C^1\!\left((0,\infty) ,[0,\infty) \right)
$.
Furthermore,
observe that
the assumptions
$
  \mu \in
  C^1\!\left([0,\infty) ,\R \right)
$,
$
  \sigma \cdot \sigma  \in
  C^2\!\left([0,\infty) ,[0,\infty) \right)
$ 
and 
$
  \mu(0) > \tfrac{(\sigma \cdot \sigma)'(0)}{4} > 0
$
imply that there exist real numbers 
$\eps \in (0,1)$ and $c_1 \in (0,\infty)$, which we fix for the rest of this proof, such that for all $x \in (0, \eps)$ it holds that 
\begin{align} \label{l:num_bound:1}
0< \tfrac{1}{4} (\sigma \cdot \sigma)'(0)
 \leq
 (\sigma \cdot \sigma)'(0)-c_1 x
 \leq
 (\sigma \cdot \sigma)'(x)
 &\leq
 (\sigma \cdot \sigma)'(0) + c_1 x
 \qquad\text{ and}
 \\
\label{l:num_bound:2} 
0 <  \tfrac{\mu(0)}{2} - \tfrac{(\sigma \cdot \sigma)'(0)}{8}
 \leq
 \mu(0) - \tfrac{(\sigma \cdot \sigma)'(0)}{4} - c_1 x
 \leq
 \mu(x) - \tfrac{(\sigma \cdot \sigma)'(x)}{4} 
 &\leq
 \mu(0) - \tfrac{(\sigma \cdot \sigma)'(0)}{4} + c_1 x.
\end{align}
In the next step we observe that the assumption
$ \sigma(0) = 0 $ shows that 
\begin{align} \label{l:num_bound:4}
\lim_{(0,\infty) \ni x \to 0} \tfrac{ \sigma (x)}{\sqrt{x}}
=
\sqrt{\lim_{(0,\infty) \ni x \to 0}  \tfrac{(\sigma \cdot \sigma)(x)}{x}}
=\sqrt{(\sigma \cdot \sigma)'(0)} \in (0,\infty).
\end{align}
Hence, there exist real numbers $\delta \in (0,1)$ and $\lambda \in (0,\infty)$ such that for all $x\in (0,\delta)$ it holds  that $\sigma(x) \geq \lambda \sqrt{x}$. 
This, the continuity of $ \sigma $ 
and the assumption that $ \sigma((0,\infty)) \subseteq (0,\infty) $ 
imply that for all $ y \in (0,\infty) $ 
it holds that
\begin{align} 
\label{l:num_bound:222222}
  \int_0^y 
  \tfrac{ 1 }{ \sigma(z) } 
  \, dz 
\leq
  \int_0^{ \delta } 
  \tfrac{ 1 }{ \sigma(z) } 
  \, dz  
  + 
  \left| 
    \int_{\delta}^y 
    \tfrac{ 1 }{ \sigma(z) } 
    dz 
  \right|
\leq 
 \int_0^{\delta} \tfrac{1}{\lambda \sqrt{z}} \, dz  +  \left| \int_{\delta}^y \tfrac{1}{\sigma(z)}\,  dz\right|  < \infty
 .
\end{align}
This ensures that $ \phi $ is well-defined. 
Furthermore, observe that $ \phi $ is strictly increasing and continuous and 
note that 
$
  \phi|_{(0,\infty)} \in
  C^1\!\left((0,\infty) , [0,\infty) \right)
$. 
This and the assumption that 
$
  \smallint_1^{\infty} \tfrac{1}{\sigma(z)} \, dz = \infty
$ 
show that 
$ \phi \colon [0,\infty) \to [0,\infty) $
is bijective,
that  
$
  \phi^{ - 1 } \colon [0,\infty) \to [0,\infty)
$ 
is also strictly increasing and continuous 
and that 
$
  \phi^{-1}|_{(0,\infty)} \in
  C^1\!\left( (0,\infty) , [0,\infty) \right)
$. 
In the next step we observe that fact that
for all $ x \in (0,\infty) $ 
it holds that
$ \sigma( \phi^{ - 1 }( x ) ) > 0 $ 
ensures that $ g $ is well-defined.
Furthermore, we note that l'Hospital's rule, 
the fact that $ \phi(0) = 0 $
and \eqref{l:num_bound:4} show that
 \begin{align} \label{l:num_bound:5}
\lim_{(0,\infty) \ni x \to 0} \tfrac{\phi(x)}{\sqrt{x}}
&=
\lim_{(0,\infty) \ni x \to 0} 2 \phi'(x) \sqrt{x} 
= \lim_{(0,\infty) \ni x \to 0} \tfrac{2 \sqrt{x}}{\sigma(x)}
=
\tfrac{2}{\sqrt{(\sigma \cdot \sigma)'(0)}} \in (0,\infty).
\end{align} 
In addition, observe that \eqref{l:num_bound:4}, \eqref{l:num_bound:5}
and the identity $ \phi^{-1}(0) = \phi(0) = 0 $ imply that
 \begin{align} \label{l:num_bound:6} 
\lim_{(0,\infty) \ni x \to 0} \tfrac{x}{\sigma(\phi^{-1}(x))}
&=
\lim_{(0,\infty) \ni x \to 0} \tfrac{\phi(x)}{\sigma(x)}
=
\lim_{(0,\infty) \ni x \to 0} \tfrac{\phi(x)}{\sqrt{x}}
\cdot
\lim_{(0,\infty) \ni x \to 0} \tfrac{\sqrt{x}}{\sigma(x)}
=
\tfrac{2}{(\sigma \cdot \sigma)'(0)} \quad \in (0,\infty).
\end{align} 
This shows that
\begin{equation} 
\begin{split} 
 \label{l:num_bound:7} 
&
  \lim_{(0,\infty) \ni x \to 0} 
  \left[ 
    x \cdot g(x)
  \right] 
  =
\lim_{(0,\infty) \ni x \to 0} 
  \left[ 
    x 
    \cdot \left( \tfrac{\mu - \frac{1}{4} (\sigma \cdot \sigma)'}{\sigma} \right)\!\left( \phi^{-1}(x) \right)
  \right] 
\\
&
=
  \left[
    \lim_{ (0,\infty) \ni x \to 0 } 
    \left( 
      \tfrac{ x }{ \sigma\left( \phi^{-1}(x) \right)
      } 
    \right)
  \right] 
  \left[ 
    \lim_{ (0,\infty) \ni x \to 0 } 
    \left( \mu(x) - \tfrac{1}{4} (\sigma \cdot \sigma)'(x) \right)
  \right]
=
 \tfrac{2\mu(0)}{(\sigma \cdot \sigma)'(0)}  - \tfrac{1}{2} = \alpha.
\end{split} 
\end{equation} 
Moreover, 
observe 
that identities 
$
  (\sigma \cdot \sigma)'(x) = 2 \sigma(x) \sigma'(x)
$ 
and 
$
  \big( \phi^{-1} \big)' \big(\phi(x)\big) =  \sigma(x)
$
for all $ x \in (0,\infty) $ 
imply that for all $ y \in (0,\infty) $ 
it holds that
\begin{align} \label{l:num_bound:8} 
 g'\big( \phi( y ) \big)
&=
  \left(
    \tfrac{
      \left( \mu' - \frac{(\sigma \cdot \sigma)''}{4}\right) \sigma  
      -
      \left( \mu - \frac{(\sigma \cdot \sigma)'}{4}\right) \sigma' 
    }{
      \sigma \cdot \sigma
    } 
  \right)\!( y ) 
  \cdot
  \sigma( y )
=
  \mu'(y) - \tfrac{(\sigma \cdot \sigma)''(y)}{4}
  -
  \left(
    \tfrac{
      \left( 
        \mu - \frac{(\sigma \cdot \sigma)'}{4}\right)(\sigma \cdot \sigma)'
    }{
      2 (\sigma \cdot \sigma)
    } 
  \right)\!( y )
  .
 \end{align}
In the next step we define
a real number 
$
  c_2
  := 
  \tfrac{ 8 c_1 }{ 9 }
  \big( 
    ( \sigma \cdot \sigma )'( 0 ) 
  \big)^{ - 3 / 2 }
  \in (0, \infty)
  .
$
The assumption that $ \sigma(0) = 0 $ 
and estimate \eqref{l:num_bound:1}
show then that for all $y \in (0,\eps)$ it holds that
\begin{align}  \label{l:num_bound:9}  \nonumber
0 \leq \phi(y) 
& = 
 \int_{0}^y \tfrac{1}{\sigma(z)} \, dz
 =
 \int_{0}^y \tfrac{1}{\sqrt{ \int_{0}^z {(\sigma \cdot \sigma)'(x)} \, dx }}\,  dz
\leq
 \int_{0}^y \tfrac{1}{\sqrt{ (\sigma \cdot \sigma)'(0)z - c_1 z^2  }}\,  dz
 \\ \nonumber
 &
 =
 \int_{0}^y \tfrac{1}{\sqrt{ (\sigma \cdot \sigma)'(0)z }}\,  dz
 +
  \int_{0}^y \tfrac{
\sqrt{ (\sigma \cdot \sigma)'(0)z }- \sqrt{ (\sigma \cdot \sigma)'(0)z - c_1 z^2  } 
  }{\sqrt{ (\sigma \cdot \sigma)'(0)z   } \sqrt{ (\sigma \cdot \sigma)'(0)z - c_1 z^2  }}\,  dz
   \\
 &
 =
\tfrac{2\sqrt{y}}{\sqrt{ (\sigma \cdot \sigma)'(0)}} 
 +
  \int_{0}^y \tfrac{ c_1 z^2 
  }{\sqrt{ (\sigma \cdot \sigma)'(0)z   } \sqrt{ (\sigma \cdot \sigma)'(0)z - c_1 z^2  }
  \left(
  \sqrt{ (\sigma \cdot \sigma)'(0)z   } + \sqrt{ (\sigma \cdot \sigma)'(0)z - c_1 z^2  }
  \right)
  }\,  dz
   \\ \nonumber
 &
 \leq
\tfrac{2\sqrt{y}}{\sqrt{ (\sigma \cdot \sigma)'(0)}} 
 +
  \int_{0}^y \tfrac{ c_1 z^2 
  }{\sqrt{ (\sigma \cdot \sigma)'(0)z   } \sqrt{\frac{1}{4}(\sigma \cdot \sigma)'(0)z  }
  \left(
  \sqrt{ (\sigma \cdot \sigma)'(0)z   } + \sqrt{ \frac{1}{4}(\sigma \cdot \sigma)'(0)z }
  \right)
  } \, dz
   \\  \nonumber
 &
 =
\tfrac{2\sqrt{y}}{\sqrt{ (\sigma \cdot \sigma)'(0)}} 
 +
  \int_{0}^y \tfrac{4 c_1 z^2 
  }{
  3
  \left( (\sigma \cdot \sigma)'(0)z \right)^{3/2}
  }\,  dz
 =
\tfrac{2\sqrt{y}}{\sqrt{ (\sigma \cdot \sigma)'(0)}} 
 +
 \tfrac{ 8c_1 y^{3/2} 
  }{
  9
  \left( (\sigma \cdot \sigma)'(0) \right)^{3/2}
  } 
=
\tfrac{2\sqrt{y}}{\sqrt{ (\sigma \cdot \sigma)'(0)}} 
 +
 c_2 y^{3/2}.
\end{align}
Estimate \eqref{l:num_bound:1}, 
estimate \eqref{l:num_bound:2} and 
the identity $\alpha = \tfrac{4 \mu(0) - (\sigma \cdot \sigma)'(0)}{2(\sigma \cdot \sigma)'(0)}$
therefore imply that 
for all $y \in (0,\eps)$ it holds that
 {\allowdisplaybreaks  \begin{align} \nonumber \label{l:num_bound:10}
& \big(\phi(y)\big)^2
  \cdot
\left(
\tfrac{
\left( \mu - \frac{(\sigma \cdot \sigma)'}{4}\right)(\sigma \cdot \sigma)'
}
{2(\sigma \cdot \sigma)
 } 
 \right)\!( y ) 
- \alpha
\leq
\left(
\tfrac{2\sqrt{y}}{\sqrt{ (\sigma \cdot \sigma)'(0)}} 
 +
 c_2y^{3/2}\right)^2
\tfrac{
\left( \mu(0) - \frac{(\sigma \cdot \sigma)'(0)}{4} + c_1 y
 \right)
\big(
(\sigma \cdot \sigma)'(0) + c_1y 
\big)
}
{
2\int_{0}^y {(\sigma \cdot \sigma)'(x)} \, dx
}
-
\alpha
\\
&\leq
\left(
\tfrac{4y}{ (\sigma \cdot \sigma)'(0)} 
+
\tfrac{4c_2y^2}{\sqrt{ (\sigma \cdot \sigma)'(0)}} 
 +
 c_2^2 y^3\right)
\tfrac{
\left( \mu(0) - \frac{(\sigma \cdot \sigma)'(0)}{4} + c_1 y
 \right)
\big(
(\sigma \cdot \sigma)'(0) + c_1y 
\big)
}
{
2 (\sigma \cdot \sigma)'(0)y -c_1 y^2
}
-
\tfrac{4 \mu(0) - (\sigma \cdot \sigma)'(0)}{2(\sigma \cdot \sigma)'(0)}
\\  \nonumber
&= \Big(
(\sigma \cdot \sigma)'(0) + c_1y 
\Big)
\left(
\tfrac{1}{ (\sigma \cdot \sigma)'(0)} 
+
\tfrac{c_2y}{\sqrt{ (\sigma \cdot \sigma)'(0)}} 
 +
 \tfrac{c_2^2 y^2}{4}\right)
\left(\tfrac{
 4 \mu(0) - (\sigma \cdot \sigma)'(0) + 4c_1 y
}
{
2 (\sigma \cdot \sigma)'(0) -c_1 y
}
 \right)
-
\tfrac{4 \mu(0) - (\sigma \cdot \sigma)'(0)}{2(\sigma \cdot \sigma)'(0)}.
\end{align} }%
Analogously, we obtain from \eqref{l:num_bound:1} that for all $y \in (0,\eps)$  it holds that
 {\allowdisplaybreaks  \begin{align} \label{l:num_bound:11}  \nonumber
\phi(y) & = 
 \int_{0}^y \tfrac{1}{\sigma(z)} \, dz
 =
 \int_{0}^y \tfrac{1}{\sqrt{ \int_{0}^z {(\sigma \cdot \sigma)'(x)} \, dx }}\,  dz
\geq
 \int_{0}^y \tfrac{1}{\sqrt{ (\sigma \cdot \sigma)'(0)z + c_1 z^2  }} \, dz
 \\ \nonumber
 &
 =
 \int_{0}^y \tfrac{1}{\sqrt{ (\sigma \cdot \sigma)'(0)z }} \, dz
 +
  \int_{0}^y \tfrac{
\sqrt{ (\sigma \cdot \sigma)'(0)z }- \sqrt{ (\sigma \cdot \sigma)'(0)z + c_1 z^2  } 
  }{\sqrt{ (\sigma \cdot \sigma)'(0)z   } \sqrt{ (\sigma \cdot \sigma)'(0)z + c_1 z^2  }}\,  dz
   \\
 &
 =
\tfrac{2\sqrt{y}}{\sqrt{ (\sigma \cdot \sigma)'(0)}} 
 -
  \int_{0}^y \tfrac{ c_1 z^2 
  }{\sqrt{ (\sigma \cdot \sigma)'(0)z   } \sqrt{ (\sigma \cdot \sigma)'(0)z + c_1 z^2  }
  \left(
  \sqrt{ (\sigma \cdot \sigma)'(0)z   } + \sqrt{ (\sigma \cdot \sigma)'(0)z + c_1 z^2  }
  \right)
  } \, dz
   \\ \nonumber
 &
 \geq
\tfrac{2\sqrt{y}}{\sqrt{ (\sigma \cdot \sigma)'(0)}} 
 -
  \int_{0}^y \tfrac{ c_1 z^2 
  }{\sqrt{ (\sigma \cdot \sigma)'(0)z   } \sqrt{\frac{1}{4}(\sigma \cdot \sigma)'(0)z  }
  \left(
  \sqrt{ (\sigma \cdot \sigma)'(0)z   } + \sqrt{ \frac{1}{4} (\sigma \cdot \sigma)'(0)z }
  \right)
  } \, dz
   \\  \nonumber
 &
 =
\tfrac{2\sqrt{y}}{\sqrt{ (\sigma \cdot \sigma)'(0)}} 
 -
  \int_{0}^y \tfrac{ 4c_1 z^2 
  }{
3  \left( (\sigma \cdot \sigma)'(0)z \right)^{3/2}
  } \, dz
 =
\tfrac{2\sqrt{y}}{\sqrt{ (\sigma \cdot \sigma)'(0)}} 
 -
 \tfrac{ 8c_1 y^{3/2} 
  }{
9  \left( (\sigma \cdot \sigma)'(0) \right)^{3/2}
  } 
=
\tfrac{2\sqrt{y}}{\sqrt{ (\sigma \cdot \sigma)'(0)}} 
 -
 c_2y^{3/2}.
\end{align} }%
Next note that \eqref{l:num_bound:1} shows that for all $ y \in (0,\eps) $ 
it holds that 
$
  \frac{ 2 \sqrt{y} }{ \sqrt{ (\sigma \cdot \sigma)'(0) } 
  } 
  -
  c_2 y^{ 3 / 2 } > 0 
$.
This, \eqref{l:num_bound:1} and \eqref{l:num_bound:2} imply that for all $y \in (0,\eps)$ it holds that
 {\allowdisplaybreaks \begin{align} \nonumber \label{l:num_bound:12}
& 
  \big(
    \phi(y)
  \big)^2
  \cdot 
  \left(
    \tfrac{
      \left( \mu - \frac{(\sigma \cdot \sigma)'}{4}\right)(\sigma \cdot \sigma)'
    }{
      2 (\sigma \cdot \sigma)
    } 
  \right)\!( y ) 
  - \alpha
\geq
  \left(
    \tfrac{2\sqrt{y}}{\sqrt{ (\sigma \cdot \sigma)'(0)}} 
    -
    c_2 y^{ 3 / 2 }
  \right)^{ \! 2 }
\tfrac{
\left( \mu(0) - \frac{(\sigma \cdot \sigma)'(0)}{4} - c_1 y
 \right)
\big(
(\sigma \cdot \sigma)'(0) - c_1y 
\big)
}
{
2\int_{0}^y {(\sigma \cdot \sigma)'(x)} \, dx
}
-
\alpha
\\
&\geq
\left(
\tfrac{4y}{ (\sigma \cdot \sigma)'(0)} 
-
\tfrac{4 c_2 y^2}{\sqrt{ (\sigma \cdot \sigma)'(0)}} 
 +
 c_2^2 y^3\right)
\tfrac{
\left( \mu(0) - \frac{(\sigma \cdot \sigma)'(0)}{4} - c_1 y
 \right)
\big(
(\sigma \cdot \sigma)'(0) - c_1 y 
\big)
}
{
2 (\sigma \cdot \sigma)'(0)y + c_1 y^2
}
-
\tfrac{4 \mu(0) - (\sigma \cdot \sigma)'(0)}{2(\sigma \cdot \sigma)'(0)}
\\ \nonumber
&= \Big(
(\sigma \cdot \sigma)'(0) - c_1y 
\Big)
\left(
\tfrac{1}{ (\sigma \cdot \sigma)'(0)} 
-
\tfrac{c_2 y}{\sqrt{ (\sigma \cdot \sigma)'(0)}} 
 +
 \tfrac{c_2^2 y^2}{4}\right)
\left(\tfrac{
 4 \mu(0) - (\sigma \cdot \sigma)'(0) - 4c_1 y
}
{
2 (\sigma \cdot \sigma)'(0) + c_1 y
}
 \right)
-
\tfrac{4 \mu(0) - (\sigma \cdot \sigma)'(0)}{2(\sigma \cdot \sigma)'(0)}.
\end{align} }%
Next observe that \eqref{l:num_bound:10} and \eqref{l:num_bound:12} prove that
\begin{align} \label{l:num_bound:13}
  \limsup_{(0,\eps) \ni x \to 0}
  \left[ 
    \frac{1}{x} \left| 
\big(\phi(x)\big)^2 
  \cdot
\left(
\tfrac{
\left( \mu - \frac{(\sigma \cdot \sigma)'}{4}\right)(\sigma \cdot \sigma)'
}
{2(\sigma \cdot \sigma)
 } 
 \right)\!(x)  
-
\alpha
\right| 
  \right] 
  < \infty.
\end{align}
Finally, we note that 
\eqref{l:num_bound:8}, the continuity of $ \phi $, 
the fact taht $ \phi(0) = 0 $, 
the assumptions that
$
  \mu \in
  C^1\!\left([0,\infty) , \R \right)
$
and
$ (\sigma \cdot \sigma) \in
  C^2\!\left([0,\infty), [0,\infty) \right)
$
and \eqref{l:num_bound:13} and \eqref{l:num_bound:5} 
ensure that
 {\allowdisplaybreaks  \begin{align} \nonumber \label{l:num_bound:14}
& \limsup_{(0,\eps) \ni x \to 0} \left| g'(x) + \tfrac{\alpha}{x^2} \right|
=
 \limsup_{(0,\eps) \ni x \to 0} \left| -g'\big(\phi(x)\big) - \tfrac{\alpha}{(\phi(x))^2} \right|
\\ 
&
\leq
 \limsup_{(0,\eps) \ni x \to 0}\left(  
\left| \mu'(x)\right| + \left|\tfrac{(\sigma \cdot \sigma)''(x)}{4}
\right|\right)
+
 \limsup_{(0,\eps) \ni x \to 0} \left| 
\left(
\tfrac{
\left( \mu - \frac{(\sigma \cdot \sigma)'}{4}\right)(\sigma \cdot \sigma)'
}
{2(\sigma \cdot \sigma)
 } 
 \right)
 (x)  
-
 \tfrac{\alpha}{(\phi(x))^2} \right|
\\ \nonumber
&
=
\left| \mu'(0)\right| + \left|\tfrac{(\sigma \cdot \sigma)''(0)}{4}
\right|
+
 \left[
 \limsup_{(0,\eps) \ni x \to 0}
 \frac{1}{x} \left| 
\big(\phi(x)\big)^2  
\left(
\tfrac{
\left( \mu - \frac{(\sigma \cdot \sigma)'}{4}\right)(\sigma \cdot \sigma)'
}
{2(\sigma \cdot \sigma)
 } 
 \right)
 (x)  
-
\alpha \right|
\right]
\left[
 \limsup_{(0,\eps) \ni x \to 0} \Big| 
 \tfrac{
x}
{
(\phi(x))^2
}\Big|
\right]
< \infty.
\end{align} }%
This finishes the proof of Lemma~\ref{l:num_bound}.
\end{proof}
\begin{lemma} 
\label{l:prop_h}
Let 
$ \alpha \in (0,\infty)$,
$
  \mu \in
  C^1\!\left([0,\infty) , \R \right)
$,
$
  \sigma \in
  C\!\left([0,\infty) , [0,\infty) \right)
$ 
satisfy
$
  \sigma(0) = 0 
$, 
$
  \sigma( (0,\infty) ) \subseteq (0,\infty)
$, 
$
  \sigma \cdot \sigma  \in
  C^2\!\left( [0,\infty) , [0,\infty) \right)
$,
$
  ( \sigma \cdot \sigma )'( 0 ) > 0
$,
$
  \alpha 
  =
  \frac{ 2 \mu( 0 ) }{
    ( \sigma \cdot \sigma )'( 0 )
  }
  -
  \frac{ 1 }{ 2 }
$,
let 
$
  \phi \colon [0,\infty) \to [0,\infty)
$ 
and
$ 
  g \colon (0, \infty) \to \R 
$
be functions defined by
$ 
  \phi(y) 
  := 
  \int_0^y \tfrac{ 1 }{ \sigma(z) } \, dz
$
and
$  
  g(z)
  :=
  \big( 
    \frac{ \mu - \nicefrac{1}{4} (\sigma \cdot \sigma)'}{\sigma} 
  \big)\big( \phi^{-1}(z) \big)
$
for all $ y \in [0,\infty) $,
$ z \in (0,\infty) $,
let
$
  L:= 
  \big[ 
    \sup_{z \in (0,\infty)} 
    g'(z)
  \big]^+ 
  \in [0,\infty]
$
and
  %
  %
  %
  %
  %
assume that 
\begin{align}
\label{eq:polynomial.growth.mu.sigma_lemma}
\inf_{\rho\in[1,\infty)}\limsup_{x\to\infty}
\left[
  \tfrac{
     [ \mu(x) ]^+
   }{ x }
  +
  \tfrac{
    \sigma(x)
  }{ x }
  +
  \tfrac{
    |\mu(x)|
  }{ 
    x^{\rho}
  }
  +
  \big[ g'(x) \big]^+ 
  +
  \tfrac{
    |
      g'(x)
    |
  }{
    x^{ \rho }
  }
\right]
<\infty
  .
\end{align}
%
%
%
%
  %
%
%
%
%
%
Then 
$ \phi $ is well-defined and bijective, 
$
  \phi^{-1}|_{(0,\infty)} \in
  C^1\!\left( (0,\infty) , [0,\infty) \right)
$, 
$ g $ is well-defined and continuously differentiable,
$ L < \infty $, 
it holds
for all $ t \in (0,\nicefrac{1}{L}) $ 
that the function 
$ (0,\infty) \ni x \mapsto x - t g(x) \in \R $ is bijective,
it holds that
\begin{equation} 
\begin{split} 
\lim_{(0,\infty ) \ni x \to 0} \tfrac{| \mu(x) - \mu(0) | + | (\sigma \cdot \sigma)'(x) - (\sigma \cdot \sigma)'(0)|}{\sigma(x)}
 < \infty
\end{split} 
\end{equation} 
and it holds that there exists a real number 
$ c \in [0,\infty) $ such that
for all $ x, y \in (0,\infty) $ it holds that
\begin{align} 
&
  \left( x - y \right)
  \left( g(x) - g(y) \right)
  \leq L \left( x - y \right)^2
  ,
\\ &
  \left| g(x) - g(y) \right|  
  \leq 
  c \left| x - y \right|
  \big( 
    1 + \tfrac{ 1 }{ x y } + x^c + y^c 
  \big)
  ,
\\ & 
  \alpha - c \left( x + x^c \right)  
  \leq x g(x)  
  \leq c \left( 1 + x^2 \right)
  .
\end{align}
\end{lemma}

\begin{proof}[Proof 
of Lemma~\ref{l:prop_h}]
The assumption that
$
  \limsup_{ x \to \infty }
  \frac{ \sigma(x) }{ x }
  < \infty
$
implies that
$
  \smallint_1^{ \infty } 
  \tfrac{ 1 }{ \sigma(z) } \, dz 
  = \infty
$.
%
%
%
%
%
%
%
%
%
Moreover,
the assumptions
that
$
  \sigma \cdot \sigma 
  \in
  C^2\!\left( 
    [0,\infty) , [0,\infty) 
  \right)
$ 
and that
$
  \sigma(
    (0,\infty)
  ) \subseteq (0,\infty)
$ 
show that 
$
  \sigma|_{ ( 0 , \infty ) } 
  \in
  C^1\!\left( (0,\infty) , [0,\infty) \right)
$.
%
%
%
%
%
%
%
%
Lemma~\ref{l:num_bound} 
implies that $ \phi $ is well-defined and bijective, 
that 
$
  g \in C^1\!\left( (0,\infty) , \R \right)
$,
that
\begin{align}  \label{l:prop_h:6a}
 \lim_{(0,\infty) \ni x \to 0} x g(x) = \alpha
  \qquad 
  \text{and that} 
  \qquad
  \limsup_{ 
    (0,\infty) \ni x \to 0
  }  
  \left|
    g'(x) + \tfrac{ \alpha }{ x^2 } 
  \right|
  < \infty .
\end{align}
This implies that 
$
  \liminf_{ (0,\infty) \ni x \to 0 } g(x) = \infty
$ 
and that
$
  \limsup_{ (0,\infty) \ni x \to 0 } g'(x) < \infty
$.
Moreover, Lemma~\ref{l:num_bound} yields that 
$
\lim_{(0,\infty) \ni x \to 0} \tfrac{ \sigma (x)}{\sqrt{x}}
=\sqrt{(\sigma \cdot \sigma)'(0)} \in (0,\infty).
$
Hence, 
there exist real numbers
$ \eps \in (0,1) $, 
$ \gamma \in (0,\infty) $ 
such that for all 
$ x \in (0,\eps) $ 
it holds that 
$ \sigma(x) \geq \gamma \sqrt{x} $. 
This, 
$
  \mu \in
  C^1\! \left( [0,\infty) , \R \right)
$ 
and 
$
  \sigma \cdot \sigma  \in
  C^2\!\left( [0,\infty) , [0,\infty) \right)
$ 
imply 
that there exist real numbers
$ \eps \in (0,1) $, $ \gamma \in (0,\infty) $ 
such that
\begin{equation} 
\begin{split} 
\label{l:prop_h:3} 
\lim_{(0,\infty ) \ni x \to 0} \tfrac{| \mu(x) - \mu(0) | + | (\sigma \cdot \sigma)'(x) - (\sigma \cdot \sigma)'(0)|}{\sigma(x)}
&\leq
\lim_{(0,\eps ) \ni x \to 0} \tfrac{| \mu(x) - \mu(0) | + | (\sigma \cdot \sigma)'(x) - (\sigma \cdot \sigma)'(0)|}{\gamma x}
\\
&
= \tfrac{ |\mu'(0) |+ | (\sigma \cdot \sigma)''(0) | }{\gamma } < \infty.
\end{split} 
\end{equation} 
Observe that $ \phi $ is strictly increasing
and continuous and that 
$
  \phi|_{ (0,\infty) } 
  \in
  C^1\!\left( 
    (0,\infty) , [0,\infty) 
  \right)
$. 
This, the fact that
$ \phi(0) = 0 $, 
the continuity of $ \phi $ 
and 
the fact that 
$
  \smallint_1^{ \infty } 
  \tfrac{ 1 }{ \sigma(z) } \, dz =\infty 
$ 
imply that 
$ \phi $ is bijective,
that  
$
  \phi^{ - 1 } \colon [0,\infty) \to [0,\infty)
$ 
is 
strictly increasing and continuous 
and that 
$
  \phi^{ - 1 }|_{ (0,\infty) } 
  \in
  C^1\!\left( 
    (0,\infty) , [0,\infty) 
  \right)
$.
Next note that
the assumption 
that 
$
  \limsup_{ (0,\infty) \ni x \to 0 } g'(x) < \infty
$,
\eqref{eq:polynomial.growth.mu.sigma_lemma}
and 
the fact that
$
  g \in C^1\!\left( (0,\infty) , \R) \right)
$ 
yield that $ L < \infty $.
Moreover, observe that for all 
$
  x, y \in (0,\infty)
$ 
it holds that  
\begin{align}  
\label{l:prop_h:17NEW}
  \left( x - y \right)
  \left( g(x) - g(y) \right)
\leq 
  L \left( x - y \right)^2.
\end{align}
In the next step we note that
\eqref{l:prop_h:6a} and \eqref{eq:polynomial.growth.mu.sigma_lemma} 
imply that there exists 
$
  \lambda \in [3,\infty)
$, 
which we fix for the rest of this proof, 
such that for all 
$
  x \in (0,\infty)
$, 
$
  y \in [1,\infty)
$,
$
  z \in (0,1)
$ 
it holds that
\begin{align}  
\label{l:prop_h:7aa} 
  \left|
    g'(x) + \tfrac{ \alpha }{ x^2 }
  \right| 
& \leq 
  \lambda 
  \,
  \big( 1 + x^{ \lambda } 
  \big)
  ,
\qquad 
  | g'(y) | 
\leq 
  \lambda \, y^{ (\lambda - 2) }
  ,
 \qquad
  z g(z) \leq 
  \lambda
  \qquad 
  \text{and} 
  \qquad
  |
    g'(z) + \tfrac{ \alpha }{ z^2 } 
  |
  \leq 
  \lambda .
\end{align}
Moreover, 
the mean value theorem implies 
that for every $ x, y \in (0,\infty) $ with $ x < y $ 
there exists a real number $ \xi \in (x,y) $ 
such that
$
  \tfrac{
    g(x) - \frac{ \alpha }{ x } - g(y) 
    + \frac{ \alpha }{ y } 
  }{ x - y } 
  = g'( \xi ) + \tfrac{ \alpha }{ \xi^2 }
$.
%
%
%
This and  \eqref{l:prop_h:7aa}  show that for all $x,y \in (0,\infty)$ it holds that
\begin{align}  
\label{l:prop_h:7} \nonumber
  | g(x) - g(y) | 
& \leq 
  \big|
    g(x) - \tfrac{ \alpha }{ x } - g(y) 
    + \tfrac{ \alpha }{ y } 
  \big| 
  + 
  \alpha 
  \big|
    \tfrac{ 1 }{ x } - \tfrac{ 1 }{ y }
  \big|
\leq 
  \lambda \left| x - y \right|
  \big( 
    1+ x^{ \lambda } + y^{ \lambda }
  \big) 
  + \alpha \left| x - y \right| \tfrac{ 1 }{ x y }
\\ & \leq 
  ( \alpha + \lambda )  
  \left| x - y \right|
  \big(
    1 + \tfrac{ 1 }{ x y } + 
    x^{ \lambda } + y^{ \lambda }
  \big)
  .
\end{align}
%
%
%
%
%
%
%
In the next step
we note that 
\eqref{l:prop_h:7aa} ensures 
that for all $ x \in [1,\infty) $ 
it holds that
\begin{align} 
\label{l:prop_h:17a}
  x g(x) 
&
  = x g(1) 
  + x \smallint_1^x g'(z) \, dz 
\geq 
  x g(1) 
  - \lambda x \smallint_1^x z^{ (\lambda - 2) } \, dz
= x g(1) 
  - \lambda x 
  \tfrac{ 
    ( x^{ (\lambda - 1) } - 1 )
  }{ ( \lambda - 1 ) } 
  \geq 
  \alpha - 
  \big(
    | g(1) | + \alpha + \lambda 
  \big)  
  \, 
  x^{ \lambda }
\end{align}
and that for all $x \in (0,1)$ it holds that
\begin{equation} 
\begin{split} 
\label{l:prop_h:17b}
  x g(x) & = x g(1) 
  + 
  x \smallint_1^x g'(z) + \tfrac{ \alpha }{ z^2 } \, dz
  - x \smallint_1^x \tfrac{ \alpha }{ z^2 } \, dz
\geq 
  x \left( g(1) - \alpha \right) 
  + \alpha - x \smallint_x^1 \lambda \, dz
  \geq 
  \alpha - 
  \left( 
    | g(1) | + \lambda + \alpha
  \right) x .
\end{split} 
\end{equation} %
Moreover, note that 
the fact that $ L \in [0,\infty) $ 
shows that for all $ x \in [1,\infty) $ 
it holds that
\begin{align}   
\label{l:prop_h:17c}
  x g(x) & = x g(1) 
+ 
  x \smallint_1^x g'(z) \, dz 
\leq 
  x g(1) + L x \left( x - 1 \right) 
\leq 
  \left( L + |g(1)| \right) 
  \left( 1 + x^2 \right)
  .
\end{align}
%
%
%
%
%
In the next step we define 
a real number $ c \in \R $
through
$
  c := \alpha + L + | g(1) | + \lambda 
$
and we observe that 
\eqref{l:prop_h:7aa}, 
\eqref{l:prop_h:17a}, \eqref{l:prop_h:17b} and \eqref{l:prop_h:17c}  show  that for all 
$
  x \in (0,\infty)
$ 
it holds that
\begin{align}  
\label{l:prop_h:17d}
  \alpha - c \left( x + x^c \right) 
\leq 
  x g(x)  
\leq 
  c \left( 1 + x^2 \right)
  .
\end{align}
%
%
%
%
%
%
%
%
%
%
%
%
%
%
%
Observe that 
$g \in  C^1\! \left( (0,\infty) , \R \right)$
implies that for all $ x \in (1, \infty) $, $ t \in [0,\infty ) $ 
it holds that
\begin{equation}
\begin{split}
  x - t g(x) 
&=
  x 
  - t 
  \left[ 
    g( x ) - g( 1 ) + g(1) 
  \right]
 \geq 
  x - t 
  \left[ 
    L \left( x - 1 \right) + g(1) 
  \right]
\\ & =
  x \left( 1 - t L \right) - 
  t \left( g(1) - L \right)
  .
\end{split}
\end{equation}
This implies that for all 
$
  t \in (0, \frac{ 1 }{ L } ) 
$ 
it holds that 
$
  \lim_{ x \to \infty } 
  \left( x - t g(x) \right) = \infty
$. 
Combining
this and
the fact that
$
  \liminf_{ x \searrow 0 } 
  g(x) 
  = \infty
$
with the continuity of the function $ g $
yields that
for all $ t \in (0,\frac{1}{L}) $ 
it holds that the function 
$ (0,\infty) \ni x \mapsto x - t g(x) \in \R $ is surjective. This together with Lemma~\ref{l:inverse_function} ensures that
for all $ t \in (0,\frac{1}{L}) $ 
it holds that the function 
$ (0,\infty) \ni x \mapsto x - t g(x) \in \R $ is bijective.
The proof of Lemma~\ref{l:prop_h}
is thus completed.
\end{proof}

\begin{theorem}   
\label{thm:num_theorem}
Let $ T, \alpha \in (0,\infty) $, 
$\mu  \in
  C^1\!\left([0,\infty) , \R \right)$,
 $\sigma \in
  C\!\left([0,\infty) , [0,\infty) \right)$ satisfy 
 $\sigma(0)=0$, $\sigma((0,\infty)) \subseteq (0,\infty)$, $\sigma \cdot \sigma  \in
  C^2\!\left([0,\infty) , [0,\infty) \right)$,
  $
    ( \sigma \cdot \sigma )'(0) > 0 
  $,
  $
    \alpha = 
      \frac{ 
        2 \mu(0) 
      }{
        ( \sigma \cdot \sigma )'( 0 ) 
      }
    - \frac{ 1 }{ 2 } 
  $,
let
$
  ( 
    \Omega, \mathcal{F}, \P, ( \mathcal{F}_t )_{ t \in [0,T] } 
  )
$
be a stochastic basis, let
$
  W \colon [0,T] \times \Omega \to \R
$
be a standard $ ( \mathcal{F}_t )_{ t \in [0,T] } $-Brownian motion,
  let $X\colon[0,T]\times \Omega\to[0,\infty) $
  be an adapted stochastic process with continuous sample paths satisfying 
  \begin{equation} \label{thm:num_theorem:ass1}
    X_t=X_0+\int_0^t \mu(X_s)\,ds+\int_0^t\sigma(X_s)\,dW_s
  \end{equation}
  $\P$-a.s.\
  for all $t \in [0,T]$,
  let 
  $ \phi \colon [0,\infty) \to [0,\infty) $ 
  and
  $ g \colon [0, \infty) \to \R $  
  be functions
  defined by
  $ 
    \phi(y) := \int_{0}^y \tfrac{1}{\sigma(z)} \, dz
  $,
  $
    g( 0 ) := 0
  $
  and
  $  
    g(z)
    :=
    \big( 
      \frac{ \mu - \nicefrac{ 1 }{ 4 } ( \sigma \cdot \sigma )' }{ \sigma } 
    \big)\big( 
      \phi^{-1}(z) 
    \big)
  $
  for all $ y \in [0,\infty) $, 
  $ z \in (0,\infty) $,
  let
  $
    L:= 
    \big[ 
      \sup_{z \in (0,\infty)} 
      g'(z)
    \big]^+ 
    \in [0,\infty]
  $
  %
  %
  %
  %
  %
  and assume that
  $
    \sup_{ t \in [0,T] }
    \E\!\left[
      \left| \phi( X_0 ) \right|^r 
      +
      \left| \phi( X_t ) \right|^{ - q }
    \right] 
    < \infty
  $
  for all $ r \in \R $, $ q \in [1,1 + 2 \alpha) $
  and 
\begin{align}
\label{eq:polynomial.growth.mu.sigma}
\inf_{
  \rho \in [1,\infty)
}
\limsup_{ x \to \infty }
\left[
  \tfrac{
     [ \mu(x) ]^+
   }{ x }
  +
  \tfrac{
    \sigma(x)
  }{ x }
  +
  \tfrac{
    |\mu(x)|
  }{ 
    x^{\rho}
  }
  +
  \tfrac{ 
    \sigma(x)
  }{ 
    1 + ( \phi(x) )^{ \rho }  
  }
  +
  \big[ g'(x) \big]^+ 
  +
  \tfrac{
    |
      g'(x)
    |
  }{
    x^{ \rho }
  }
\right]
<\infty
  .
\end{align}
%
%
%
%
%
%
%
%
%
%
%
Then 
$ L < \infty $ 
and there exists a unique family
$
  Y^h \colon 
  ( \N_0 \cap [0,\nicefrac{T}{h}] ) \times \Omega \to [0,\infty) 
$, $h \in (0,T] \cap (0,\nicefrac{1}{L})$, 
of stochastic processes 
satisfying 
$Y^{h}_0 = \phi(X_0)$  and
  \begin{equation} \label{thm:num_theorem:umformung}
    Y^{h}_{n}= Y^{h}_{n-1} + g(Y^{h}_{n}) \, h + W_{nh} -W_{(n-1)h}  
  \end{equation}
for all $n \in \N \cap [0,\nicefrac{T}{h}]$, $h \in (0,T] \cap (0,\nicefrac{1}{L})$
and 
it holds for all 
$ \eps \in ( 0,\infty) $, $ p \in [1,\infty) $
that
\begin{align} \label{thm:num_theorem:ass8}
  \sup_{
    h \in (0,T] \cap [ 0 , \nicefrac{ 1 }{ ( 4 L ) } ] 
  } 
    \left[ 
      h^{
        \left( 
          \eps 
          -
          \frac{ \left( \alpha \wedge \nicefrac{ 1 }{ 2 } \right) }{ p } 
        \right)
      }
      \,
      \bigg\| 
        \sup_{ n \in \N_0 \cap [0,\nicefrac{ T }{ h } ] }  
        \big| 
          X_{ n h } - \phi^{ - 1 }( Y^h_n ) 
        \big| 
        \,
      \bigg\|_{ L^p( \Omega; \R ) }
    \right] 
    < \infty
    .
\end{align}
\end{theorem}

\begin{proof} [Proof of Theorem~\ref{thm:num_theorem}]

It follows from Lemma~\ref{l:prop_h} that
$ \phi $ is well-defined and bijective, 
that 
$ 
  \phi^{ - 1 }|_{(0,\infty)} 
  \in
  C^1\! \left( (0,\infty) , [0,\infty) \right)
$, 
that $ g $ is well-defined,
that $ g|_{(0,\infty)}$ is continuously differentiable,
that $ L < \infty $
and that
\begin{equation} 
\begin{split} \label{thm:num_theorem:3}
\lim_{(0,\infty ) \ni x \to 0} \tfrac{| \mu(x) - \mu(0) | + | (\sigma \cdot \sigma)'(x) - (\sigma \cdot \sigma)'(0)|}{\sigma(x)}
 < \infty.
\end{split} 
\end{equation} 
Furthermore, 
Lemma~\ref{l:prop_h} implies that
there exists a real number $c \in \R$, which we fix for the rest of this proof, such
that for all 
$
  x, y \in (0,\infty)
$ 
it holds that  
\begin{align}  
\label{thm:num_theorem:17NEW}
&
  \left( x - y \right) \left( g( x ) - g( y ) \right)
  \leq 
  L \left( x - y \right)^2
  ,
\\ &
\label{thm:num_theorem:7} 
  \left| g( x ) - g( y ) \right| 
  \leq
  c \left| x - y \right|
  \big( 
    1 + \tfrac{ 1 }{ x y } + x^c + y^c
  \big)
  ,
\\ &
\label{thm:num_theorem:17d}
  \alpha - c \left( x + x^c \right) 
  \leq 
  x g(x)  
  \leq c \left( 1 + x^2 \right) .
\end{align}
Moreover,
Lemma~\ref{l:prop_h} shows that
for all $ t \in (0,\nicefrac{1}{L}) $ 
it holds that the function 
$ (0,\infty) \ni x \mapsto x - t g(x) \in \R $ is bijective. 
%
%
%
%
%
%
%
%
%
%
%
%
%
This proves that
there exists a unique family 
$
  Y^h \colon 
  (\N_0 \cap [0,\nicefrac{T}{h}]) \times \Omega
  \to [0,\infty) 
$, 
$
  h \in (0,T] \cap ( 0, \nicefrac{ 1 }{ L } )
$, 
of 
stochastic processes 
satisfying 
$ Y^h_0 = \phi( X_0 ) $ 
and
  \begin{equation} \label{thm:num_theorem:umformung2}
  Y^h_n = Y^h_{ n - 1 } 
  + g\big( Y^h_n \big) h 
  + W_{ n h } - W_{ ( n - 1 ) h }  
  \end{equation}
for all 
$
  n \in \N \cap [0, \nicefrac{ T }{ h } ] 
$,
$ h \in (0,T]  \cap ( 0 , \nicefrac{ 1 }{ L } ) $.
In the next step, we define a mapping 
$
  Z \colon [0,T] \times \Omega \to [0,\infty)
$ 
by
$
  Z_t := \phi(X_t)
$
for all $ t \in [0,T] $.
Note that
$Z$ is an adapted stochastic process with continuous sample paths.
Moreover, observe that the assumption
that
$
  \sup_{ t \in [0,T] } 
  \E\!\left[
    ( Z_t )^{ - 1 } 
  \right] < \infty
$ 
implies that
$
  \sup_{ t \in [0,T] } 
  \P\!\left[ 
    Z_t = 0 
  \right] = 0
$.
Next note that 
Lemma~\ref{l:transformierte_gleichung}, 
\eqref{thm:num_theorem:3}, 
$ \sigma(0) = 0 $ and 
$
  \mu(0) - \tfrac{ 1 }{ 4 } 
  ( \sigma \cdot \sigma )'( 0 ) > 0
$ 
yield that
\begin{align} \label{thm:num_theorem:6}
 \int_0^T \left|\left( \tfrac{\mu - \frac{1}{4} (\sigma \cdot \sigma)'}{\sigma} \right) \! \left( \phi^{-1}(Z_s) \right) \1_{\{Z_s \in (0,\infty)\}} \right| \, ds < \infty
\end{align}
 $\P$-a.s.\
and 
\begin{equation} 
\begin{split} 
 \label{thm:num_theorem:77} 
  Z_t &= Z_0 + \int_0^t \left( \tfrac{\mu - \frac{1}{4} (\sigma \cdot \sigma)'}{\sigma} \right) \! \left( \phi^{-1}(Z_s) \right) \1_{\{Z_s \in \phi((0,\infty)) \}} \, ds + W_t
= Z_0 + \int_0^t  g(Z_s) \, ds + W_t
\end{split} 
\end{equation} $\P$-a.s.\
for all $t\in[0,T]$. 
%
%
%
%
%
%
%
Lemma~\ref{l:lemma_positive_mom} together with \eqref{thm:num_theorem:17d}
and $\E [ (Z_0)^{q} ] < \infty$ for all $q \in (0,\infty)$ implies that
for all $q \in (0,\infty)$ it holds that
$\sup_{t \in [0,T]} \E [ ( Z_t)^{q} ] 
 < \infty.$
%
%
%
%
Then Lemma~\ref{l:lemma_positive_mom_mit_sup} yields that
for all $q \in (0,\infty)$ it holds that
\begin{align} \label{thm:num_theorem:612} 
\left\| \sup_{s \in [0,T]}  Z_s  \, \right\|_{ L^{q}( \Omega; \R ) }  
 < \infty.
\end{align}
%
%
%
%
%
%
In the next step we note that 
\eqref{thm:num_theorem:17d}
and the assumption that
$
  \E\!\left[ ( Z_0 )^q \right] < \infty
$
for all $ q \in \R $ 
show 
that for all $ q \in (0,\infty) $ 
it holds that $\E\big[
   | g(Z_0) |^q 
 \big]
< \infty$. 
%
%
%
%
%
%
%
%
%
%
Corollary~\ref{l:uni_mom_num2} 
(applied with $ O = (0,\infty) $ and 
$ \xi = Z_0 + \1_{ \{ Z_0 = 0 \} } $),
$
  \P[Z_0 = 0] = 0
$,
\eqref{thm:num_theorem:17NEW} and
\eqref{thm:num_theorem:17d} 
prove that for all $ q \in (0,\infty) $ 
it holds that
\begin{align} \label{thm:num_theorem:613b}
\sup_{ h \in (0,T] \cap (0,\frac{1}{4L}]} 
\left\|
 \sup_{ n \in \N_0 \cap [0,T/h]}   Y^{h}_{n} \right\|_{ L^{q}( \Omega; \R ) }
< \infty.
\end{align}
Corollary~\ref{l:num_corollary} together with
\eqref{thm:num_theorem:17NEW} and \eqref{thm:num_theorem:7} shows that for all $\eps, \kappa  \in \left( 0,\infty \right)$, $\gamma \in (0,1]$  it holds that
\begin{align} \label{thm:num_theorem:888} \nonumber
&
 \sup_{h \in (0,T] \cap [0,\frac{1}{4L}]} \left[ 
h^{- \gamma \left( \frac{1}{2} - \frac{\eps}{2\gamma}\right)}
\left\| 
\sup_{
n \in \N_0 \cap [0,T/h]
}
\left| Z_{nh} -  Y^{h}_{n} \right| 
\right\|_{ L^{1}( \Omega; \R ) } \right] 
\\  &
\leq
2 T c  
 \sup_{h \in (0,T] \cap [0,\frac{1}{4L}]} \left[ 
\left( \tfrac{1}{1-hL} \right)^{\lceil T / h \rceil}
h^{- \gamma \left( \frac{1}{2} - \frac{\eps}{2\gamma}\right)}
\Big\| 
\sup_{u \in [0,T]}   
\left| Z_u - Z_{\lceil u \rceil_{(0,h,2h,\dots, \lceil T/h \rceil h) }} \right|
 \Big\|_{ L^{\gamma(1 + \nicefrac{1}{\kappa})}( \Omega; \R ) }^{\gamma} \right]
\\ & \quad \cdot \nonumber
\sup_{u\in[0,T]} \left[
\big\| 
 \tfrac{2}{(Z_u)^{1+\gamma}} 
\big\|_{ L^{1+\kappa}( \Omega; \R ) } 
+  \tfrac{\gamma}{1+\gamma}  
\big\| (Z_u)^{\frac{1}{\gamma} -\gamma}  
\big\|_{ L^{1+\kappa}( \Omega; \R ) } 
  + 2
\big\|
\left( Z_u \right)^{ c + 1 - \gamma }
\big\|_{ L^{1+\kappa}( \Omega; \R ) } 
+ \tfrac{1}{1+\gamma} 
 \right].
\end{align}
Estimate~\eqref{thm:num_theorem:888} 
(applied with 
$ 
  \eps \in ( 0, 2 \alpha \wedge 1) 
$, 
$
  \gamma = 2 \alpha \wedge 1 - \eps
$,
$
  \kappa = 
  \frac{ 1 }{ 2 } 
  ( 
    \frac{ 1 + 2 \alpha }{ 1 + \gamma } - 1 
  )
$), 
Theorem~\ref{l:cir_hr},
the assumption that
$
  \mu \in
  C^1\!\left( [0,\infty) , \R \right)
$,
the assumption that
$
  \sigma \cdot \sigma 
  \in
  C^2\!\left(
    [0,\infty) , [0,\infty) 
  \right)
$,
\eqref{eq:polynomial.growth.mu.sigma},
\eqref{thm:num_theorem:612} 
and the assumption that
$
  \sup_{ t \in [0,T] } 
  \E\!\left[
    ( Z_t )^{ - q } 
  \right] < \infty 
$ 
for all 
$
  q \in [ 1, 1 + 2 \alpha ) 
$
show that for all 
$ 
  \eps \in ( 0, 2 \alpha \wedge 1 )
$ 
it holds that
\begin{align} 
\label{thm:num_theorem:888NEW} 
&
  \sup_{
    h \in (0,T] \cap [ 0, \frac{ 1 }{ 4 L } ]
  } 
  \left[ 
    h^{
      ( 
        \eps - 
        \min\{ \alpha , \nicefrac{ 1 }{ 2 } \} 
      )
    }
    \left\| 
      \sup_{
        n \in \N_0 \cap [ 0, T / h ]
      }
      \big| 
        Z_{ n h } - Y^h_n 
      \big| 
    \right\|_{ 
      L^1( \Omega; \R ) 
    } 
  \right] 
  < \infty .
\end{align}
%
%
Remark \ref{remark.L1.to.Lp}, 
\eqref{thm:num_theorem:888NEW},
\eqref{thm:num_theorem:612} 
and 
\eqref{thm:num_theorem:613b} 
yield that for all 
$
  p \in [1,\infty)
$,  
$
  \eps \in ( 0, \min\{ 2 \alpha , 1 \} )
$, 
$
  \kappa \in \left( 0, \infty \right)
$  
it holds that
\begin{align} \label{thm:num_theorem:888dd}  \nonumber
& 
  \sup_{
    h \in (0,T] \cap [ 0 , \frac{ 1 }{ 4 L } ]
  } 
  \left[ 
    h^{
      - 
      \left[
        \frac{ 
          \min\{ \alpha , \nicefrac{ 1 }{ 2 } \} - \eps
        }{ p ( 1 + \kappa ) }
      \right]
    }
    \left\|
      \sup_{
        n \in \N_0 \cap [0, T / h ]
      } 
      \left| 
        Z_{ n h } - Y^h_n 
      \right| 
    \right\|_{ 
      L^p( \Omega; \R )  
    } 
  \right]
\\ & \leq
  \sup_{
    h \in (0,T] \cap [ 0, \frac{ 1 }{ 4 L } ] 
  } 
  \Bigg( 
    \left[
      h^{ 
        ( \eps - \min\{ \alpha , \nicefrac{1}{2} \} )
      }
      \left\| 
        \sup\nolimits_{
          n \in \N_0 \cap [0,T/h]
        } 
        | Z_{ n h } - Y^h_n | 
      \right\|_{ L^1( \Omega; \R ) } 
    \right]^{
      \frac{ 1 }{ p ( 1 + \kappa ) } 
    }  
\\ &  \quad \cdot \nonumber
  \left[ 
    \left\| 
      \sup\nolimits_{ s \in [0,T] } 
      \left| Z_s \right| 
    \right\|_{ 
      L^{ p + \frac{ p - 1 }{ \kappa } 
      }( \Omega; \R ) 
    }^{ 
      1 - \frac{ 1 }{ p (1 + \kappa ) } 
    } 
    +
    \left\|
      \sup\nolimits_{
        n \in \N_0 \cap [0,T/h]
      } 
      \big| Y_n^h \big| 
    \right\|_{ 
      L^{ p + \frac{ p - 1 }{ \kappa } 
      }( \Omega; \R )
    }^{ 1 - \frac{ 1 }{ p ( 1 + \kappa ) }
    }
  \right]
  \Bigg)
  < \infty .
\end{align}
Next note that
estimate \eqref{eq:polynomial.growth.mu.sigma}
proves that there exists a real number 
$
  \rho \in \R 
$, 
which we fix for the rest of this proof, 
such that for all 
$
  x \in [0,\infty)
$ 
it holds that
$
  \sigma(x) 
  \leq \rho \,
  \big(
    1 + ( \phi(x) )^{ \rho } 
  \big)
$.
This together with the  monotonicity and the continuity  of $\phi$
shows that for all 
$
  p \in [1,\infty)
$, 
$
  h \in (0,T] \cap (0,\frac{ 1 }{ 4 L } ]
$, 
$
  \delta \in (0,\infty)
$ 
it holds that
 {\allowdisplaybreaks  
\begin{align} 
\label{thm:num_theorem:614}
&
  \left\| 
    \sup_{
      n \in \N_0 \cap [0, T / h ]
    } 
    \sup_{
      z \in 
      [ 
        X_{ n h } \wedge \phi^{ - 1 }( Y^h_n ) , 
        X_{ n h } \vee \phi^{ - 1 }( Y^h_n ) 
      ] 
    } 
    \sigma(z) 
  \right\|_{ 
    L^{ 
      p ( 1 + \nicefrac{ 1 }{ \delta } ) 
    }( \Omega; \R ) 
  }
\nonumber
\\ & \leq  
\left\|
  \sup_{
    n \in \N_0 \cap [0,T/h]
  } 
  \sup_{
    z \in 
    [ 
      X_{ n h } \wedge \phi^{ - 1 }( Y^h_n ) , 
      X_{ n h } \vee \phi^{ - 1 }( Y^h_n ) 
    ] 
  } 
  \rho \,
  \big(
    1 + ( \phi(z) )^{ \rho } 
  \big)
\right\|_{ 
  L^{ p ( 1 + \nicefrac{ 1 }{ \delta } ) 
  }( \Omega; \R ) 
}
%
%
%
%
%
%
%
%
%
%
%
%
\\
\nonumber
& \leq 
  \rho 
  \left[
    \left\|
      \sup_{ s \in [0,T] } 
      \big[ 
        1 + ( Z_s )^{ \rho } 
      \big] 
    \right\|_{ 
      L^{ p ( 1 + \nicefrac{ 1 }{ \delta } ) 
      }( \Omega; \R ) 
    } 
    + 
    \sup_{ 
      \tilde{h} \in (0,T] \cap [ 0 , \frac{ 1 }{ 4 L } ]
    } 
    \left\|  
      \sup_{ 
        n \in \N_0 \cap [ 0 , T / \tilde{h} ] 
      } 
      \left[
        1 + ( Y^{ \tilde{h} }_n )^{ \rho } 
      \right]
    \right\|_{ 
      L^{ p ( 1 + \nicefrac{ 1 }{ \delta } ) 
    }( \Omega; \R ) } 
  \right]
  .
\end{align} }%
%
%
%
%
%
%
%
%
%
%
%
%
H\"older's inequality 
and \eqref{thm:num_theorem:614} 
show that for all 
$
  p \in [1,\infty)
$, 
$
  h \in (0,T] \cap (0, \frac{ 1 }{ 4 L } ]
$, 
$
  \delta \in (0,\infty)
$ 
it holds that
 {\allowdisplaybreaks  \begin{align} \label{thm:num_theorem:511} \nonumber
&
  \left\| 
    \sup_{ n \in \N_0 \cap [ 0, T / h ]
    }  
    \big| 
      X_{ n h } - 
      \phi^{-1}( Y^h_n )
    \big| 
  \right\|_{ 
    L^p( \Omega; \R ) 
  }
=
\left\| \sup_{ n \in \N_0 \cap [0,T/h]} \left| \int_{X_{nh}}^{\phi^{-1}(Y^{h}_{n})}  \tfrac{\sigma(z)}{\sigma(z)} \, dz \right| \right\|_{ L^{p}( \Omega; \R ) }
\\ & \nonumber 
\leq 
\left\|
\sup_{ n \in \N_0 \cap [0,T/h]} 
\sup_{z \in [X_{nh} \wedge \phi^{-1}(Y^{h}_n),X_{nh} \vee \phi^{-1}(Y^{h}_n)] }  \sigma(z) \right\|_{ L^{p(1 + \nicefrac{1}{\delta})}( \Omega; \R ) }
\left\| \sup_{ n \in \N_0 \cap [0,T/h]} \left|  \int_{X_{nh}}^{\phi^{-1}(Y^{h}_{n})} \hspace{-0.15cm}  \tfrac{1}{\sigma(z)} \, dz \right| \right\|_{ L^{p(1+\delta)}( \Omega; \R ) }
\\ \nonumber
& \leq 
 \rho \left[
\left\|\sup_{s \in [0,T]} \big[ 1+(Z_s)^{\rho} \big] \right\|_{ L^{p(1 + \nicefrac{1}{\delta})}( \Omega; \R ) } + 
\sup_{ \tilde{h} \in (0,T] \cap [0,\frac{1}{4L}]} 
\left\|  
 \sup_{ n \in \N_0 \cap [0,T/\tilde{h}]} 
\left[
1+(Y^{\tilde{h}}_{n})^{\rho} 
\right]
\right\|_{ L^{p(1 + \nicefrac{1}{\delta})}( \Omega; \R ) } \right]
\\ 
& \quad \cdot
\left\| \sup_{ n \in \N_0 \cap [0,T/h]} \left| Z_{nh} - Y^h_n \right| \right\|_{ L^{p(1+\delta)}( \Omega; \R ) }.
%
%
\end{align} }%
%
%
%
This,
\eqref{thm:num_theorem:612}, \eqref{thm:num_theorem:613b} 
and \eqref{thm:num_theorem:888dd} 
imply that for all 
$ p \in [1,\infty) $, 
$
  \eps \in \left( 0 , \min\{ 2 \alpha , 1 \} \right)
$, 
$
  \kappa, \delta \in (0,\infty)
$  
it holds that
 {\allowdisplaybreaks  \begin{align} \label{thm:num_theorem:511WW} 
&
  \sup_{
    h \in (0,T] \cap [ 0 , \frac{ 1 }{ 4 L } ] 
  } 
  \left[ 
    h^{
      -
      \left[
      \frac{
        \min\{ \alpha , \nicefrac{1}{2} \} - \eps
      }{
        p ( 1 + \delta ) ( 1 + \kappa ) 
      } 
      \right]
    }
    \left\| 
      \sup_{ n \in \N_0 \cap [0, T / h ] 
      }  
      \big| 
        X_{ n h } - \phi^{ - 1 }( Y^h_n ) 
      \big| 
    \right\|_{ 
      L^p( \Omega; \R ) 
    }
  \right] < \infty.
\end{align} }%
%
%
%
This proves that
for all 
$
  p \in [1,\infty)
$, 
$
  \eps \in 
  \big( 
    0 , 
    \frac{
      \min\{ \alpha , \nicefrac{ 1 }{ 2 } \}
    }{ p }
  \big)
$ 
it holds that
 {\allowdisplaybreaks  
\begin{align} 
\label{thm:num_theorem:511WQQW} 
&
  \sup_{
    h \in (0,T] \cap [ 0, \frac{ 1 }{ 4 L } ]
  } 
  \left[ 
    h^{
      - 
      \left[
        \frac{ \alpha \wedge \nicefrac{ 1 }{ 2 } 
        }{ p } - \eps 
      \right]
    }
    \left\| 
      \sup_{ n \in \N_0 \cap [0, T / h ] }  
      \big| 
        X_{ n h } - 
        \phi^{ - 1 }( Y^h_n ) 
      \big| 
    \right\|_{ 
      L^p( \Omega; \R ) 
    }
  \right] < \infty.
\end{align} }%
The proof of Theorem~\ref{thm:num_theorem}
is thus completed.
\end{proof}

Theorem~\ref{thm:num_theorem}
uses the assumption that 
suitable inverse moments 
of a transformation 
of the solution process of the SDE~\eqref{thm:num_theorem:ass1}
is finite,
that is,
in the setting of Theorem~\ref{thm:num_theorem}
that the quantity
$
  \sup_{ t \in [0,T] }
  \E\!\left[
    \left| \phi( X_t ) \right|^{ - q }
  \right] 
$
is finite
for all $ q \in [1,1 + 2 \alpha) $.
The next result, Lemma~\ref{l:inverse.moments.CIR},
gives a sufficient condition to ensure finiteness
of suitable inverse moments of
Cox-Ingersoll-Ross processes.
Lemma~\ref{l:inverse.moments.CIR}
extends and is based on Lemma~A.1 in 
Bossy \& Diop~\cite{BossyDiop2007}.

\begin{lemma} 
\label{l:inverse.moments.CIR}
Let $ T, \beta, \delta \in (0,\infty) $, $ \gamma \in \R $, 
$ p \in ( 0, \nicefrac{ 2 \delta }{ \beta^2 } ) $,
let
$
  ( 
    \Omega, \mathcal{F}, \P, ( \mathcal{F}_t )_{ t \in [0,T] } 
  )
$
be a stochastic basis, let
$
  W \colon [0,T] \times \Omega \to \R
$
be a standard $ ( \mathcal{F}_t )_{ t \in [0,T] } $-Brownian motion,
let $X\colon[0,T]\times \Omega\to[0,\infty) $
be an adapted stochastic process with continuous sample paths satisfying 
  \begin{equation} \label{l:theorem_cir:ass1}
    X_t=X_0+\int_0^t \delta - \gamma X_s\,ds+ \int_0^t \beta \sqrt{X_s} \, dW_s
  \end{equation}
  $\P$-a.s.\
  for all $t \in [0,T]$
  and assume that 
  $ 
    \inf_{ \eps \in ( 0, \nicefrac{ 2 \delta }{ \beta^2 } - p ) } 
    \sup_{ u \in [1,\infty) } 
    u^{ (p + \eps) } 
    \,
    \E\big[
      \exp( - u X_0 )
    \big] 
    < \infty
  $.
Then it holds that
  $\sup_{t\in[0,T]}
    \E\! \left[\left(X_t\right)^{-p}\right]<\infty$.
\end{lemma}

\begin{proof}[Proof 
of Lemma~\ref{l:inverse.moments.CIR}]
First of all, note that the assumption 
$ 
    \inf_{ \eps \in ( 0, \nicefrac{ 2 \delta }{ \beta^2 } - p ) } 
    \sup_{ u \in [1,\infty) } 
    u^{ (p + \eps) } 
    \,
    \E\big[
      \exp( - u X_0 )
    \big] 
    < \infty
$
implies that
$
    \inf_{ \eps \in ( 0, \nicefrac{ 2 \delta }{ \beta^2 } - p ) } 
    \sup_{ u \in (0,\infty) } 
    u^{ (p + \eps) } 
    \,
    \E\big[
      \exp( - u X_0 )
    \big] 
    < \infty
$.
This shows that there exist real numbers 
$ C \in ( 0 , \infty ) $,
$ \eps \in \left( 0, \nicefrac{ 2 \delta }{ \beta^2 } - p \right) $, 
which we fix for the rest of this proof, 
such that for all $ u \in (0,\infty) $ 
it holds that 
$
  \E\big[
    \exp( - u X_0 )
  \big] 
  \leq 
  C \, 
  u^{ - ( p +¸\eps ) }
$.
This shows that for all $ t \in [0,T] $ it holds that
\begin{align}  
\nonumber
\label{l:inverse.moments.CIR:1}
&
  \int_1^\infty 
    u^{ (p - 1) }
    \left[
      \tfrac{ 
        u \beta^2 
      }{ 
        2 \gamma 
      }
      \left(
        1 - e^{ - \gamma t } 
      \right)
      + 1 
    \right]^{
      - \frac{ 2 \delta }{ \beta^2 } 
    }
    \E\!\left[
      \exp\!\left(
        - 
        \tfrac{ 
          e^{ - \gamma t } 
        }{
          \frac{ 
            u \beta^2 
          }{
            2 \gamma
          }
          (
            1 - e^{ - \gamma t } 
          ) 
          + 1 
        }
	X_0 u
      \right)
    \right] 
  du
\\&
\leq C 
  \int_1^\infty u^{ (p-1) }
    \left[
      \tfrac{ u \beta^2 }{ 2 \gamma }
      \left(
        1 - e^{ - \gamma t } 
      \right)
      +
      1
    \right]^{
      - \frac{ 2 \delta }{ \beta^2 }
    }
    \left[
      \tfrac{
        e^{ - \gamma t } 
      }{
        \frac{ u \beta^2 
        }{
          2 \gamma
        }
        ( 
          1 - e^{ - \gamma t } 
        )
        + 1 
      }
      u
    \right]^{
      - ( p + \eps ) 
    } 
  du
\\& \nonumber
= C \, e^{ \gamma t ( p + \eps ) } 
  \int_1^{ \infty } 
    u^{ - ( 1 + \eps ) } 
    \left[
      \tfrac{ 
        u \beta^2 }{ 2 \gamma 
      }
      \left( 
        1 - e^{ - \gamma t } 
      \right)
      + 1 
    \right]^{
      \left[
        p + \eps 
        - \frac{ 2 \delta }{ \beta^2 } 
      \right]
    }
  du
\leq 
  C \, 
  e^{ 
    \gamma t ( p + \eps )
  }
  \int_1^\infty 
    u^{ - ( 1 + \eps ) } 
  \, du
= \tfrac{C}{\eps} e^{ \gamma t (p+\eps)}.
\end{align}
This, line 8 in the proof of Lemma A.1 in Bossy \& Diop~\cite{BossyDiop2007}
and Fubini's theorem 
imply that for all $t\in [0,T]$ it holds that
\begin{equation}  \begin{split}
  &
    \E\!\left[
      \left( X_t \right)^{ - p }
    \right]
  =
  \E\!\left[
    \tfrac{ 1 }{ \Gamma(p) }
    \int_0^\infty 
      u^{ (p - 1) }
      \left[
        \tfrac{ u \beta^2 }{ 2 \gamma }
        \left( 
          1 - e^{ - \gamma t }
        \right)
        + 1 
      \right]^{
        - \frac{ 2 \delta }{ \beta^2 }
      }
      \exp\!\left(
        -
        \tfrac{
          e^{ - \gamma t } 
        }{
          \frac{ 
            u \beta^2 
          }{ 
            2 \gamma 
          }
          ( 
            1 - e^{ - \gamma t }
          )
          + 1
        }
        X_0 u
      \right)
    du 
  \right]
\\ & = 
  \tfrac{ 1 }{
    \Gamma(p)
  } 
  \Bigg(
    \int_0^1 
      u^{ (p - 1) }
      \left[
        \tfrac{ u \beta^2 }{ 2 \gamma }
        ( 1 - e^{ - \gamma t } )
        + 1 
      \right]^{
        - \frac{ 2 \delta }{ \beta^2 } 
      }
      \E\!\left[
        \exp\!\left(
          - 
          \tfrac{ 
            e^{ - \gamma t } 
          }{
            \frac{
              u \beta^2
            }{
              2 \gamma
            }
            ( 1 - e^{ - \gamma t } ) + 1
          }
	  X_0 u
        \right)
      \right]
    du
\\ & \qquad \quad 
+
  \int_1^\infty 
    u^{ (p - 1) }
    \left[
      \tfrac{ u \beta^2 
      }{
        2 \gamma
      }
      ( 1 - e^{ - \gamma t } )
      + 1 
    \right]^{ - \frac{ 2 \delta }{ \beta^2 } }
    \E\!\left[
      \exp\!\left(
        - 
        \tfrac{ 
          e^{ - \gamma t } 
        }{
          \frac{ u \beta^2 }{ 2 \gamma }
          ( 1 - e^{ - \gamma t } )
          + 1 
        }
	X_0 u
      \right)
    \right]
  du
\Bigg)
\\ & \leq 
  \tfrac{ 1 }{ \Gamma(p) } 
  \bigg(
    \int_0^1 
      u^{ (p - 1) }
    \, du
    + 
    \tfrac{ C }{ \eps } 
    e^{ \gamma t ( p + \eps ) }
  \bigg)
= 
  \tfrac{ 1 }{ \Gamma(p) } 
  \Big(
    \tfrac{ 1 }{ p }
    + 
    \tfrac{ C }{ \eps } 
    e^{ \gamma t ( p + \eps ) }
  \Big) 
  < \infty
  .
\end{split}     \end{equation}
This finishes the proof of Lemma~\ref{l:inverse.moments.CIR}.
\end{proof}

\begin{corollary}   \label{c:theorem_cir}
Let $ T,
\beta \in (0,\infty) $, $ \gamma \in \R $, $ \delta \in ( \nicefrac{ \beta^2 }{ 4 }, \infty ) $,
let
$
  ( 
    \Omega, \mathcal{F}, \P, ( \mathcal{F}_t )_{ t \in [0,T] } 
  )
$
be a stochastic basis, let
$
  W \colon [0,T] \times \Omega \to \R
$
be a standard $ ( \mathcal{F}_t )_{ t \in [0,T] } $-Brownian motion,
  let $X\colon[0,T]\times \Omega\to[0,\infty) $
  be an adapted stochastic process with continuous sample paths satisfying  
  $
    \inf_{
      r \in (0, \nicefrac{ 2 \delta }{ \beta^2 } - q ) 
    } 
    \sup_{
      u \in [1,\infty)
    } 
    u^{ (q + r) } 
    \,
    \E\big[
      \exp( - u X_0 )
    \big] 
    +
    \E\big[ 
      (X_0)^s
    \big] 
    < \infty
  $ 
  for all 
  $ 
    q \in [ \nicefrac{ 1 }{ 2 } , \nicefrac{ 2 \delta }{ \beta^2 } )
  $,
  $ s \in \R $
  and
  \begin{equation} \label{l:theorem_cir:ass1bb}
    X_t=X_0+\int_0^t \delta - \gamma X_s\,ds+ \int_0^t \beta \sqrt{X_s} \, dW_s
  \end{equation}
  $\P$-a.s.\
  for all $t \in [0,T]$.
%
%
%
%
Then there exists a unique family 
$
  Y^h \colon
  \big[
    0, \lfloor \nicefrac{ T }{ h } \rfloor h
  \big]
  \times \Omega \to [0,\infty)
$, 
$ 
  h \in (0,T] \cap (0, \nicefrac{ 2 }{ ( - \gamma )^+ } )
$, 
of mappings satisfying
for all
$
  h \in (0,T] \cap ( 0, \nicefrac{ 2 }{ ( - \gamma )^+ } )
$,
$ n \in \N_0 \cap [0,\nicefrac{T}{h} - 1] $,
$ t \in ( n h , (n+1) h ] $
that
$
  Y_0^h = X_0
$ 
and
\begin{equation} 
\label{eq:LBEcor}
  Y_{t}^h
  =
  \left[ 
    n + 1 - \tfrac{ t }{ h }
  \right]
  Y_{ n h }^h
  +
  \left[ 
    \tfrac{ t }{ h } - n 
  \right]
  \left[
    \tfrac{
      ( Y_{ n h }^h )^{ \nicefrac{ 1 }{ 2 } }
      +
      \frac{
        \beta
      }{ 2 }
      ( W_{ n h + h } - W_{ n h } )
      +
      \sqrt{
        \left[
          ( Y_{ n h }^h )^{ \nicefrac{ 1 }{ 2 } }
          +
          \frac{ \beta }{ 2 }
          ( W_{ n h + h } - W_{ n h } )
        \right]^2
          +
          ( 2 + \gamma h ) 
          ( \delta - \frac{ \beta^2 }{ 4 } )
          h
      }
    }{ 
      ( 2 + \gamma h ) 
    }
  \right]^2
\end{equation}
and it holds for all 
$ p, \eps \in \left( 0, \infty \right) $
that
\begin{align} \label{l:theorem_cir:statement}
\sup_{
  h \in (0,T] \cap (0, \nicefrac{1}{(-2\gamma)^+}]
} 
\left[ 
  h^{  
    \left[
      \eps  
      - 
      \frac{ 
        (\nicefrac{2 \delta}{\beta^2}) \wedge 1  - \nicefrac{1}{2} 
      }{
        ( p \vee 1) 
      } 
    \right]
  }
  \bigg\| 
    \!
    \sup_{
      t\in[0,\lfloor \nicefrac{T}{h}\rfloor h]
    }
    \big| 
      X_t -  Y^h_t 
    \big|
  \bigg\|_{ L^{p}( \Omega; \R ) }
\right]
  < \infty
  .
\end{align}

\end{corollary}
\begin{proof} [Proof of Corollary~\ref{c:theorem_cir}]
We fix $ p, \eps \in (0,\infty) $ 
throughout this proof and we
%
%
%
%
define a function $\phi \colon [0,\infty) \to [0,\infty)$ by
\begin{align}  
\label{cir:bsp_numerik:1}
  \phi(x) := 
  \int_0^x \tfrac{ 1 }{ \beta \sqrt{z} } \, dz 
  = 
  \tfrac{ 2 }{ \beta } \sqrt{x}
\end{align}
for all $x \in [0,\infty)$.
Note that $ \phi $ is bijective
and observe 
that for all $ x \in [0,\infty) $ it holds that 
$
  \phi^{-1}( x ) = \tfrac{ \beta^2 }{ 4 } x^2
$.
%
%
%
%
%
We define 
$ 
  \alpha := \tfrac{ 2 \delta }{ \beta^2 } - \tfrac{ 1 }{ 2 } \in (0, \infty)
$ 
and we
a function 
$
  g \colon [0,\infty) \to \R
$ 
by 
$
  g(0) := 0 
$ 
and
$ g(x):= 
\tfrac{\alpha}{x} - \tfrac{\gamma x}{2}
$
for all $ x \in (0,\infty) $.
The fact that
$
  g|_{ (0,\infty) } \in C^1( (0,\infty), \R )
$ 
and that
for all $ x \in (0,\infty) $ it holds that
$ 
  g'(x)
=
  - \tfrac{ \alpha }{ x^2 } - \tfrac{ \gamma }{ 2 }
$
proves that 
$
  L :=
  \big[ 
    \sup_{x \in (0,\infty)} g'(x) 
  \big]^+ 
  =  
  \left( - \nicefrac{\gamma}{2} \right)^+
$
and that
$
  \limsup_{ x \to \infty } 
  \nicefrac{
    \left|  g'(x) \right|
  }{ x } 
  = 0
$.
Lemma~\ref{l:inverse.moments.CIR}
implies that for all 
$
  q \in [ \tfrac{ 1 }{ 2 } , \tfrac{ 2 \delta }{ \beta^2 } )
$ 
it holds that
$
  \sup_{ t \in [0,T] }
  \E\!\left[
    \left( X_t \right)^{ - q }
  \right]
  < \infty
$.
This shows that 
for all 
$
  q \in [ 1, 1 + 2 \alpha ) 
$ 
it holds that
$
  \sup_{ t \in [0,T] }
  \E\!\left[
    \phi\left( X_t \right)^{ - q }
  \right] < \infty
$.
Now we apply Theorem~\ref{thm:num_theorem}
to obtain that
there exists a unique family
$
  Y^h \colon 
  \{ 0, h, 2 h, \ldots , \lfloor \nicefrac{ T }{ h } \rfloor h \}
  \times \Omega \to [0,\infty) 
$, 
$
  h \in (0,T] \cap (0, \nicefrac{ 1 }{ L } ) 
$,
of mappings
satisfying $ Y^h_0 = X_0 $ 
and
\begin{equation} 
\label{cor:num_theorem:umformung2}
  \phi\!\left(
    Y^h_{ n h }
  \right)
  = 
  \phi\!\left(
    Y^h_{ ( n - 1 ) h }
  \right)
  + 
  g\!\left(
    \phi\big( Y^h_{ n h } \big) 
  \right) h 
  + W_{ n h } - W_{ (n - 1 ) h }  
\end{equation}
for all 
$ n \in \N \cap [ 0 , \nicefrac{ T }{ h } ] $, $ h \in (0,T] \cap (0, \nicefrac{ 1 }{ L } ) $
and it holds 
for all $ \varepsilon, p \in (0,\infty) $ that
\begin{align} \label{eq:cor_cir:statement}
\sup_{
  h \in (0,T] \cap (0, \nicefrac{1}{(-2\gamma)^+}]
} 
\left[ 
  h^{  
    \left[
      \eps  
      - 
      \frac{ 
        (\nicefrac{2 \delta}{\beta^2}) \wedge 1  - \nicefrac{1}{2} 
      }{
        ( p \vee 1) 
      } 
    \right]
  }
  \left\| 
    \sup_{
      n \in \N_0 \cap [0,T/h] 
    }
    \left| 
      X_{nh} -  Y^{h}_{nh} 
    \right|
  \right\|_{ L^{p}( \Omega; \R ) }
\right]
  < \infty
  .
\end{align}
The solution of the implicit equation~\eqref{cor:num_theorem:umformung2}
is well-known (see (4) in Alfonsi~\cite{Alfonsi2005})
and the linear interpolations of the resulting discrete-time processes
satisfy equation~\eqref{eq:LBEcor}.
Finally, \eqref{eq:cor_cir:statement}
together with the fact that for all 
$ p, \eps \in (0,\infty) $
it holds that
\begin{equation}
  \left\|
    \sup_{ 
      s, t \in [0,T] , s \neq t
    }
    \left[
    \frac{ 
      \left|
        X_t - X_s
      \right|
    }{
      \left| t - s \right|^{ ( \eps - \frac{ 1 }{ 2 } ) }
    }
    \right]
  \right\|_{
    L^p( \Omega; \R ) 
  }
  < \infty
\end{equation}
implies~\eqref{l:theorem_cir:statement}.
This finishes the proof of Corollary~\ref{c:theorem_cir}.
\end{proof}

\subsubsection*{Acknowledgements}

Special thanks are due to Steffen Dereich for encouraging us to investigate this problem.
Moreover, we are very grateful to Florian Mueller-Reiter from the swissQuant Group AG for a
number of quite helpful comments regarding the use of the Heston model in the financial practice.

This project has been partially supported by the research project
``Numerical approximation of stochastic differential equations with non-globally Lipschitz
continuous coefficients'' (HU1889/2-1)
funded by the German Research Foundation.

\bibliographystyle{acm}
\bibliography{../../Bib/bibfile}

\def\cprime{$'$} \def\cprime{$'$}
\begin{thebibliography}{10}

\bibitem{Alfonsi2005}
{\sc Alfonsi, A.}
\newblock On the discretization schemes for the {CIR} (and {B}essel squared)
  processes.
\newblock {\em Monte Carlo Methods Appl. 11}, 4 (2005), 355--384.

\bibitem{Alfonsi2013}
{\sc Alfonsi, A.}
\newblock Strong order one convergence of a drift implicit {E}uler scheme:
  {A}pplication to the {CIR} process.
\newblock {\em Statistics \& Probability Letters 83}, 2 (2013), 602 -- 607.

\bibitem{BakshiCaoChen1997}
{\sc Bakshi, G., Cao, C., and Chen, Z.}
\newblock Empirical performance of alternative option pricing models.
\newblock {\em The Journal of Finance 52}, 5 (1997), pp. 2003--2049.

\bibitem{ForoushTahmasebi2012}
{\sc Bastani, A.~F., and Tahmasebi, M.}
\newblock Strong convergence of split-step backward {E}uler method for
  stochastic differential equations with non-smooth drift.
\newblock {\em Journal of Computational and Applied Mathematics 236}, 7 (2012),
  1903 -- 1918.

\bibitem{BerkaouiBossyDiop2008}
{\sc Berkaoui, A., Bossy, M., and Diop, A.}
\newblock Euler scheme for {SDE}s with non-{L}ipschitz diffusion coefficient:
  strong convergence.
\newblock {\em ESAIM Probab. Stat. 12\/} (2008), 1--11 (electronic).

\bibitem{BossyDiop2007}
{\sc Bossy, M., and Diop, A.}
\newblock An efficient discretization scheme for one dimensional {SDEs} with a
  diffusion coefficient function of the form $|x|^{\alpha}$,
  $\alpha\in[1/2,1)$.
\newblock {\em Working paper, INRIA\/} (2007).

\bibitem{cir85}
{\sc Cox, J.~C., Ingersoll, J. J.~E., and Ross, S.~A.}
\newblock A theory of the term structure of interest rates.
\newblock {\em Econometrica 53}, 2 (1985), 385--407.

\bibitem{CoxHutzenthalerJentzen2013}
{\sc Cox, S.~G., Hutzenthaler, M., and Jentzen, A.}
\newblock Local {L}ipschitz continuity in the initial value and strong
  completeness for nonlinear stochastic differential equations.
\newblock {\em arXiv:1309.5595v1\/} (2013), 54 pages.

\bibitem{dz92}
{\sc {Da Prato}, G., and Zabczyk, J.}
\newblock {\em Stochastic equations in infinite dimensions}, vol.~44 of {\em
  Encyclopedia of Mathematics and its Applications}.
\newblock Cambridge University Press, Cambridge, 1992.

\bibitem{DeelstraDelbaen1998}
{\sc Deelstra, G., and Delbaen, F.}
\newblock Convergence of discretized stochastic (interest rate) processes with
  stochastic drift term.
\newblock {\em Appl. Stochastic Models Data Anal. 14}, 1 (1998), 77--84.

\bibitem{DereichNeuenkirchSzpruch2012}
{\sc Dereich, S., Neuenkirch, A., and Szpruch, L.}
\newblock An {E}uler-type method for the strong approximation of the
  {C}ox-{I}ngersoll-{R}oss process.
\newblock {\em Proc. R. Soc. Lond. Ser. A Math. Phys. Eng. Sci. 468}, 2140
  (2012), 1105--1115.

\bibitem{DuffiePanSingleton2000}
{\sc Duffie, D., Pan, J., and Singleton, K.}
\newblock Transform analysis and asset pricing for affine jump-diffusions.
\newblock {\em Econometrica 68}, 6 (2000), 1343--1376.

\bibitem{EngelmannKosterOeltz2011}
{\sc Engelmann, B., Koster, F., and Oeltz, D.}
\newblock Calibration of the {H}eston stochastic local volatility model: A
  finite volume scheme.
\newblock {\em Available at SSRN: http://ssrn.com/abstract=1823769 or
  http://dx.doi.org/10.2139/ssrn.1823769\/} (2011).

\bibitem{Etore2006}
{\sc Etor{\'e}, P.}
\newblock On random walk simulation of one-dimensional diffusion processes with
  discontinuous coefficients.
\newblock {\em Electron. J. Probab. 11\/} (2006), no. 9, 249--275.

\bibitem{g08a}
{\sc Giles, M.~B.}
\newblock Improved multilevel {M}onte {C}arlo convergence using the {M}ilstein
  scheme.
\newblock In {\em Monte {C}arlo and quasi-{M}onte {C}arlo methods 2006}.
  Springer, Berlin, 2008, pp.~343--358.

\bibitem{GoeingJaeschkeYor2003}
{\sc G{\"o}ing-Jaeschke, A., Yor, M., et~al.}
\newblock A survey and some generalizations of bessel processes.
\newblock {\em Bernoulli 9}, 2 (2003), 313--349.

\bibitem{GyoengyRasonyi2011}
{\sc Gy{\"o}ngy, I., and R{\'a}sonyi, M.}
\newblock A note on {E}uler approximations for {SDE}s with {H}\"older
  continuous diffusion coefficients.
\newblock {\em Stochastic Process. Appl. 121}, 10 (2011), 2189--2200.

\bibitem{Halidias2012}
{\sc Halidias, N.}
\newblock Semi-discrete approximations for stochastic differential equations
  and applications.
\newblock {\em Int. J. Comput. Math. 89}, 6 (2012), 780--794.

\bibitem{h98}
{\sc Heinrich, S.}
\newblock Monte {C}arlo complexity of global solution of integral equations.
\newblock {\em J. Complexity 14}, 2 (1998), 151--175.

\bibitem{h01}
{\sc Heinrich, S.}
\newblock Multilevel {M}onte {C}arlo methods.
\newblock In {\em Large-Scale Scientific Computing}, vol.~2179 of {\em Lect.
  Notes Comput. Sci.} Springer, Berlin, 2001, pp.~58--67.

\bibitem{h93}
{\sc Heston, S.~L.}
\newblock A closed-form solution for options with stochastic volatility with
  applications to bond and currency options.
\newblock {\em Commun. Comput. Phys. 6}, 2 (1993), 327--343.

\bibitem{HighamMao2005}
{\sc Higham, D.~J., and Mao, X.}
\newblock Convergence of {M}onte {C}arlo simulations involving the
  mean-reverting square root process.
\newblock {\em Journal of Computational Finance 8}, 3 (2005), 35--61.

\bibitem{HurdKuznetsov2008}
{\sc Hurd, T.~R., and Kuznetsov, A.}
\newblock Explicit formulas for {L}aplace transforms of stochastic integrals.
\newblock {\em Markov Process. Related Fields 14}, 2 (2008), 277--290.

\bibitem{HutzenthalerJentzen2014Memoires}
{\sc Hutzenthaler, M., and Jentzen, A.}
\newblock Numerical approximations of stochastic differential equations with
  non-globally {L}ipschitz continuous coefficients.
\newblock {\em To appear in Mem. Amer. Math. Soc.\/} (2014), arXiv:1203.5809.

\bibitem{KahlJaeckel2005}
{\sc Kahl, C., and J{\"a}ckel, P.}
\newblock Not-so-complex logarithms in the {H}eston model.
\newblock {\em Wilmott magazine 19}, 9 (2005), 94--103.

\bibitem{k05}
{\sc Kebaier, A.}
\newblock Statistical {R}omberg extrapolation: a new variance reduction method
  and applications to option pricing.
\newblock {\em Ann. Appl. Probab. 15}, 4 (2005), 2681--2705.

\bibitem{LordKoekkoekDijk2010}
{\sc Lord, R., Koekkoek, R., and Dijk, D.~V.}
\newblock A comparison of biased simulation schemes for stochastic volatility
  models.
\newblock {\em Quantitative Finance 10}, 2 (2010), 177--194.

\bibitem{MilsteinSchoenmakers2013}
{\sc Milstein, G.~N., and Schoenmakers, J.}
\newblock Uniform approximation of the cox-ingersoll-ross process.
\newblock {\em arXiv preprint arXiv:1312.0876\/} (2013).

\bibitem{NeuenkirchSzpruch2014}
{\sc Neuenkirch, A., and Szpruch, L.}
\newblock First order strong approximations of scalar {SDEs} defined in a
  domain.
\newblock {\em Numerische Mathematik\/} (2014), 1--34.

\bibitem{PrevotRoeckner2007}
{\sc Pr{\'e}v{\^o}t, C., and R{\"o}ckner, M.}
\newblock {\em A concise course on stochastic partial differential equations},
  vol.~1905 of {\em Lecture Notes in Mathematics}.
\newblock Springer, Berlin, 2007.

\bibitem{Przybylowicz2013}
{\sc Przybylowicz, P.}
\newblock Optimality of {E}uler-type algorithms for approximation of stochastic
  differential equations with discontinuous coefficients.
\newblock {\em International Journal of Computer Mathematics\/} (2013), 1--19.

\bibitem{RevuzYor1994}
{\sc Revuz, D., and Yor, M.}
\newblock {\em Continuous martingales and {B}rownian motion}, second~ed.,
  vol.~293 of {\em Grundlehren der Mathematischen Wissenschaften [Fundamental
  Principles of Mathematical Sciences]}.
\newblock Springer-Verlag, Berlin, 1994.

\bibitem{RogersWilliams2000b}
{\sc Rogers, L. C.~G., and Williams, D.}
\newblock {\em Diffusions, {M}arkov processes and martingales. {V}ol. 2}.
\newblock Cambridge Mathematical Library. Cambridge University Press,
  Cambridge, 2000.
\newblock It{\^o} calculus, Reprint of the second (1994) edition.

\bibitem{Yan2002}
{\sc Yan, L.}
\newblock The {E}uler scheme with irregular coefficients.
\newblock {\em The Annals of Probability 30}, 3 (2002), 1172--1194.

\end{thebibliography}
%
\end{document}